\documentclass[11pt]{article}

\usepackage[english]{babel} % Multilingual support for Plain TEX or LATEX
\usepackage[T1]{fontenc} % Standard package for selecting font encodings
\usepackage{lmodern}

\usepackage{amsmath,amsfonts,amssymb,amsthm} % math
\usepackage{cases} % Num­bered cases en­vi­ron­ment
 \usepackage{listings}
\usepackage{enumitem}

\usepackage{graphicx} % Enhanced support for graphics
\usepackage{pdfpages} % Include PDF documents in LATEX
\usepackage{subfig}
\usepackage{overpic}

\usepackage{booktabs} % Publication quality tables in LATEX
\usepackage{tikz}

\usepackage{authblk}

\usepackage[nottoc]{tocbibind}
\usepackage[titletoc,title]{appendix}
\numberwithin{equation}{section}

\usepackage{xcolor} %Driver-independent color extensions for LATEX and pdfLATEX
\usepackage{colortbl}
\definecolor{blun}{cmyk}{0.8, 0.5, 0, 0.7}

\usepackage[bookmarks,colorlinks,pdfauthor={Cinzia Soresina},linkcolor=blun,citecolor=blun]{hyperref}

\providecommand{\keywords}[1]{\textbf{\textit{Keywords:}} #1}

\usepackage[author={Cinzia}]{pdfcomment}
\definecolor{inchworm}{rgb}{0.7, 0.93, 0.36}

\newcommand{\Lap}{\Delta}
\newcommand{\tr}{\textnormal{tr}}
\newcommand{\N}{\mathbb N}
\newcommand{\cross}{\text{cross}}
\newcommand{\bif}{\text{bif}}

\theoremstyle{plain}
\newtheorem{theorem}{Theorem}[section]
\newtheorem{proposition}[theorem]{Proposition}

\newtheorem{remark}[theorem]{Remark}

\begin{document}
%%%%%%%%%%%%%%%%%%%%%%%%%%%%%%%%
%\begin{frontmatter}

\title{\Large{On the influence of cross-diffusion in pattern formation}}
% Authors
%\author{\normalsize{M. Breden, C. Kuehn, C. Soresina}\\ \footnotesize{Zentrum Mathematik - Technische Universit\"at M\"unchen}\\[-0.1cm] \footnotesize{Boltzmannstr. 3, 85748 Garching bei M\"unchen, Germany}}
\author[1,2]{M. Breden}
\author[2]{C. Kuehn}
\author[2]{C. Soresina}
\affil[1]{CMAP, \'Ecole Polytechnique, route de Saclay, 91120 Palaiseau, France}
\affil[2]{Zentrum Mathematik, Technische Universit\"at M\"unchen, Boltzmannstr. 3, 85748 Garching bei M\"unchen, Germany}
\maketitle

% \linenumbers
% Abstract
\begin{abstract}
In this paper we consider the Shigesada-Kawasaki-Teramoto (SKT) model to account for stable inhomogeneous steady states exhibiting \emph{spatial segregation}, which describe a situation of coexistence of two competing species. We provide a deeper understanding on the conditions required on both the cross-diffusion and the reaction coefficients for non-homogeneous steady states to exist, by combining a detailed linearized analysis with advanced numerical bifurcation methods via the continuation software \texttt{pde2path}. We report some numerical experiments suggesting that, when cross-diffusion is taken into account, there exist positive and stable non-homogeneous steady states outside of the range of parameters for which the coexistence homogeneous steady state is positive. Furthermore, we also analyze the case in which self-diffusion terms are considered.
\end{abstract}

\keywords{bifurcations; cross-diffusion; Turing instability; SKT model; \texttt{pde2path}} 

% Mathematics Subject Classification
%\subjclass[2010]{Primary: 35Q92. Secondary: 92D25, 35K59, 65P30}

% 35Q92 PDEs in connection with biology and other natural sciences <--
% 37N25 Dynamical systems in biology

% 92D25 Population dynamics (general)  <--

% 35K59 Quasilinear parabolic equations  <--
% 35J62 Quasilinear elliptic equations

% 37M20 Computational methods for bifurcation problems
% 65P30 Bifurcation problems  <--

%%%%%%%%%%%%%%%%%%%%%%%%%%%%%%%%%%%%%%%%%%%%%
\section{Introduction}\label{sec_int}

Competition is a fundamental aspect in the interaction of populations. Intraspecific competition occurs among individuals of the same species, when the resources (for survival or reproduction) are limited. Interspecific competition happens among individuals of different species and it mainly occurs when the population exploit the same (limited) resources. The outcomes of these types of interaction among ecologically similar species can be the competitive exclusion or the coexistence of the species, through niche differentiation or spatial segregation. In particular, spatial segregation describes a situation where two competing species coexist, but they mainly concentrate in different regions of the habitat: territorial segregation leads to an exclusive exploitation of the resources and it can minimize the encounters, and consequently also the conflicts, between individuals \cite{wilson1975}.

Both intraspecific and interspecific competition have been observed in real populations, often related to territoriality and aggressive behavior. For instance, the exclusion principle explains the dramatic decline of native Eurasian red squirrels on the British Isles and parts of northern Italy caused by the introduction of eastern grey squirrels \cite{wauters2005review}. On the contrary, spatial segregation has been observed in birds, mammals, amphibians, fishes and insects. More concrete examples are wolves and coyotes \cite{benson2013inter}, jaguars and pumas \cite{palomares2016fine}, different species of reed warbler \cite{hoi1991territorial}, lesser kestrels \cite{cecere2018spatial}, gibbons \cite{suwanvecho2012interspecific}, and salamanders \cite{hoffman2003habitat}.

From a modelling point of view, several mathematical models have been proposed to explain these natural phenomena and to predict the outcomes, both for a biological control purpose and for general ecological studies. A simple mathematical model can be formulated in terms of ordinary differential equations of Lotka--Volterra--Gause type. This model is able to predict the mutual exclusion or the coexistence of the species depending on the parameter values, but spatial movements of individuals are here neglected. When spatial variations of the populations are considered, reaction-diffusion equations can be formulated. In this context, spatial segregation is explained through non-homogeneous solutions of the reaction-diffusion model, and then linked to pattern formation. However, typical Lotka--Volterra--Gause reaction-diffusion models with only constant diffusion coefficients fail to produce such patterns, at least on convex domains~\cite{kishimoto1985spatial}. To account for stable inhomogeneous steady states exhibiting spatial segregation, the so called SKT model, which includes non-linear cross-diffusion terms, was proposed in \cite{shigesada1979spatial}. The system is given by
\begin{equation}
\begin{cases}
\partial_t u=\Lap((d_1+d_{11} u +d_{12} v)u)+(r_1-a_1 u-b_1 v)u,& \textnormal{on } \mathbb{R}_+\times \Omega,\\[0.1cm]
\partial_t v=\Lap((d_2+d_{22} v +d_{21} u)v)+(r_2-b_2 u-a_2 v)v,& \textnormal{on }  \mathbb{R}_+\times \Omega,\\[0.1cm]
\dfrac{\partial u}{\partial n}=\dfrac{\partial v}{\partial n}=0,&\textnormal{on } \mathbb{R}_+\times \partial\Omega,\\[0.1cm]
u(0,x)=u_{in}(x),\; v(0,x)=v_{in}(x),& \textnormal{on } \Omega,
\end{cases}\label{cross}
\end{equation}
where the quantities $u(t,x),\; v(t,x)\geq 0$ represent the population densities of two species at time $t$ and position $x$, confined and competing for resources on a bounded and connected domain $\Omega~\subset~\mathbb{R}^N$. The coefficients $d_i,\,r_i,\,a_i,\,b_i \, (i=1,2)$ describe the diffusion, the intrinsic growth, the intra-specific competition and the inter-specific competition rates, while $d_{11},\,d_{22}$ and $d_{12},\,d_{21}$ stand for competition pressure, they are called self- and cross-diffusion coefficients. To prevent any confusion, we will refer to $d_1$ and $d_2$ as the \emph{standard} diffusion coefficients. Throughout this paper we consider the cross-diffusion system \eqref{cross} and assume that the standard diffusion coefficients are positive, and that all the other coefficients are non-negative. 

The homogeneous steady states are the total extinction $(0,0)$, two non-coexistence states $(\bar u,0)~=~(r_1/a_1,0)$ and $(0,\bar v)=(0,r_2/a_2)$, and one coexistence state 
$$(u_*,v_*)=\left( \dfrac{r_1a_2-r_2b_1}{a_1a_2-b_1b_2},\dfrac{r_2a_1-r_1b_2}{a_1a_2-b_1b_2}\right).$$
While the non-coexistence equilibria exist for all the parameter values, the coexistence steady state is admissible (i.e. $u_*>0$ and $v_*>0$) only in two cases:
\begin{itemize}
\item \textit{weak competition} or \textit{strong intra-specific competition}, when 
\begin{equation}\label{weak_comp}
\dfrac{b_1}{a_2}<\dfrac{r_1}{r_2}<\dfrac{a_1}{b_2}.
\end{equation}
In this case, for the homogeneous system (when all diffusion coefficients are taken equal to zero), the coexistence steady state is stable, while the non-coexistence ones are unstable. With only standard diffusion, in a convex domain and with zero-flux boundary conditions, any non-negative solution generically converges to the homogeneous one, and this implies that the two species coexist but their densities are homogeneous in the whole domain \cite{kishimoto1985spatial}. 
\item \textit{strong competition} or \textit{strong inter-specific competition}, when \begin{equation}\label{strong_comp}
\dfrac{a_1}{b_2}<\dfrac{r_1}{r_2}<\dfrac{b_1}{a_2}.
\end{equation}
In this case, for the homogeneous system the coexistence steady state is unstable, while the non-coexistence ones are stable. Adding only standard diffusion, in a convex domain and with zero-flux boundary conditions, it has been shown that if positive and non-constant steady states exist, they must be unstable~\cite{kishimoto1985spatial}, and numerical simulations suggest that any non-negative solution generically converges to either $(\bar u, 0)$ or $(0,\bar v)$, that is, the competitive exclusion principle occurs between the two species.
\end{itemize}

Starting from the seminal paper~\cite{shigesada1979spatial}, the question of the existence of non-homogeneous steady states for~\eqref{cross}, when cross-diffusion terms are taken into account, has been extensively investigated, both numerically and theoretically.
We provide below a brief overview of the main results that are the most relevant for the present work. For a broader discussion we refer the reader the review papers~\cite{IidNimYam18,Jun10} and the references therein.

In the weak competition regime, if one of the cross-diffusion coefficients is sufficiently large compared to all other parameters, then the homogeneous co-existence steady state looses its stability and non-homogeneous steady states appear~\cite{LouNi96}. Besides, the shape and the amplitude of these pattern can be predicted \cite{gambino2012turing}. Still in the weak competition regime, the question of whether non-homogeneous steady states can exist if both cross-diffusion coefficients are large and \emph{qualitatively similar} is also raised in~\cite{LouNi96}. A (partial) negative answer is given in~\cite{CheJun06}, in the particular case where $b_1=b_2=0$, for which entropy methods can be used to prove that all steady states must be homogeneous. Yet, it is known that for other cross-diffusion systems the transition in parameter space between the validity of entropy methods and bifurcation techniques is intricate~\cite{JuengelKuehnTrussardi}. Hence, we are going to quantify the parameter dependence of the bifurcations in the SKT system carefully.

The strong competition case is more complicated, already when only standard diffusion is considered (see again~\cite{LouNi96} and the references therein). Nevertheless, non-homogeneous steady states can also be obtained in that case, and previous results~\cite{LouNi96} seem to indicate that cross-diffusion is even more helpful in obtaining patterns in the strong competition regime, as the conditions on the cross-diffusion coefficients are less restrictive. In particular, non-homogeneous steady states are proven to exist when both cross-diffusion coefficients are large.

In both regimes (weak and strong competition), the qualitative properties of these non-homogeneous steady states have also been investigated. In~\cite{Ni98}, the existence of spike-layer solutions is shown, in the asymptotic limit where one of the cross-diffusion coefficients goes to infinity. In the so-called triangular case, where $d_{21}=0$, intricate bifurcation diagrams of steady states were obtained numerically in \cite{iida2006diffusion,izuhara2008reaction}, and then validated in~\cite{breden2013global,breden2018existence}. \medskip

In this work, we provide a deeper understanding on the conditions required on both the cross-diffusion and the reaction coefficients for non-homogeneous steady states to exist, by combining a detailed linearized analysis with numerical bifurcation methods. More precisely, we analytically predict the existence of \emph{local} branches of non-homogeneous steady states bifurcating from the homogeneous one, and then continue these branches numerically to get a more \emph{global} picture, including secondary bifurcations which can sometimes lead to stability changes. To this end, we used the continuation software for PDEs \texttt{pde2path} \cite{dohnal2014pde2path, uecker2018hopf, uecker2014pde2path}, based on a FEM discretization of the stationary problem.
Here is a summary of the main contributions of this paper.
\begin{itemize}[leftmargin=0.3cm]
\item We recover two conditions on the reaction coefficients that were already noticed in previous studies~\cite{LouNi96} about the existence of non-homogeneous steady states, and explain their role from the linearized analysis, see Propositions \ref{prop:weak} and \ref{prop:strong} (and Figures~\ref{fig:signs_weak} and~\ref{fig:signs_strong}). We point out that a similar study was also done in~\cite{MimKaw80} in the triangular case. 
\item This condition allows us to clearly explain, why cross-diffusion can be ``more helpful'' to obtain non-homogeneous steady states in the strong competition case, and why, when both cross-diffusion coefficients are large, it is harder to get such non-homogeneous steady states in the weak competition regime.
\item Nevertheless, we show in Theorem~\ref{th:large_cd} that such non-homogeneous solutions do actually exist in the weak competition regime when both cross-diffusion coefficients are large and equal, thus answering positively the question raised in~\cite{LouNi96}.
\item We also study how the bifurcation structure is affected, when one of the cross-diffusion parameter is varied, and in particular what becomes of the non-homogeneous steady states obtained in the extensively studied triangular case, when $d_{21}$ is then turned on.
\item Finally, we go beyond the linearized analysis by numerically computing global bifurcation diagrams of steady states in the non-triangular case. Our results highlight, among other things, that the non-homogeneous solutions obtained in the triangular case for the weak competition regime can quickly collapse and disappear when the other cross-diffusion coefficient is turned on. 
\item Our numerical continuation calculations also suggest the existence of stable non-homogeneous steady states outside of the weak~\eqref{weak_comp} or strong~\eqref{strong_comp} competition regimes, which to the best of our knowledge, have never been observed before.
\end{itemize}

\begin{remark}
We make some preliminary simplifications. We are mainly interested in the existence of non-homogeneous steady states when one or both of the cross-diffusion become large compared to the diffusion coefficients. To match with previous studies, we will mostly take $d_1=d_2=d$ as a bifurcation parameter, and study the appearance of these solutions when $d$ decreases (it has been proven~\cite{LouNi96}, that if at least one of the standard diffusion coefficients is large enough compared to the cross-diffusion coefficients, there can not exist any non-homogeneous steady state). To simplify the presentation, we start in Section~\ref{sec:without_selfdiff} by considering the case when there is no self-diffusion (i.e. we take $d_{11}=d_{22}=0$). We then discuss in Section~\ref{sec:with_selfdiff} qualitative changes that are induced by including self-diffusion.
\end{remark}

\medskip

Before going further, let us briefly mention that in order to understand the long time dynamics of~\eqref{cross}, we not only need a good understanding of the steady states (which is the objective we are pursuing in this work), but also a theory of global existence and regularity of the time depending solutions. The first main result in this direction was obtained in~\cite{Ama88,Ama90}, where a general theory about the existence of local solutions for general quasilinear parabolic PDEs is developed. To then get global solutions, one must show that some Sobolev norms remain bounded, which has been the subject of many works, all imposing some restrictions on the coefficients of the system. In particular, the triangular case ($d_{21}=0$) has been thoroughly investigated, mainly because the reduced coupling then allows to get a maximum principle for the second equation. In this setting, the most general result was obtained in~\cite{HaoNguPha15}, where the existence of a unique smooth global solution is established assuming $d_{21}=0$ and $d_{11}>0$. Another breakthrough was made in~\cite{GalGarJun03}, where an entropy structure for~\eqref{cross} was discovered. Entropy-based methods were then generalized, for instance to handle a broader class of cross-diffusion terms (see e.g.~\cite{DevLepMouTre15}), and we refer the reader to the book~\cite{Jun16} for a more complete review on the subject.

\medskip

The paper is organized as follow. In Section \ref{sec:without_selfdiff} we perform a detailed linearized analysis and we detect bifurcations from the homogeneous coexistence steady state. We exploit this result to answer the question raised in~\cite{LouNi96} and to understand what happens when one cross-diffusion coefficient increases. In Section \ref{sec:global} we both illustrate and complement the obtained results by numerically computing bifurcation diagrams of steady states, while 
Section \ref{sec:r1} shows the existence of stable non-homogeneous steady states outside of the range of parameter for which the homogeneous solution is positive.  In Section \ref{sec:with_selfdiff} we also take into account non-zero self-diffusion coefficients and we briefly describe the main changes induced by self-diffusion. Finally, in Section \ref{sec:concl} some concluding remarks and biological comments can be found.

%%%%%%%%%%%%%%%%%%%%%%%%%%%%%%%%%%%%%%%%%%%%%%%%%%%%%%%%%%%
%%%%%%%%%%%%%%%%%%%%%%%%%%%%%%%%%%%%%%%%%%%%%%%%%%%%%%%%%%%

\section{Analysis without self-diffusion}
\label{sec:without_selfdiff}

Both numerically and theoretically, one of the main ways to obtain non-homogeneous steady states for~\eqref{cross} is to study bifurcations from the homogeneous coexistence steady state $(u_*,v_*)$. We mention that a connection between usual Turing instabilities and cross-diffusion induced instabilities was made in~\cite{iida2006diffusion}.

We start by studying ``mode by mode'' the linear stability of $(u_*,v_*)$~\cite{Henry,KuehnBook1}. To do so we will consider the eigenfunctions $\psi_k$ and associated eigenvalues $-\lambda_k$ of the Laplacian with zero Neumann boundary conditions:
\begin{equation*}
\begin{cases}
-\Lap\psi_k=\lambda_k\psi_k,& \textnormal{on }  \Omega,\\[0.1cm]
\dfrac{\partial \psi_k}{\partial n}=0,&\textnormal{on } \partial\Omega,\\[0.1cm]
\end{cases}
\end{equation*}
which satisfy $\lambda_0=0$, $\lambda_k>0$ for all $k\in\N_{\geq 1}$, and $\lambda_k\rightarrow+\infty$ as $k\to+\infty$. In the sequel we always assume that the eigenvalues are labeled in ascending order. Throughout this section, we assume $d_{11}=d_{22}=0$. 

\subsection{Existence of bifurcations}
\label{sec:bif}

The Jacobian matrix of the reaction part and the linearization of the diffusion part of~\eqref{cross}, evaluated at the equilibrium $(u_*,v_*)$, are
$$J_*= \begin{pmatrix}-a_1u_* & -b_1u_*\\-b_2v_*& -a_2v_*\end{pmatrix}, \quad
J_\Delta^*=\begin{pmatrix}d+d_{12}v_*& d_{12}u_*\\d_{21}v_* & d+d_{21}u_*\end{pmatrix}.$$
Then, the characteristic matrix associated to the $k$-th mode, $k\in\N$, is
$$M_k^*=J_*-J_\Delta^*\lambda_k=
\begin{pmatrix}
-a_1u_*-(d+d_{12}v_*)\lambda_k   &  -b_1u_*-d_{12}u_*\lambda_k\\
-b_2v_* -d_{21}v_*\lambda_k         & -a_2v_*-(d+d_{21}u_*)\lambda_k
\end{pmatrix}.$$
and its determinant can be written as a second order polynomial in $d$
\begin{equation}\label{detM*d}
P_k(d):=\det M_k^*=\lambda_k^2 d^2 +
(d_{12}v_*\lambda_k^2+d_{21}u_*\lambda_k^2-\tr J_*\lambda_k)d
-d_{12}\alpha \lambda_k- d_{21}\beta \lambda_k +\det J_*,
\end{equation}
where 
\begin{equation*}
\alpha:=(b_2u_*-a_2v_*)v_*, \qquad \beta:=(b_1v_*-a_1u_*)u_*, \qquad \det J_*=(a_1a_2-b_1b_2)u_*v_*.
\end{equation*}
Note that, since we are not going to varying the inter- and intra-specific competition rates $\;a_i,\; b_i,\; (i=1,2)$ (excluding the possibility to pass from the weak to the strong competition regime or vice versa), the sign of the determinant of $J_*$ is constant. In detail, it is positive in the weak competition case, and negative in the strong competition case. Moreover, the trace of $M_k^*$ is always negative (assuming~\eqref{weak_comp} or~\eqref{strong_comp}), therefore the $k$-th mode is stable if $P_k(d)>0$, unstable if $P_k(d)<0$, and a bifurcation occurs for $P_k(d)=0$. Obviously, no bifurcation can happen for $k=0$, since $P_0$ reduces to $\det J_*$ and does not depend on $d$. Let us introduce
\begin{equation}
\label{ABC}
A_k=\lambda_k^2,\quad B_k=d_{12}v_*\lambda_k^2+d_{21}u_*\lambda_k^2-\tr J_*\lambda_k, \quad C_k=-d_{12}\alpha \lambda_k- d_{21}\beta \lambda_k +\det J_*,
\end{equation}
so that
\begin{equation*}
P_k(d)=A_k d^2 + B_k d + C_k.
\end{equation*}
Obviously $A_k\geq 0$, and we have $\tr J_*<0$ (again assuming~\eqref{weak_comp} or~\eqref{strong_comp}) which implies $B_k>0$. Therefore, a bifurcation associated to the $k$-th mode ($k\geq 1$) can occur if and only if $C_k<0$. The signs of $\alpha$ and $\beta$ are thus crucial, as they change the monotonicity of $C_k$ with respect to $d_{12}$ and $d_{21}$ respectively. These signs depend on the parameter values $r_i,\,a_i,\, b_i,\, (i=1,2)$, and the different possibilities in the weak and strong competition regimes are presented in Figures~\ref{fig:signs_weak} and~\ref{fig:signs_strong}.

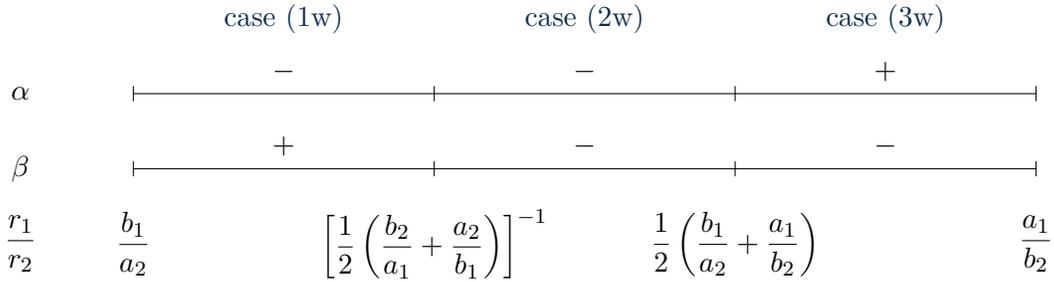
\begin{figure}[h!]
\centering
\begin{tikzpicture}
\draw (0,0) -- (12,0);
\draw (0,1) -- (12,1);
\draw (0,-0.1) -- (0,0.1);
\draw (4,-0.1) -- (4,0.1);
\draw (8,-0.1) -- (8,0.1);
\draw (12,-0.1) -- (12,0.1);
\draw (0,0.9) -- (0,1.1);
\draw (4,0.9) -- (4,1.1);
\draw (8,0.9) -- (8,1.1);
\draw (12,0.9) -- (12,1.1);
\draw (-1.5,1) node{$\alpha$};
\draw (-1.5,0) node{$\beta$};
\draw (-1.5,-1) node{$\dfrac{r_1}{r_2}$};
\draw (0,-1) node{$\dfrac{b_1}{a_2}$};
\draw (4,-1) node{$\left[\dfrac 1 2 \left(\dfrac{b_2}{a_1}+\dfrac{a_2}{b_1}\right)\right]^{-1}$};
\draw (8,-1) node{$\dfrac 1 2 \left(\dfrac{b_1}{a_2}+\dfrac{a_1}{b_2}\right)$};
\draw (12,-1) node{$\dfrac{a_1}{b_2}$};
\draw (2,0.3) node{$+$};
\draw (6,0.3) node{$-$};
\draw (10,0.3) node{$-$};
\draw (2,1.3) node{$-$};
\draw (6,1.3) node{$-$};
\draw (10,1.3) node{$+$};
\draw (2,2) node{\ref{w1}};
\draw (6,2) node{\ref{w2}};
\draw (10,2) node{\ref{w3}};
\end{tikzpicture}
\caption{Sign of the quantities $\alpha$ and $\beta$ in the weak competition regime~\eqref{weak_comp}, depending on the value of $r_1/r_2$.}\label{fig:signs_weak}
\end{figure}

\begin{figure}[h!]
\centering
\begin{tikzpicture}
\draw (0,0) -- (12,0);
\draw (0,1) -- (12,1);
\draw (0,-0.1) -- (0,0.1);
\draw (4,-0.1) -- (4,0.1);
\draw (8,-0.1) -- (8,0.1);
\draw (12,-0.1) -- (12,0.1);
\draw (0,0.9) -- (0,1.1);
\draw (4,0.9) -- (4,1.1);
\draw (8,0.9) -- (8,1.1);
\draw (12,0.9) -- (12,1.1);
\draw (-1.5,1) node{$\alpha$};
\draw (-1.5,0) node{$\beta$};
\draw (-1.5,-1) node{$\dfrac{r_1}{r_2}$};
\draw (0,-1) node{$\dfrac{a_1}{b_2}$};
\draw (4,-1) node{$\left[\dfrac 1 2 \left(\dfrac{b_2}{a_1}+\dfrac{a_2}{b_1}\right)\right]^{-1}$};
\draw (8,-1) node{$\dfrac 1 2 \left(\dfrac{b_1}{a_2}+\dfrac{a_1}{b_2}\right)$};
\draw (12,-1) node{$\dfrac{b_1}{a_2}$};
\draw (2,0.3) node{$-$};
\draw (6,0.3) node{$+$};
\draw (10,0.3) node{$+$};
\draw (2,1.3) node{$+$};
\draw (6,1.3) node{$+$};
\draw (10,1.3) node{$-$};
\draw (2,2) node{\ref{s1}};
\draw (6,2) node{\ref{s2}};
\draw (10,2) node{\ref{s3}};
\end{tikzpicture}
\caption{Sign of the quantities $\alpha$ and $\beta$ in the strong competition case~\eqref{strong_comp}, depending on the value of $r_1/r_2$.}\label{fig:signs_strong}
\end{figure}
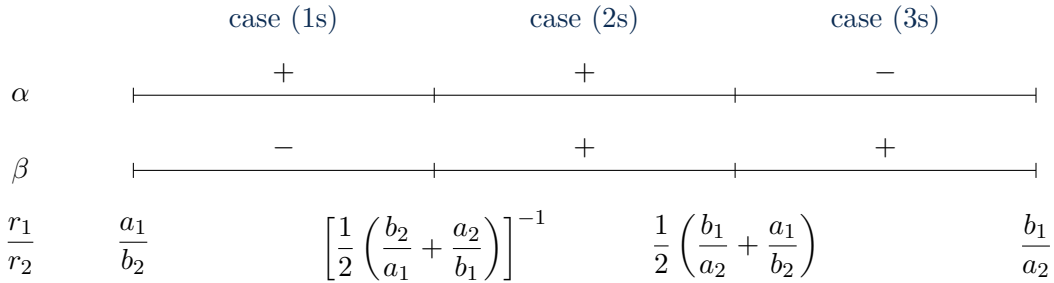

There we see why the weak competition regime~\eqref{weak_comp} and the strong competition regime~\eqref{strong_comp} can have different behaviors. Indeed, in the weak competition regime, $\alpha$ and $\beta$ cannot both be positive at the same time, which means that when both $d_{12}$ and $d_{21}$ are large, at most one of the terms $-d_{12}\alpha \lambda_k$ and $- d_{21}\beta \lambda_k$ helps to get a negative $C_k$, while the other one does not. Worse, there are cases when both $\alpha$ and $\beta$ are negative, which means that no amount of cross-diffusion will be able to produce bifurcations from $(u_*,v_*)$. On the other hand, in the strong competition regime there is always at least one of the terms between $-d_{12}\alpha \lambda_k$ and $- d_{21}\beta \lambda_k$ that helps to get a negative $C_k$. Besides, there are cases where both do help (when $\alpha$ and $\beta$ are positive), and in this regime it is easy to see that bifurcations from $(u_*,v_*)$ can occur when both $d_{12}$ and $d_{21}$ are large. These observations can be made precise as follows:

\begin{proposition}\label{prop:weak}
Consider $r_i,\;a_i,\;b_i \; (i=1,2)$ satisfying the weak competition hypothesis~\eqref{weak_comp}. 
\begin{enumerate}[wide, labelwidth=!, labelindent=0pt]
\renewcommand{\theenumi}{case (\arabic{enumi}w)}
\renewcommand{\labelenumi}{Case (\arabic{enumi}w):}
\item\label{w1} If 
\begin{equation*}
\dfrac{b_1}{a_2}<\dfrac{r_1}{r_2}<\left[\dfrac 1 2 \left(\dfrac{b_2}{a_1}+\dfrac{a_2}{b_1}\right)\right]^{-1},
\end{equation*}
then, for all $k\in\N_{\geq 1}$, there exists $d>0$ such that $P_k(d)=0$ if and only if
\begin{equation*}
0\leq d_{12} < \dfrac{1}{|\alpha|}\left(\beta d_{21}-\dfrac{\det J_*}{\lambda_k}\right).
\end{equation*}
\item\label{w2} If
\begin{equation*}
\left[\dfrac 1 2 \left(\dfrac{b_2}{a_1}+\dfrac{a_2}{b_1}\right)\right]^{-1} \leq \dfrac{r_1}{r_2} \leq \dfrac{1}{2} \left(\dfrac{b_1}{a_2}+\dfrac{a_1}{b_2}\right),
\end{equation*}
then, for all $k\in\N_{\geq 1}$ and all $d>0$, $P_k(d)>0$.
\item\label{w3} If
\begin{equation*}
\dfrac{1}{2} \left(\dfrac{b_1}{a_2}+\dfrac{a_1}{b_2}\right)<\dfrac{r_1}{r_2}<\dfrac{a_1}{b_2},
\end{equation*}
then, for all $k\in\N_{\geq 1}$, there exists $d>0$ such that $P_k(d)=0$ if and only if
\begin{equation*}
0\leq d_{21} < \dfrac{1}{|\beta|}\left(\alpha d_{12}-\dfrac{\det J_*}{\lambda_k}\right).
\end{equation*}
\end{enumerate}
\end{proposition}
\begin{proof}
We just rewrote the condition $C_k<0$ according to the signs of $\alpha$ and $\beta$ in the different cases, and used for~\ref{w2} that, in the weak competition regime~\eqref{weak_comp}, $\det J_*>0$.
\end{proof}
\begin{proposition}\label{prop:strong}
Consider $r_i,\;a_i,\;b_i \; (i=1,2)$ satisfying the strong competition hypothesis~\eqref{strong_comp}. 
\begin{enumerate}[wide, labelwidth=!, labelindent=0pt]
\renewcommand{\theenumi}{case (\arabic{enumi}s)}
\renewcommand{\labelenumi}{Case (\arabic{enumi}s):}
\item\label{s1} If 
\begin{equation*}
\dfrac{a_1}{b_2}<\dfrac{r_1}{r_2}<\left[\dfrac 1 2 \left(\dfrac{b_2}{a_1}+\dfrac{a_2}{b_1}\right)\right]^{-1},
\end{equation*}
then, for all $k\in\N_{\geq 1}$, there exists $d>0$ such that $P_k(d)=0$ if and only if
\begin{equation*}
0\leq d_{21} < \dfrac{1}{|\beta|}\left(\alpha d_{12}-\dfrac{\det J_*}{\lambda_k}\right).
\end{equation*}
\item\label{s2} If
\begin{equation*}
\left[\dfrac 1 2 \left(\dfrac{b_2}{a_1}+\dfrac{a_2}{b_1}\right)\right]^{-1} \leq \dfrac{r_1}{r_2} \leq \dfrac{1}{2} \left(\dfrac{b_1}{a_2}+\dfrac{a_1}{b_2}\right),
\end{equation*}
then, for all $k\in\N_{\geq 1}$, there exists $d>0$ such that $P_k(d)=0$.
\item\label{s3} If
\begin{equation*}
\dfrac{1}{2} \left(\dfrac{b_1}{a_2}+\dfrac{a_1}{b_2}\right)<\dfrac{r_1}{r_2}<\dfrac{b_1}{a_2},
\end{equation*}
then, for all $k\in\N_{\geq 1}$, there exists $d>0$ such that $P_k(d)=0$ if and only if
\begin{equation*}
0\leq d_{12} < \dfrac{1}{|\alpha|}\left(\beta d_{21}-\dfrac{\det J_*}{\lambda_k}\right).
\end{equation*}
\end{enumerate}
\end{proposition}
\begin{proof}
We just rewrote the condition $C_k<0$ according to the signs of $\alpha$ and $\beta$ in the different cases, and used for~\ref{s2} that, in the strong competition regime~\eqref{strong_comp}, $\det J_*<0$.
\end{proof}

We now go one step further, and use the above analysis of the linearized system to answer the question raised in~\cite{LouNi96} of whether non-homogeneous steady states can exist if both cross-diffusion coefficients are large and qualitatively similar, in the weak competition regime. We point out that, in the strong competition regime, the answer is obviously yes, at least in~\ref{s2}. In the weak competition regime, our next results show that, assuming $d_{12}=d_{21}=d_\cross$, there exists parameter values in~\ref{w1} and \ref{w3} for which, when $d_\cross$ is large enough, bifurcations of non-homogeneous steady state occur. 

\begin{theorem}\label{th:large_cd}
Consider $a_i,\;b_i,\; (i=1,2)$ such that 
\begin{equation*}
b_1b_2<a_1a_2 \qquad \text{and}\qquad 4a_1a_2<(b_1+b_2)^2,
\end{equation*}
and $r_i \; (i=1,2)$ such that 
\begin{equation*}
\frac{r_1}{r_2}=\frac{(b_1+b_2)(3a_1 a_2+b_1 b_2)}{2a_2(a_1 a_2+b_1 b_2+2b_2^2)},
\end{equation*}
 Then~\eqref{weak_comp} holds (i.e. the coexistence steady state $(u_*,v_*)$ is admissible) and $\alpha+\beta>0$, hence we must be in \ref{w1} or \ref{w3}. Besides, assuming $d_{12}=d_{21}=d_\cross$ we have that for all $k\in\N_{\geq 1}$, there exists $d>0$ such that $P_k(d)=0$ if and only if
\begin{equation*}
d_\cross > \frac{\det J_*}{(\alpha+\beta)\lambda_k}.
\end{equation*}
\end{theorem}
\begin{proof}
It is straightforward to check that, with the given definition of $r_1/r_2$ and the assumption $b_1b_2<a_1a_2$, \eqref{weak_comp} holds. The main task it to prove that $\alpha+\beta>0$. If that is the case, then $C_k$ will be negative for $d_{12}=d_{21}=d_{\cross}$ large enough. 

We start by rewriting $\alpha+\beta>0$ so that its sign depends only on the sign of a quadratic polynomial in $r:=r_1/r_2$. In order to shorten some of the formula, we introduce $a=a_1a_2$, $b=b_1b_2$. 
We have
\begin{align*}
\alpha+\beta &= 2(b_1+b_2)u_*v_*-(r_1u_*+r_2v_*) \\
%&=2(b_1+b_2)\frac{r_1a_2-r_2a_1}{a-b}\frac{r_2a_1-r_1a_2}{a-b}-\left(r_1\frac{r_1a_2-r_2a_1}{a-b}+r_2\frac{r_2a_1-r_1a_2}{a-b}\right) \\
&=2(b_1+b_2)\frac{r_1a_2-r_2b_1}{a-b}\frac{r_2a_1-r_1b_2}{a-b}-\left(r_1\frac{r_1a_2-r_2b_1}{a-b}+r_2\frac{r_2a_1-r_1b_2}{a-b}\right) \\
&=\left(\frac{r_2}{a-b}\right)^2 \left(2(b_1+b_2)(a_2 r-b_1)(a_1-b_2 r)-(a-b)\left(r(a_2 r-b_1)+(a_1-b_2 r)\right)\right) \\
&=\left(\frac{r_2}{a-b}\right)^2 Q(r),
\end{align*}
where
\begin{equation*}
Q(r) = -a_2(a+b+2b_2^2)r^2 + (b_1+b_2)(3a+b)r - a_1(a+b+2b_1^2).
\end{equation*}
Notice that $Q$ reaches its maximum at $$\frac{(b_1+b_2)(3a+b)}{2a_2(a+b+2b_2^2)},$$ which explains why we defined $r_1/r_2$ this way. We then compute
\begin{align*}
Q\left(\frac{r_1}{r_2}\right) &= \frac{1}{4a_2(a+b+2b_2^2)}\left((b_1+b_2)^2(3a+b)^2-4a(a+b+2b_2^2)(a+b+2b_1^2)\right) \\
&= \frac{1}{4a_2(a+b+2b_2^2)}\left((b_1^2+b_2^2)(a-b)^2-4a^3+10a^2b-8ab^2+2b^3\right) \\
&= \frac{1}{4a_2(a+b+2b_2^2)}\left((b_1^2+b_2^2)(a-b)^2+(-4a+2b)(a-b)^2\right) \\
&= \frac{(a-b)^2}{4a_2(a+b+2b_2^2)}\left((b_1+b_2)^2-4a_1a_2\right),
\end{align*}
which is positive by assumption, hence $\alpha+\beta>0$. From Figure~\ref{fig:signs_weak} we see we must be in \ref{w1} or {\color{blue}\ref{w3}} (one can in fact check that, if $b_1<b_2$ we are in \ref{w1}, and that if $b_1>b_2$ we are in \ref{w3}). Finally, with $d_{12}=d_{21}=d_\cross$ we see that having $C_k$ negative is equivalent to having
\begin{equation*}
d_\cross > \frac{\det J_*}{(\alpha+\beta)\lambda_k}. \qedhere
\end{equation*}
\end{proof}

We point out that most of the parameter sets that have been studied in the literature (see Table~\ref{tab:param} below) lead to $\alpha+\beta<0$, in which case the homogeneous steady state must remain stable for $d_{12}=d_{21}=d_{\cross}$ large enough. In Section~\ref{sec:global}, we present a global bifurcation diagram and non-homogeneous stationary solutions (including a stable one) obtained with a set of parameters satisfying the hypothesis of Theorem~\ref{th:large_cd}. 

\subsection{Varying one cross-diffusion parameter}
\label{sec:full}

The linearized study in the previous subsection allows us to understand how the the occurrence of bifurcations from the coexistence steady state $(u_*,v_*)$ is influenced by the variation of the cross-diffusion coefficients $d_{12}$ and $d_{21}$. 

In this subsection, we want to address the following question: given parameters values $r_i,\,a_i,\,b_i \, (i=1,2)$ and $d_{12},\,d_{21}$ for which bifurcations can occur, what happens if $d_{21}$ increases (of course a ``symmetric'' question can be asked about $d_{12}$)? 

This question is motivated from two different directions. First from fact that, as we already mentioned, the SKT system and in particular its steady states have been extensively studied in the triangular case $d_{21}=0$ (see e.g.~\cite{breden2018existence}). It is then natural to wonder what happens to these steady states in the more general case where the second cross-diffusion term is also taken into account. Second, the asymptotic regime where one of the cross-diffusion coefficients goes to infinity has also been studied, but mainly from the viewpoint of elliptic theory (see e.g.~\cite{LouNi99,LouNiYot04}), and we believe that bifurcation analysis and complementary numerical continuation analysis can shed a new light on this question.

\medskip

Let us consider values $r_i,\,a_i,\,b_i \, (i=1,2)$ and $d_{12},\,d_{21}$ for which, for all $k\in\N_{\geq 1}$, there exists $d>0$ such that $P_k(d)=0$, i.e. for which $C_k<0$ for all $k$. From the linearized analysis performed in Section~\ref{sec:bif}, we see that the effect of increasing $d_{21}$ depends only on the sign of $\beta$. That is, if $\beta$ is negative, then increasing $d_{21}$ increases each $C_k$ and the bifurcations gradually disappear. Indeed, from Proposition~\ref{prop:weak} and Proposition~\ref{prop:strong} we see that when $d_{21}$ reaches
\begin{equation}
\label{eq:d21_lim}
\tilde d_{21}^{k} := \dfrac{1}{|\beta|}\left(\alpha d_{12}-\dfrac{\det J_*}{\lambda_k}\right),
\end{equation}
the bifurcation associated to the $k$-th mode disappears. In particular, whenever $d_{21}$ is larger than ${\alpha d_{12}}/{\vert\beta\vert}$, there is no longer any bifurcation from the coexistence steady state $(u_*,v_*)$. This phenomena is illustrated in Section~\ref{sec:global}.

On the other hand, if $\beta$ is positive, then increasing $d_{21}$ decreases each $C_k$ and all bifurcations persist. In particular, we can compute the asymptotic value of $d$ for which each bifurcation takes place, in the limit $d_{21}\to\infty$. Indeed, for the $k$-th mode we have that the bifurcation occurs when $d$ crosses
\begin{equation*}
d_\bif^k = \frac{-B_k+\sqrt{B_k^2-4A_kC_k}}{2A_k},
\end{equation*}
and we can compute
\begin{equation}
\label{eq:d_lim}
d_{\bif,\infty}^k := \lim_{d_{21}\to +\infty} d_\bif^k = \dfrac{\beta}{u_*\lambda_k}.
\end{equation}
This situation is also investigated in more details in Section~\ref{sec:global}. 

Let us briefly mention that, if $\beta$ is equal to $0$, then increasing $d_{21}$ does not change any of the $C_k$, and all bifurcations persist. However, the computation~\eqref{eq:d_lim} is still valid in that case, hence all the bifurcation points collapse to $0$ when $d_{21}$ goes to infinity. However, and very interestingly, the global branches of non homogeneous steady states themselves do not seem to collapse to $0$, see Section~\ref{sec:beta0}. Having $\beta$ exactly equal to $0$ is probably not very meaningful from a biological point of view, but it turns out that the parameter values used in~\cite{izuhara2008reaction} to compute a bifurcation diagram in the strong competition regime and in the triangular case $d_{21}=0$ actually yield $\beta=0$. Therefore, this singular behavior is observed if one tries to extend the study~\cite{izuhara2008reaction} and considers the non-triangular case with $d_{21}$ large.

Of course, all the above discussion can be transposed to predict what happens when $d_{12}$ is increased, the sign of $\alpha$ then being the determinant factor.

\section{A more global picture}
\label{sec:global}

In this section, we both illustrate and complement the results obtained in Section~\ref{sec:without_selfdiff} by numerically computing bifurcation diagrams of steady states for~\eqref{cross}, and analysing the obtained solutions. We focus here on the one dimensional case, which is already very rich. The two-dimensional case was analysed looking at the process of pattern formation in \cite{gambino2013pattern}, and the bifurcation diagram will be investigated in a forthcoming work~\cite{CKCS}.

In dimension one, without loss of generality, we can consider $\Omega=(0,1)$. As already mentioned in the Introduction, we take $d_1=d_2=d$ as the bifurcation parameter, and also study what happens to the bifurcation diagrams when $d_{21}$ or $d_{12}$ is varied. All the parameter sets that we consider in this section are listed in Table \ref{tab:param}, where the sign of the crucial quantities $\alpha$ and $\beta$ are highlighted. The first and the second parameter sets were already studied in the literature (see the references in the table), while the third selects the weak competition case and leads to $\alpha+\beta>0$ in order to illustrate the analytical results of Theorem \ref{th:large_cd}. The fourth one corresponds to the strong competition case, but it selects a less particular case than the second parameter set (for which $\beta$ was exactly equal to $0$). For this last parameter set, we show that the system admits a non trivial bifurcation diagram even without cross-diffusion. However, all the non-homogeneous steady states obtained there seem to be unstable (in accordance with~\cite{kishimoto1985spatial}). We then numerically study how/when the addition of cross-diffusion leads to the stabilization of some of these non-homogeneous steady states.

\begin{remark} We use several conventions for the figures. Hereafter, thicker lines in the bifurcation diagrams denote stable solutions, while we use thinner lines for unstable ones. Circles, crosses and diamonds indicate the presence of branch points, fold points and Hopf points, respectively.
\end{remark}

\begin{table}[!h]
\centering
\begin{tabular}{ccccccc>{\columncolor[gray]{0.9}}c>{\columncolor[gray]{0.9}}cc}
$r_1$&$r_2$&$a_1$&$a_2$&$b_1$&$b_2$&regime&$\alpha$&$\beta$&\\
\midrule[2pt]
5&2&3&3&1&1&weak competition &$+$&$-$&\cite{iida2006diffusion,izuhara2008reaction,breden2018existence}\\
\midrule
2&5&1&1&0.5&3&strong competition  &$+$&$0$&\cite{izuhara2008reaction}\\
\midrule
15/2&16/7&4&2&6&1&weak competition  &$-$&$+$&\\
\midrule
5&5&2&3&5&4&strong competition  &$+$&$+$&\\
\bottomrule[2pt]
\end{tabular}
\caption{The parameter sets used in the numerical simulations presented in this section.}
\label{tab:param}
\end{table}

%%%%%%%%%%%%%%%%%%%%%%%%%%%%%%%%%%%%%%%%%%%%%%%%%%%%%%%%%%%

\subsection{First parameter set in Table~\ref{tab:param}}

As we already said, the first parameter set in Table \ref{tab:param} corresponds to the weak competition case and it was already used in literature \cite{iida2006diffusion, izuhara2008reaction, breden2018existence} in the triangular case ($d_{21}=0$). The corresponding bifurcation diagram, already known in the literature and carefully described in \cite{CKCS}, is reported in Figure \ref{wI_d21_0} with respect to the quantity $v(0)$. In this work, we will use another projection to display the bifurcation diagrams, namely $||u||_{L_2}$ instead of $v(0)$, as shown in Figure \ref{wI_L2_d21_0}. This representation reduces the number of branches (some branches on Figure~\eqref{wI_d21_0} are related to others by some symmetry, which leave $||u||_{L_2}$ invariant), and makes upcoming diagrams with many more branches easier to read. Different \emph{types} of solutions occur on the branches, and some of them are reported in Figure \ref{solTypes_triangular}; we refer to \cite{CKCS} for a more detailed characterization. Solutions in Figures \ref{TypeSol_d21_0_blue_3} and \ref{TypeSol_d21_0_blue_2} correspond to different points on the blue bifurcation branch in Figure \ref{v0-L2} and presenting half a bump, while Figure \ref{TypeSol_d21_0_red_1} shows a solution on the red branch, characterized by a complete bump.

\begin{figure}[!ht]
\centering
\subfloat[\label{wI_d21_0}]{
\begin{overpic}[width=0.5\textwidth,tics=10,trim={3cm 8cm 3cm 9cm},clip]{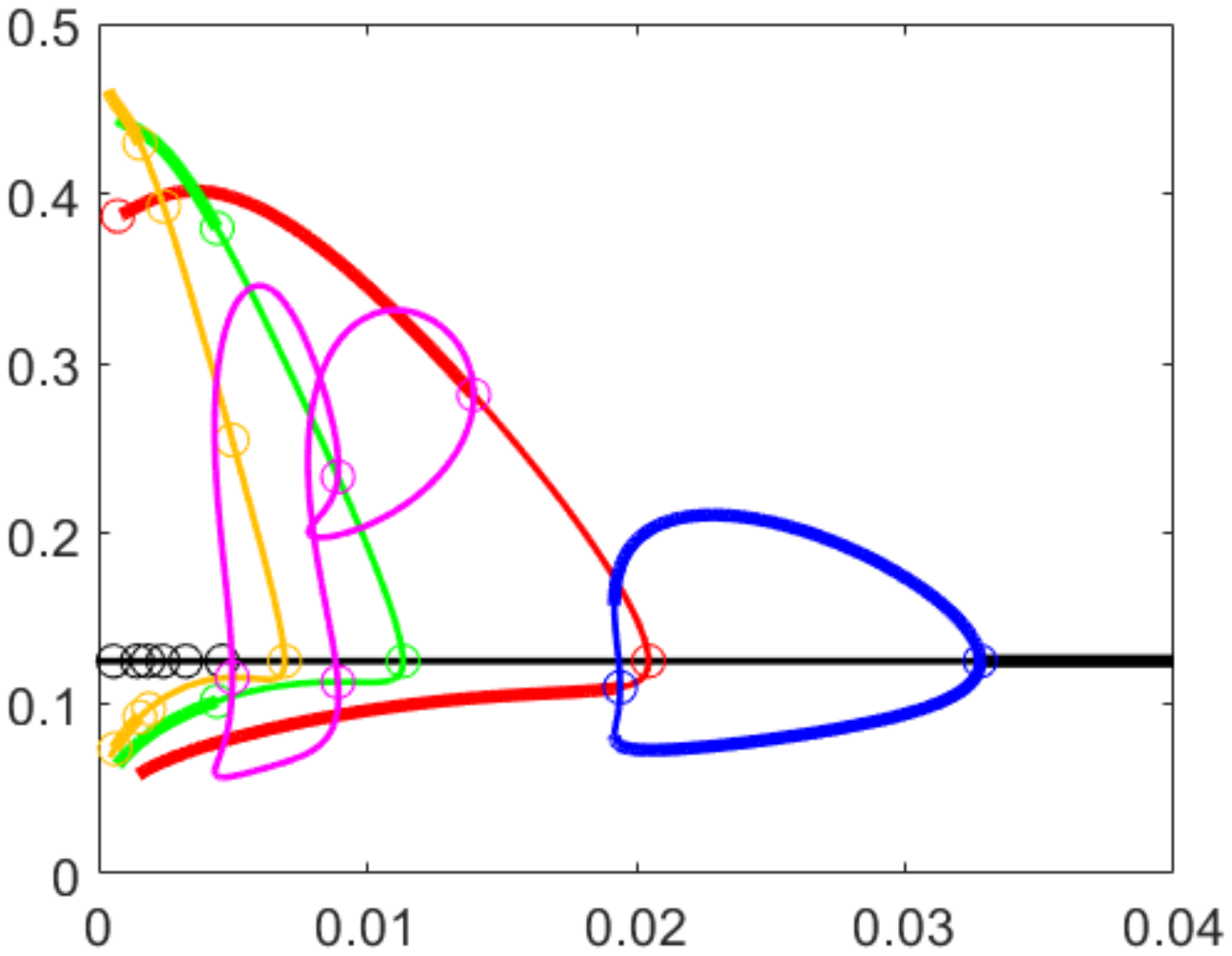}
\put(0,30){\rotatebox{90}{$v(0)$}}
\put(79,0){$d$}
\end{overpic} 
}
\subfloat[\label{wI_L2_d21_0}]{
\begin{overpic}[width=0.5\textwidth,tics=10,trim={3cm 8cm 3cm 9cm},clip]{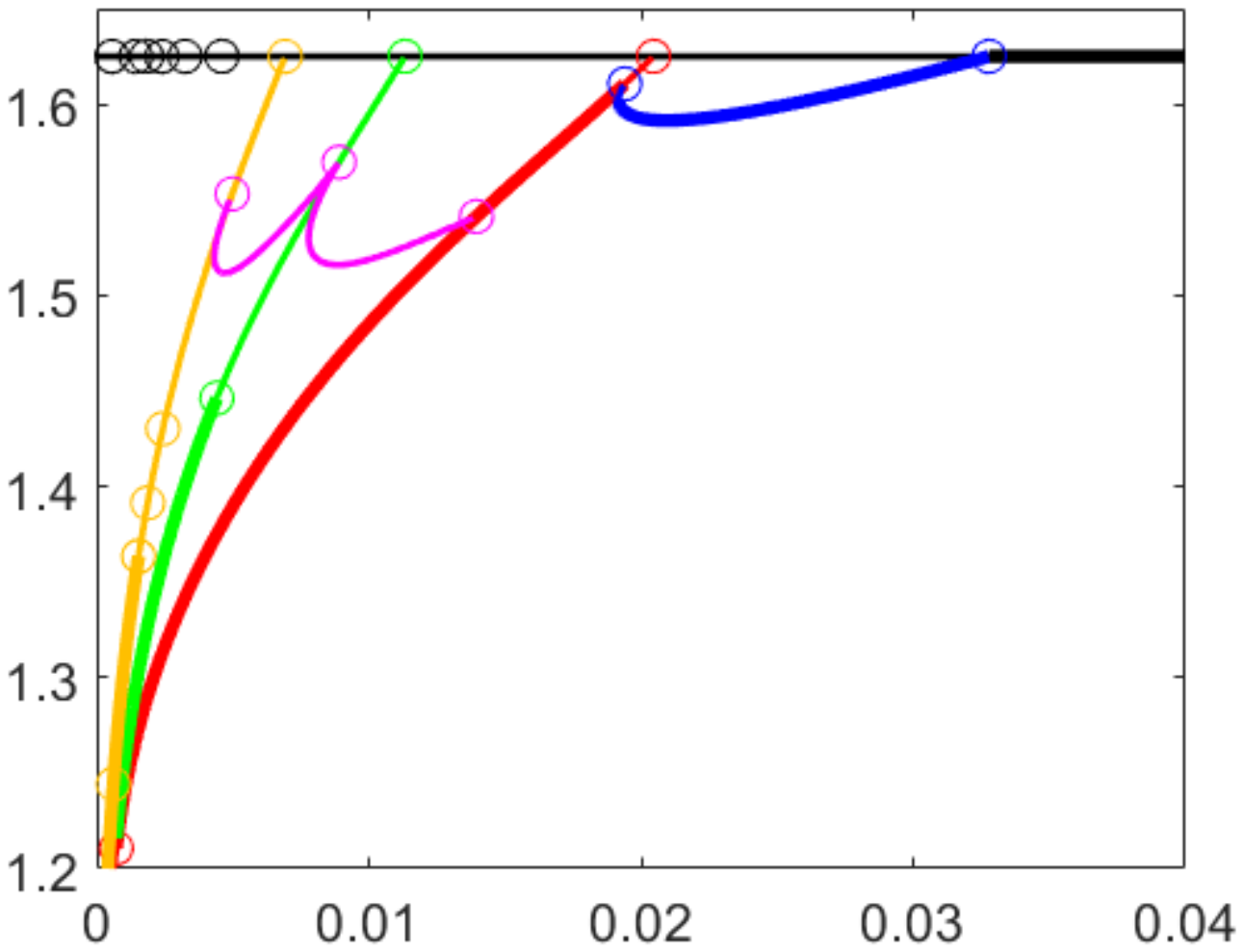}
\put(0,30){\rotatebox{90}{$||u||_{L_2}$}}
\put(79,0){$d$}
\end{overpic} 
}
\caption{Bifurcation diagrams represented with different quantities in the weak competition case (first parameter set in Table \ref{tab:param}, with $d_{12}=3$ and $d_{21}=0$). (a) ``Usual'' bifurcation diagram with respect to $v(0)$. (b) Bifurcation diagram with respect to $||u||_{L^2}$. }
\label{v0-L2}
\end{figure}

\begin{figure}[!ht]
\centering
\hspace{-0.7cm}
\subfloat[\label{TypeSol_d21_0_blue_3} blue branch, $d=3\cdot 10^{-2}$]{
\begin{overpic}[width=0.35\textwidth,tics=10,trim={3cm 8cm 3cm 8cm},clip]{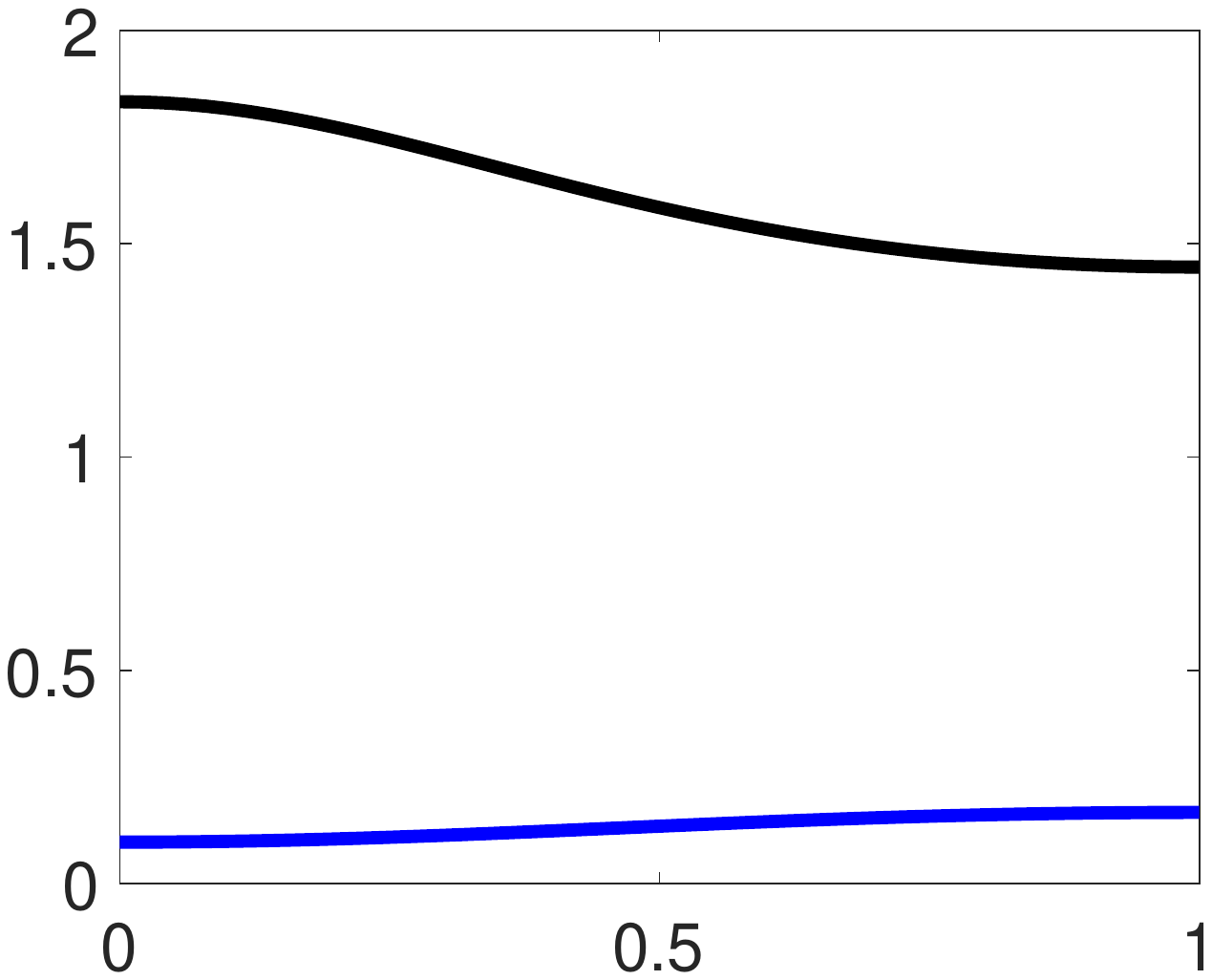}
\put(0,30){\rotatebox{90}{$u,v$}}
\put(75,2){$x$}
\end{overpic} 
}
\hspace{-0.7cm}
\subfloat[\label{TypeSol_d21_0_blue_2} blue branch, $d=2\cdot 10^{-2}$]{
\begin{overpic}[width=0.35\textwidth,tics=10,trim={3cm 8cm 3cm 8cm},clip]{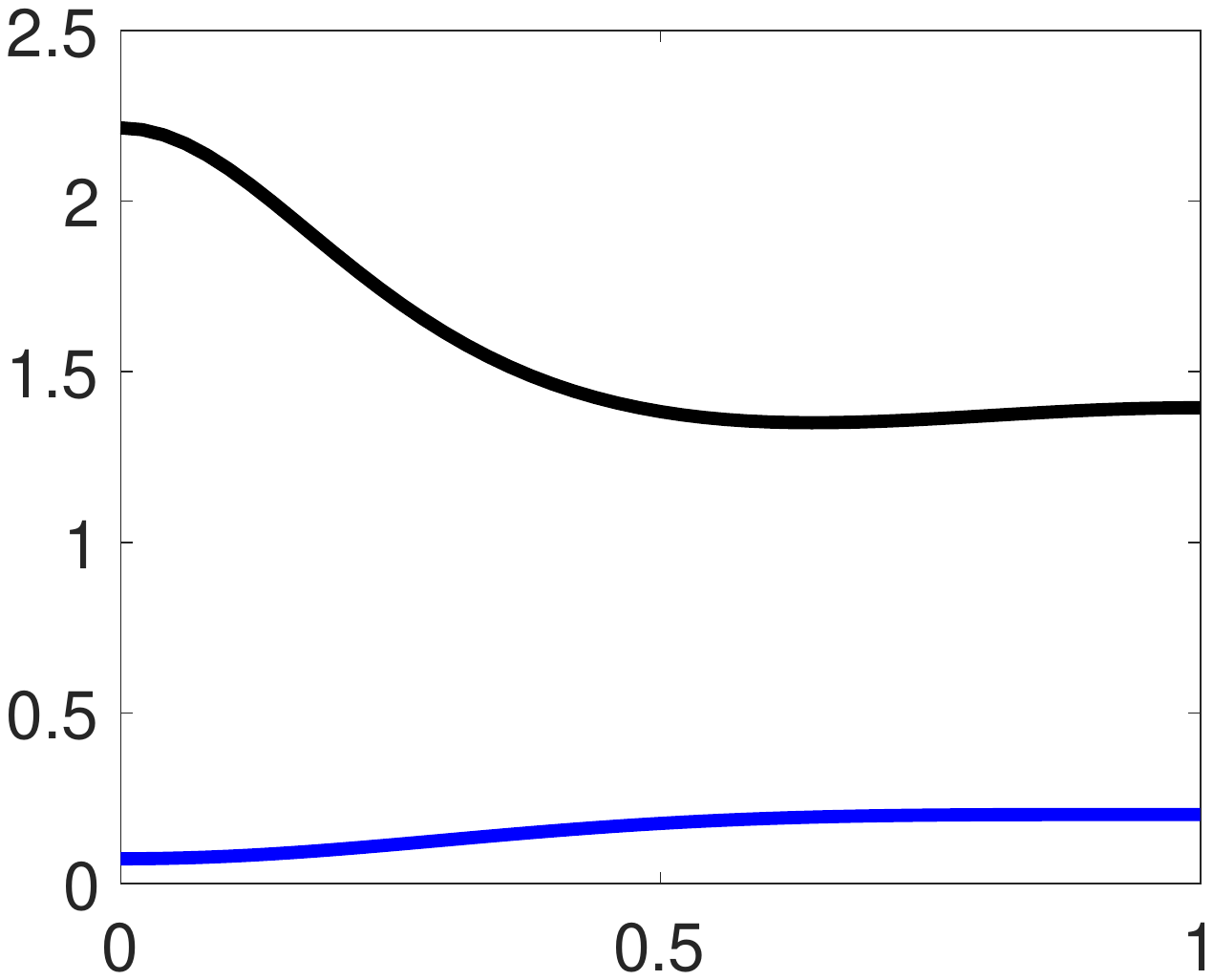}
\put(0,30){\rotatebox{90}{$u,v$}}
\put(75,2){$x$}
\end{overpic} 
}
\hspace{-0.7cm}
\subfloat[\label{TypeSol_d21_0_red_1} red branch, $d=10^{-2}$]{
\begin{overpic}[width=0.35\textwidth,tics=10,trim={3cm 8cm 3cm 8cm},clip]{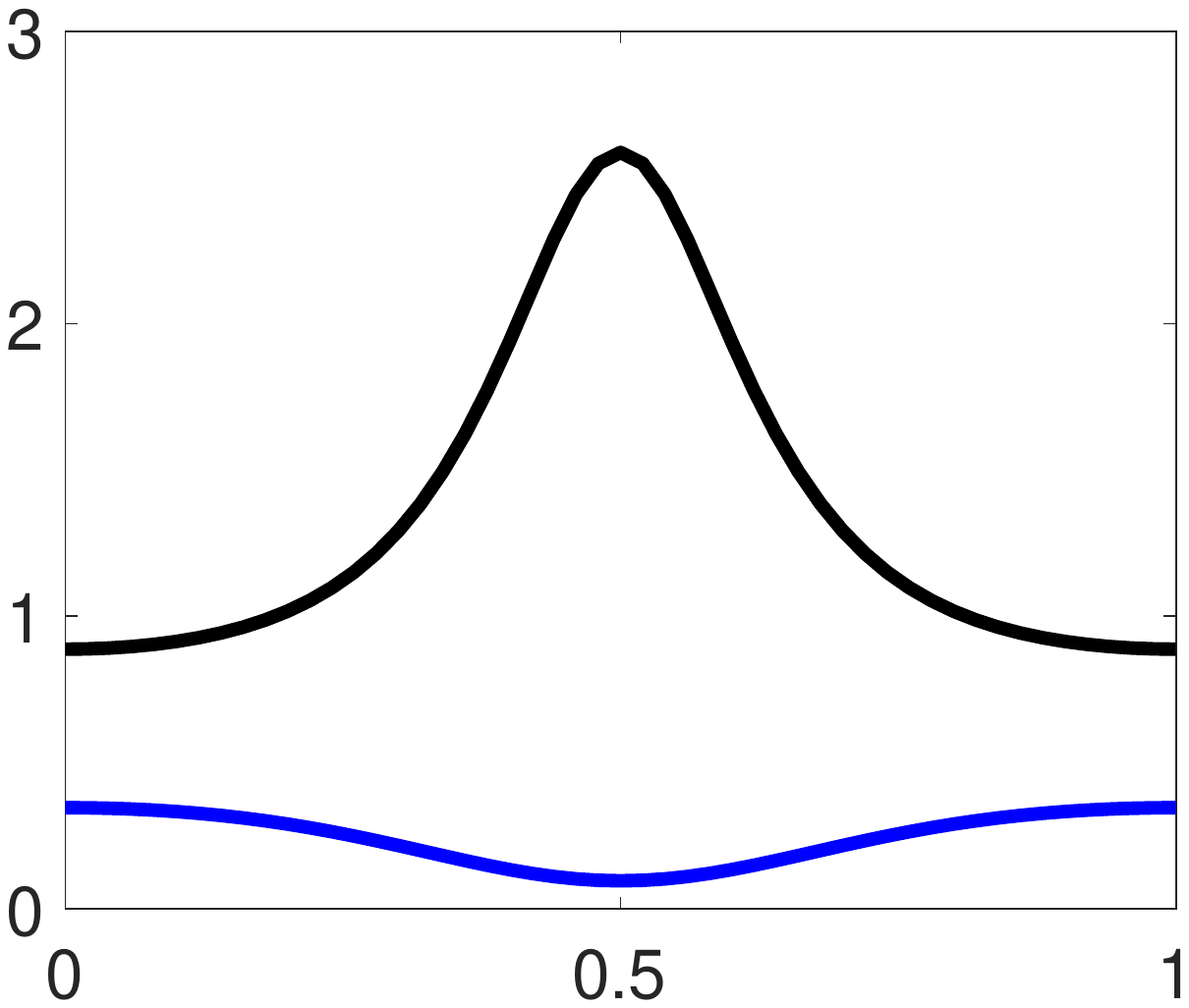}
\put(3,30){\rotatebox{90}{$u,v$}}
\put(75,2){$x$}
\end{overpic} 
}
\caption{Different types of stable solutions on the blue and red branches of the bifurcation diagram \ref{v0-L2}. The densities of $u$ and $v$ on the domain are denoted with black and blue lines respectively.}
\label{solTypes_triangular}
\end{figure}

\subsubsection{Transition between the triangular case and the full one when $d_{21}$ increases, disappearance of primary bifurcations}

With the parameter set that we choose here (the first line of Table~\ref{tab:param}, with $d_{12}=3$ and $d_{21}=0$) we are in \ref{w3} of Proposition \ref{prop:weak}, and we have $$\frac{\alpha d_{12}}{\vert\beta\vert}>\frac{\det J_*}{\lambda_1}.$$ Thus there is a bifurcation from the homogeneous steady state $(u_*,v_*)$ associated to each mode when $d_{21}=0$, as highlighted in Figure~\ref{v0-L2} (we can only display the first few, as they accumulate toward $d=0$). However, notice that $\beta$ is negative. Therefore, as discussed in Section~\ref{sec:full}, if we increase the cross-diffusion parameter $d_{21}$ the bifurcations corresponding to the first modes disappear one after the other.

In Figure~\ref{wI_d21_0p045}, corresponding to $d_{21}=0.045$, the first bifurcation branch (blue branch in Figure \ref{wI_L2_d21_0}) is no longer present, while the other branches are still bifurcating from the homogeneous one. However, stable solutions also arise on a secondary bifurcation branch (light blue), bifurcating from the red one. Choosing $d_{21}=0.055$ (Figure \ref{wI_d21_0p055}), we see that also the second (red) branch of Figure~\ref{wI_L2_d21_0} has disappeared as primary bifurcation branch from the homogeneous one, but it is still present (and it can be stable) as a secondary bifurcation branch (pastel red) from the green one. Notice that this behaviour cannot be predicted from the linear analysis, and it leads to the appearance of stable non-homogeneous stationary solutions in a parameter regime where the homogeneous steady state is still stable. In Figure \ref{TypeSol} some stable solutions (on the pastel blue and red branches for $d_{21}=0.045$, and pastel blue and pastel red branches for $d_{21}=0.055$) coexisting with the homogeneous one are reported. On these secondary branches, the solutions still present half a bump and a complete bump, as in the triangular case.

\begin{figure}[!ht]
\centering
\subfloat[\label{wI_d21_0p045}$d_{21}=0.045$]{
\begin{overpic}[width=0.5\textwidth,tics=10,trim={3cm 8.5cm 3cm 9cm},clip]{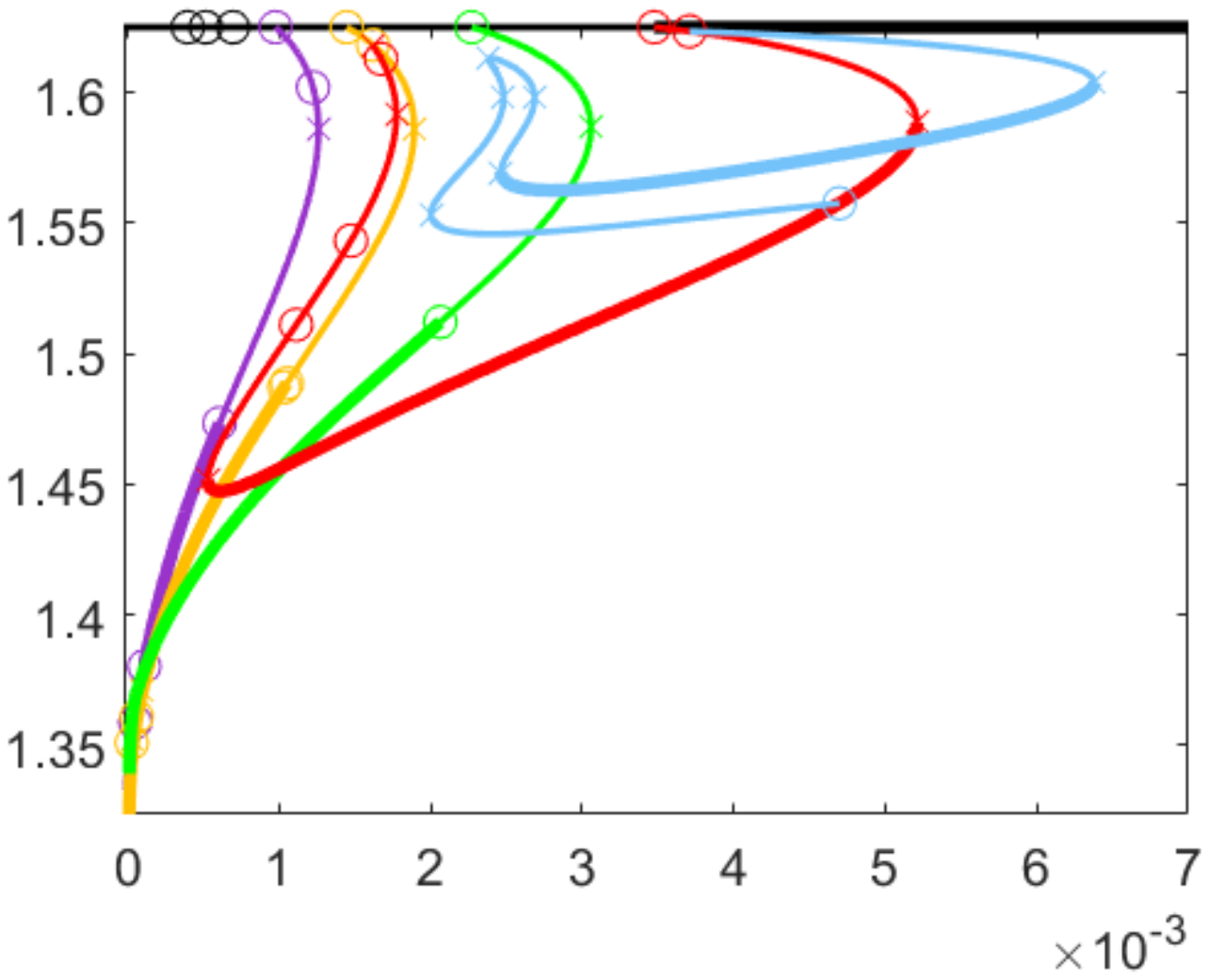}
\put(0,30){\rotatebox{90}{$||u||_{L_2}$}}
\put(90,10){$d$}
\end{overpic} 
}
\subfloat[\label{wI_d21_0p055}$d_{21}=0.055$]{
\begin{overpic}[width=0.5\textwidth,tics=10,trim={3cm 8.5cm 3cm 9cm},clip]{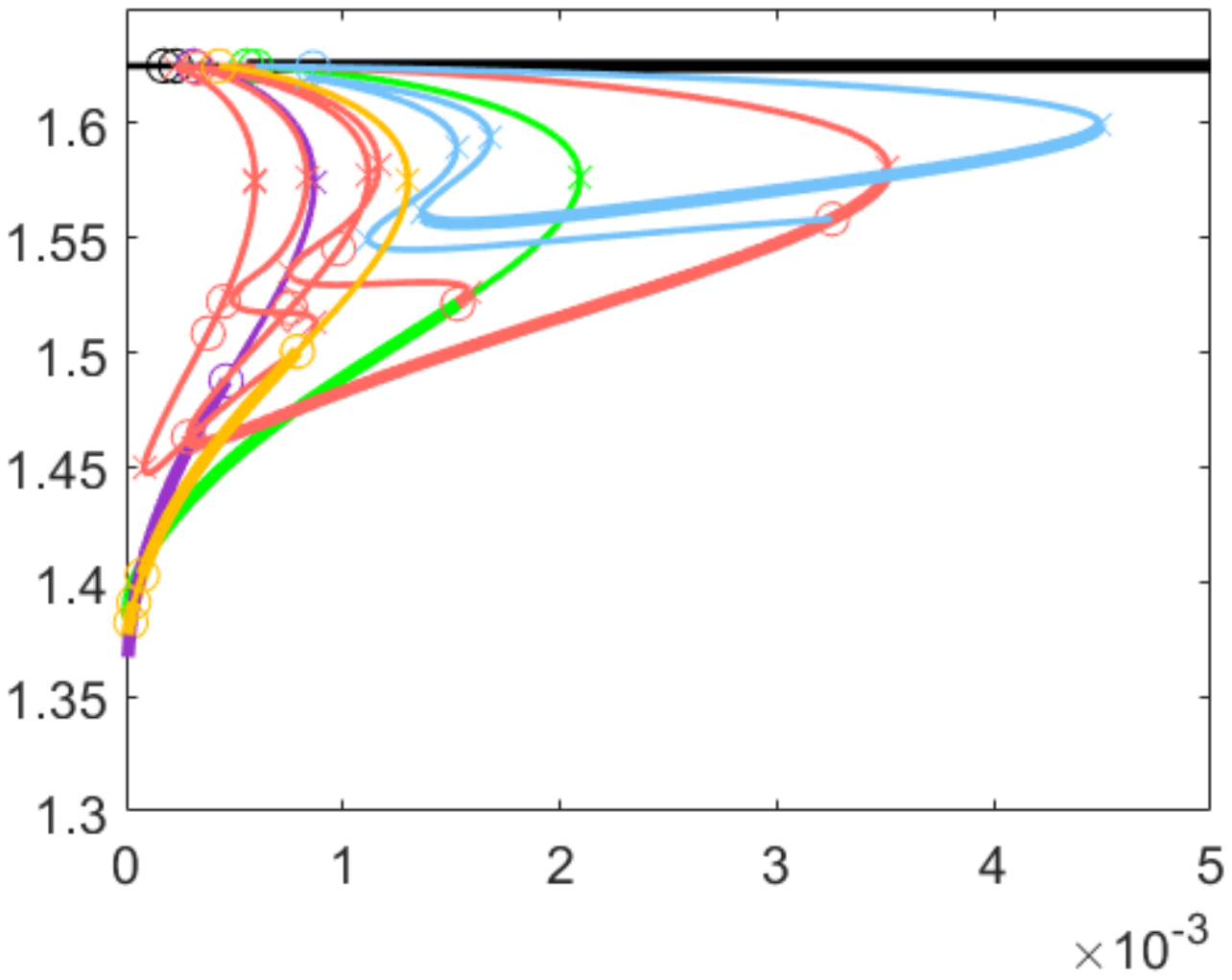}
\put(0,30){\rotatebox{90}{$||u||_{L_2}$}}
\put(90,10){$d$}
\end{overpic} 
}
\caption{Bifurcation diagrams for different values of the cross-diffusion coefficient $d_{21}$ in the weak competition case (first parameter set in Table \ref{tab:param}, with $d_{12}=3$). Notice that the range of values of $d$ for which non-trivial solutions exist is getting smaller and smaller when $d_{21}$ increases (especially compared to Figure~\ref{wI_L2_d21_0}).}
\label{DB_weak_I}
\end{figure}

\begin{figure}[!ht]
\centering
\subfloat[\label{TypeSol_p0p45_pastelblue_3} pastel blue branch, $d=3\cdot 10^{-3}$]{
\begin{overpic}[width=0.4\textwidth,tics=10,trim={3cm 8cm 3cm 8cm},clip]{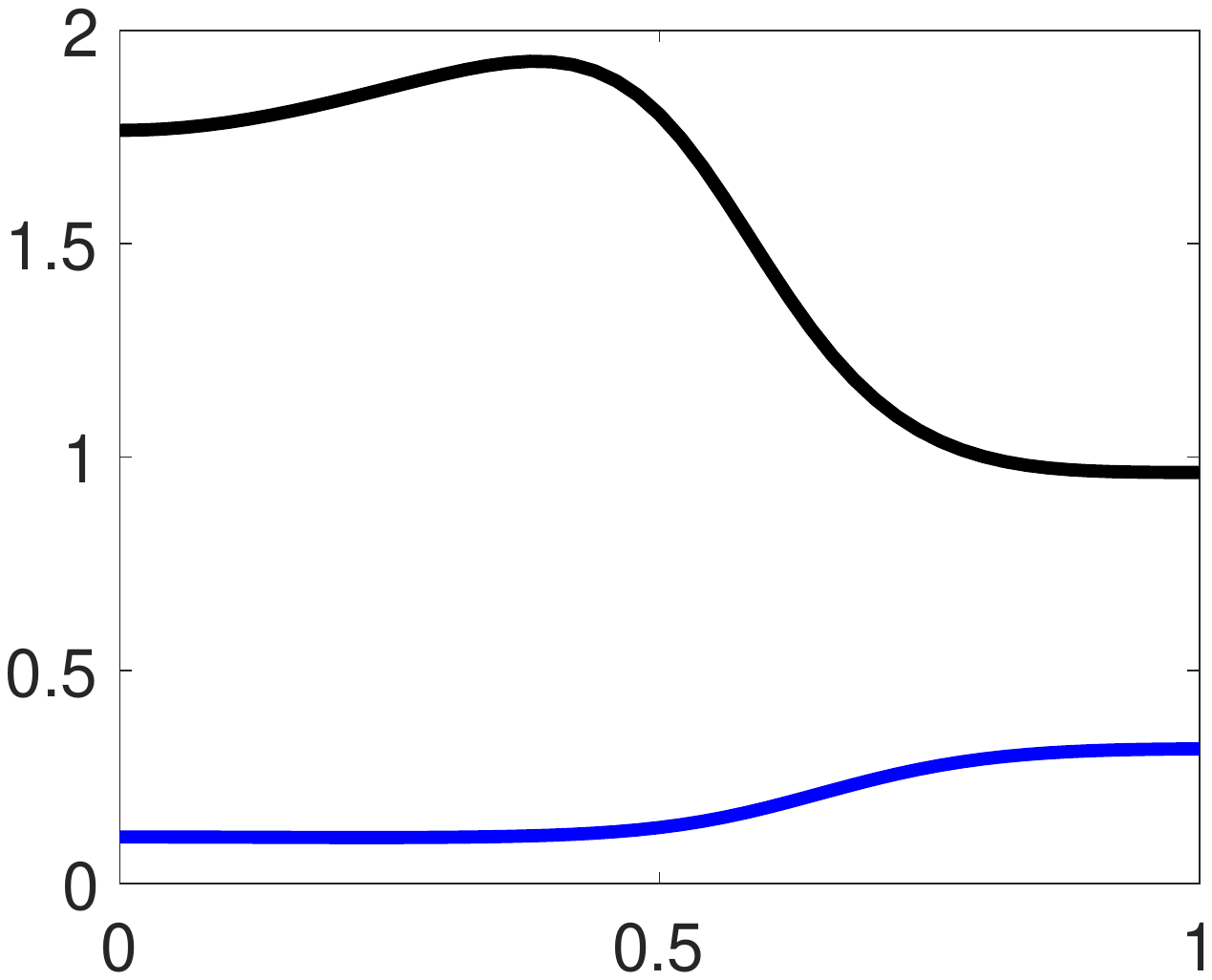}
\put(0,30){\rotatebox{90}{$u,v$}}
\put(75,2){$x$}
\end{overpic} 
}
\subfloat[\label{TypeSol_p0p55_pastelblue_1p5} pastel blue branch, $d=1.5\cdot 10^{-3}$]{
\begin{overpic}[width=0.4\textwidth,tics=10,trim={3cm 8cm 3cm 8cm},clip]{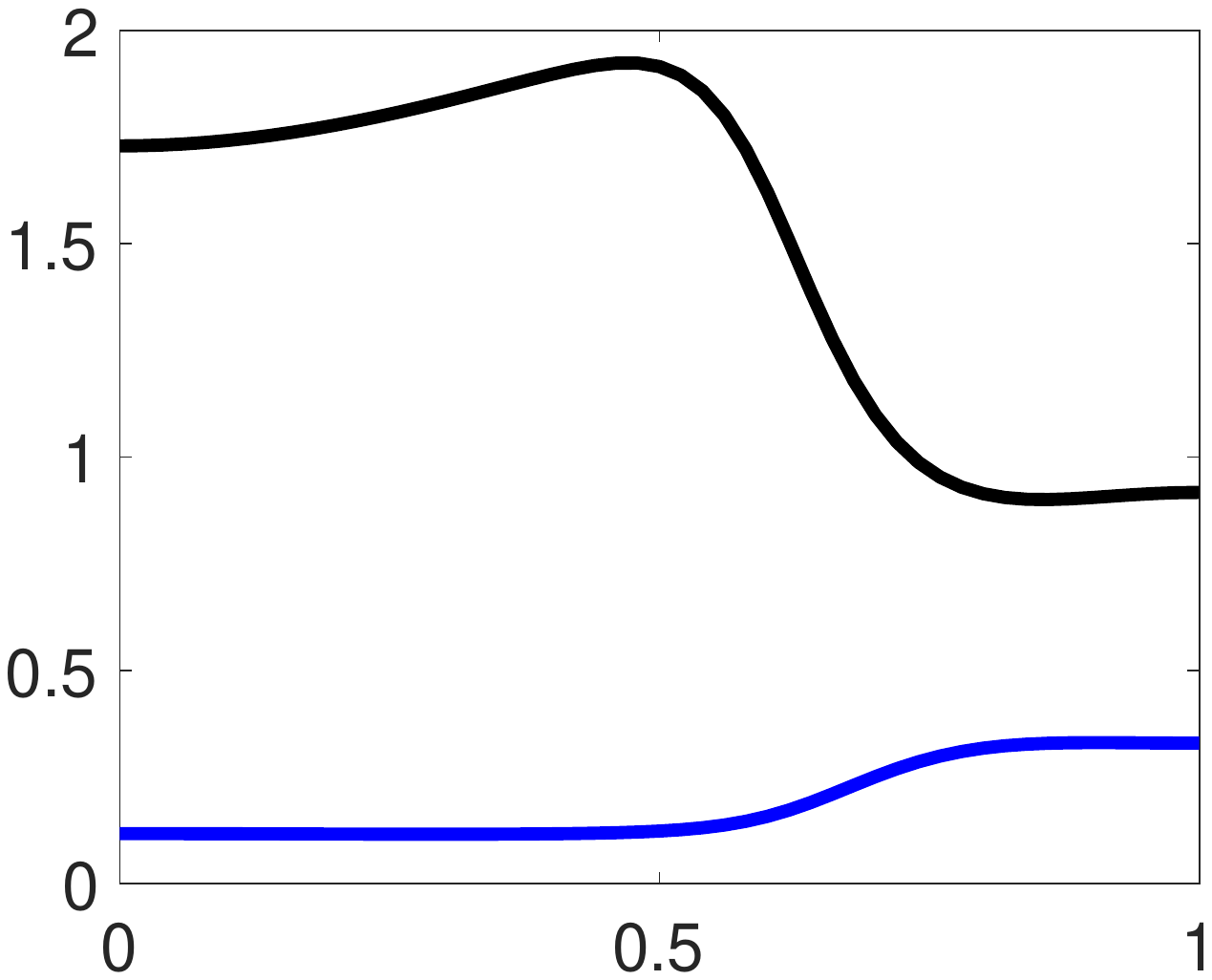}
\put(0,30){\rotatebox{90}{$u,v$}}
\put(75,2){$x$}
\end{overpic} 
}\\
\subfloat[\label{TypeSol_p0p45_red_3} red branch, $d=3\cdot 10^{-3}$]{
\begin{overpic}[width=0.4\textwidth,tics=10,trim={3cm 8cm 3cm 8cm},clip]{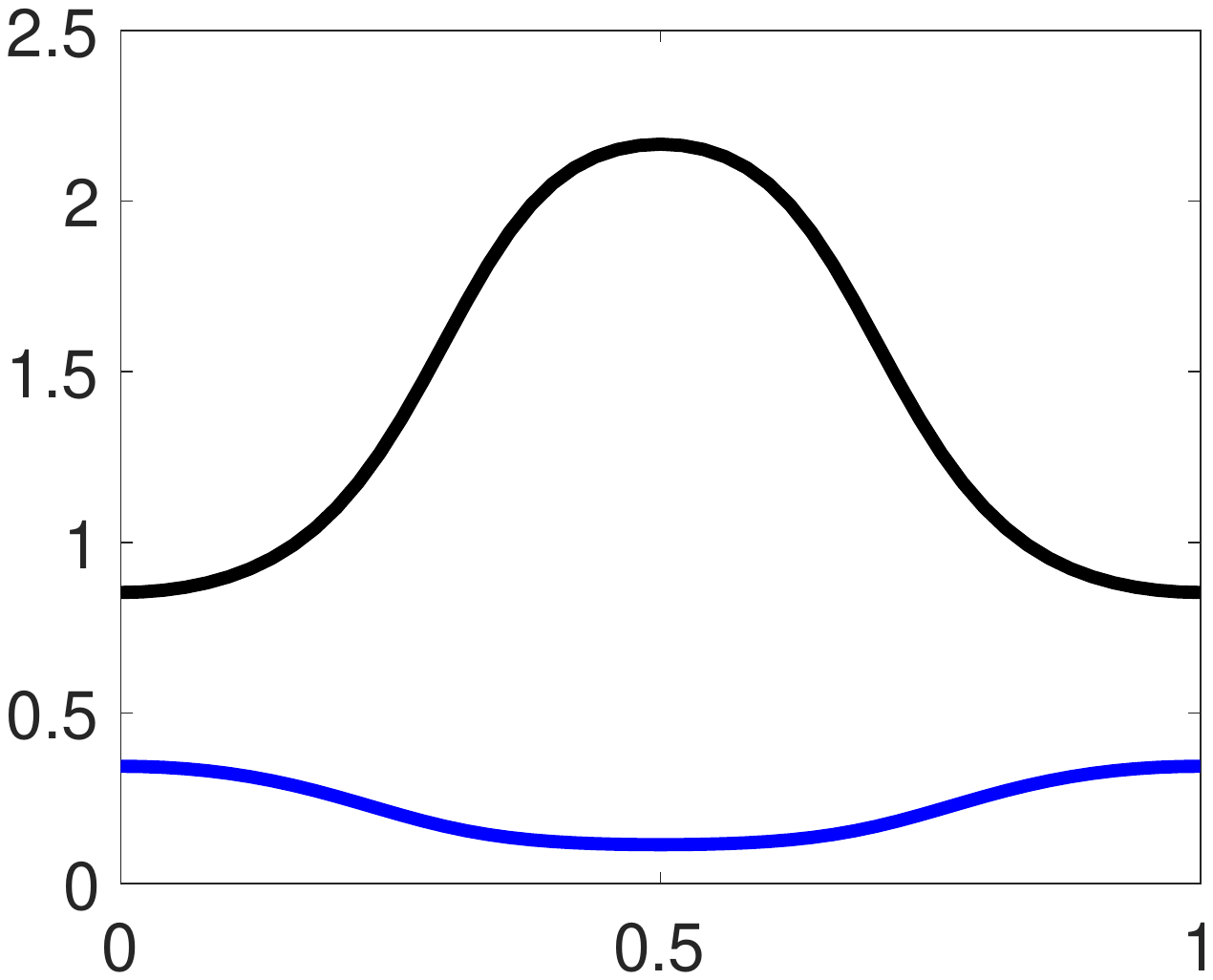}
\put(0,30){\rotatebox{90}{$u,v$}}
\put(75,2){$x$}
\end{overpic} 
}
\subfloat[\label{TypeSol_p0p55_pastelred_1p5} pastel red branch, $d=1.5\cdot 10^{-3}$]{
\begin{overpic}[width=0.4\textwidth,tics=10,trim={3cm 8cm 3cm 8cm},clip]{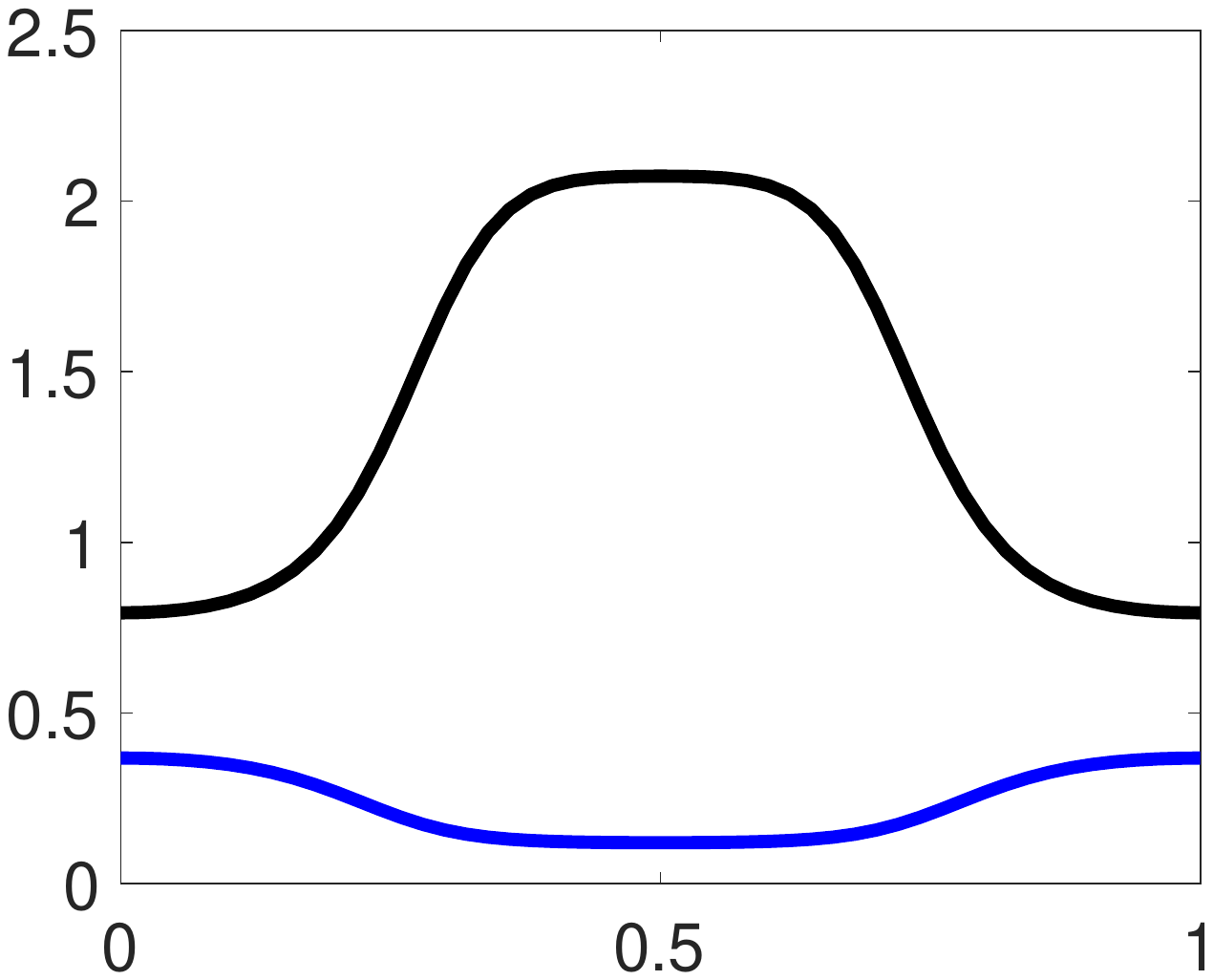}
\put(0,30){\rotatebox{90}{$u,v$}}
\put(75,2){$x$}
\end{overpic} 
}
\caption{Different types of stable solutions coexisting with the homogeneous one. Species $u$ and $v$ on the domain are denoted with black and blue lines respectively. Figures (a) and (c) refer to $d_{21}=0.045$ and correspond to solutions on the pastel blue and red branches of the bifurcation diagram in Figure \ref{wI_d21_0p045}. Figures (b) and (d) refer to $d_{21}=0.055$ and correspond to solutions on the pastel blue and pastel red branches of the bifurcation diagram in Figure \ref{wI_d21_0p055}.
}
\label{TypeSol}
\end{figure}

The disappearance of the first branch can be better visualized in Figure \ref{DB_weak_I_sequence}, where we zoom in on the first two branches and vary $d_{21}$ more slowly, showing for increasing values of $d_{21}$ the first two bifurcation branches (blue and red), corresponding to non-homogeneous solutions presenting half a bump and a complete bump (as in Figures \ref{solTypes_triangular} and \ref{TypeSol}). In Figure \ref{wI_d21_seq_0} we have $d^1_\bif>d^2_\bif$ and the first (blue) branch becomes unstable after a fold bifurcation and then it connects with the second (red) which becomes stable. In Figure \ref{wI_d21_seq_0p02} the first bifurcation point gets closer to the second but still $d^1_\bif>d^2_\bif$, the first bifurcation point is now subcritical and the bifurcation point between the two branches is closer to the homogeneous one. Note that the system exhibits now multistability as there are simultaneously locally stable steady states. In Figure \ref{wI_d21_seq_0p025} the blue branch presents a Hopf bifurcation point (which can give rise to time-periodic spatial patterns). Then the first bifurcation point crosses the second, meaning that $d^1_\bif<d^2_\bif$ (Figure \ref{wI_d21_seq_0p03}). In this case, we also see that the system can exhibit even more co-existing stable solutions for a certain range of parameters.
Further increasing $d_{21}$, the blue branch gets closer to the magenta one, shown in \ref{wI_L2_d21_0}. In Figure \ref{wI_d21_seq_0p03} the light blue branch is a secondary bifurcation branch (originating from the red one) which seems the result of the interaction of the blue and the magenta branches.  

In particular, we emphasize again that increasing $d_{21}$ leads to the appearance of stable non-homogeneous solutions coexisting with the stable homogeneous one.

\begin{figure}[!ht]
\centering
\subfloat[\label{wI_d21_seq_0}$d_{21}=0$]{
\begin{overpic}[width=0.5\textwidth,tics=10,trim={3cm 8.5cm 3cm 9cm},clip]{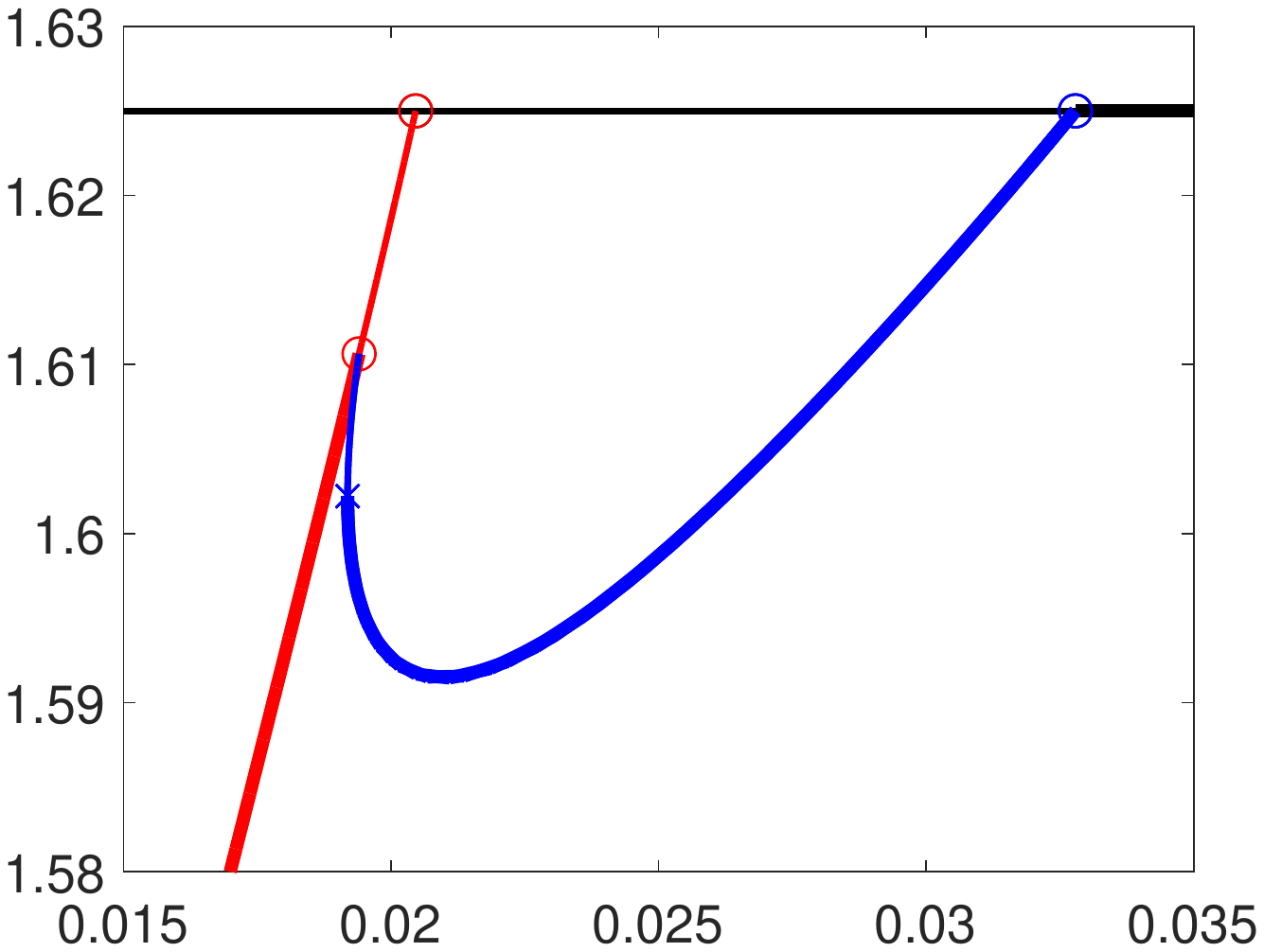}
\put(0,30){\rotatebox{90}{$||u||_{L_2}$}}
\put(90,10){$d$}
\end{overpic} 
}
\subfloat[\label{wI_d21_seq_0p02}$d_{21}=0.02$]{
\begin{overpic}[width=0.5\textwidth,tics=10,trim={3cm 8.5cm 3cm 9cm},clip]{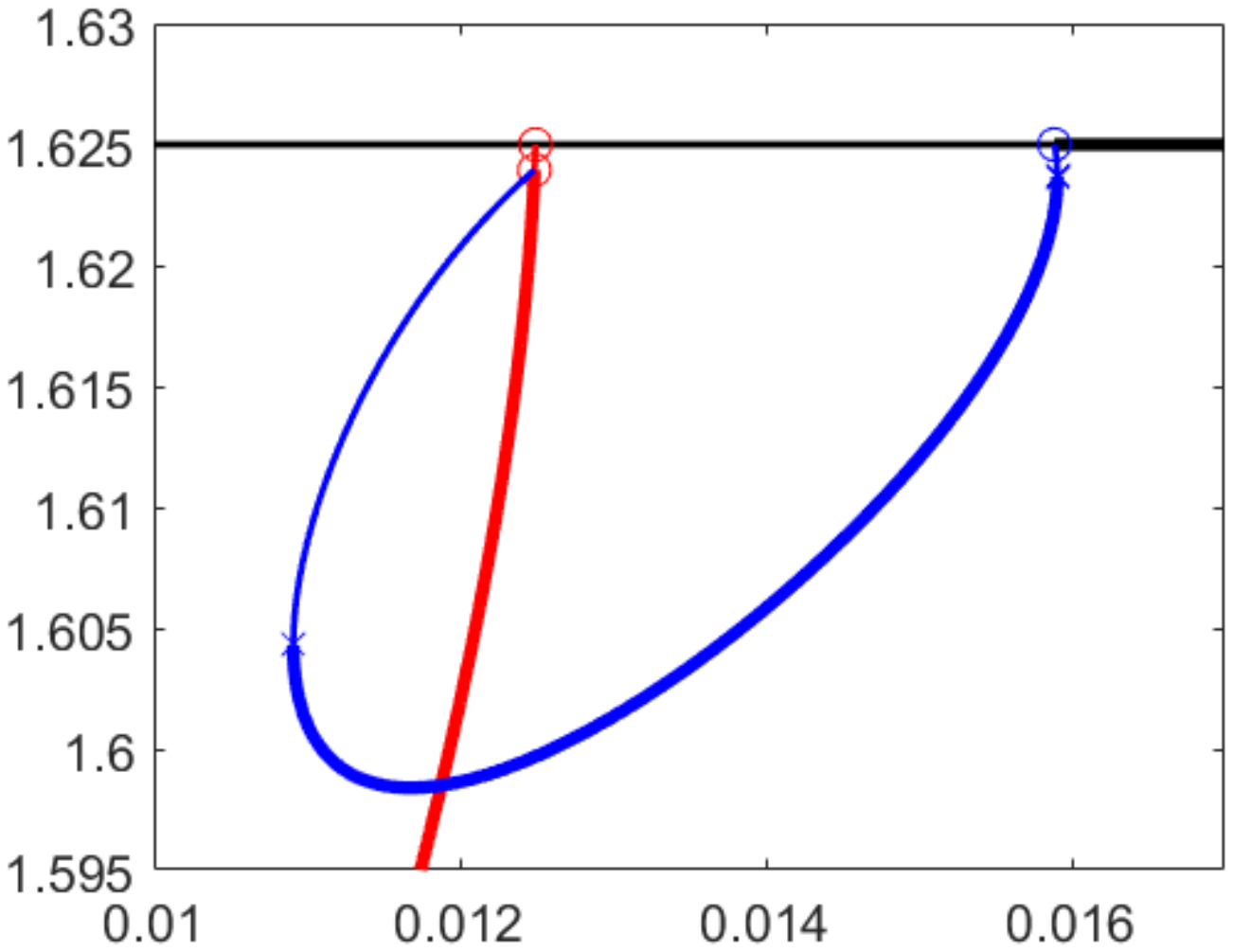}
\put(0,30){\rotatebox{90}{$||u||_{L_2}$}}
\put(90,10){$d$}
\end{overpic}
}\\
\subfloat[\label{wI_d21_seq_0p025}$d_{21}=0.025$]{
\begin{overpic}[width=0.5\textwidth,tics=10,trim={3cm 8.5cm 3cm 9cm},clip]{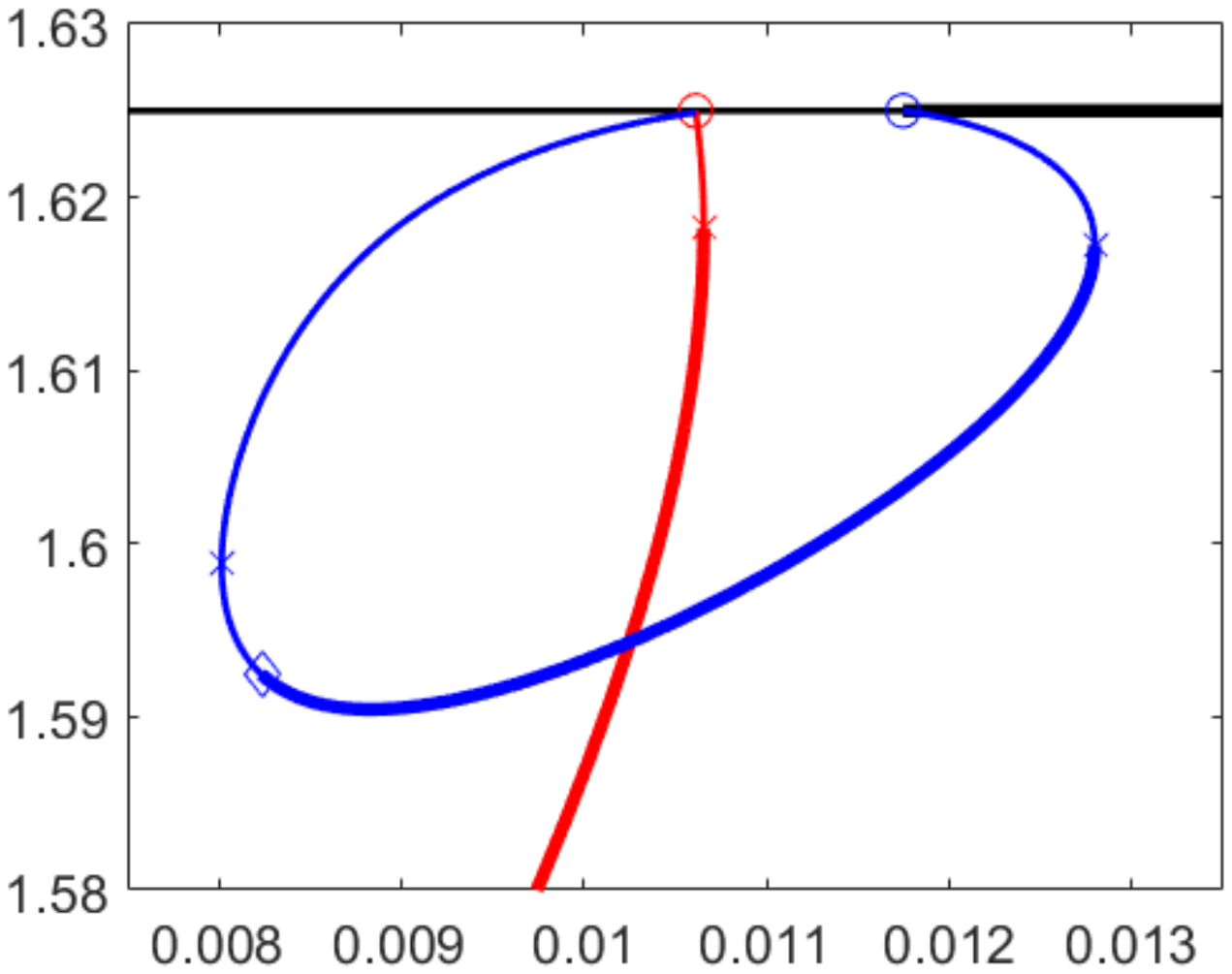}
\put(0,30){\rotatebox{90}{$||u||_{L_2}$}}
\put(90,10){$d$}
\end{overpic}
}
\subfloat[\label{wI_d21_seq_0p03}$d_{21}=0.03$]{
\begin{overpic}[width=0.5\textwidth,tics=10,trim={3cm 8.5cm 3cm 9cm},clip]{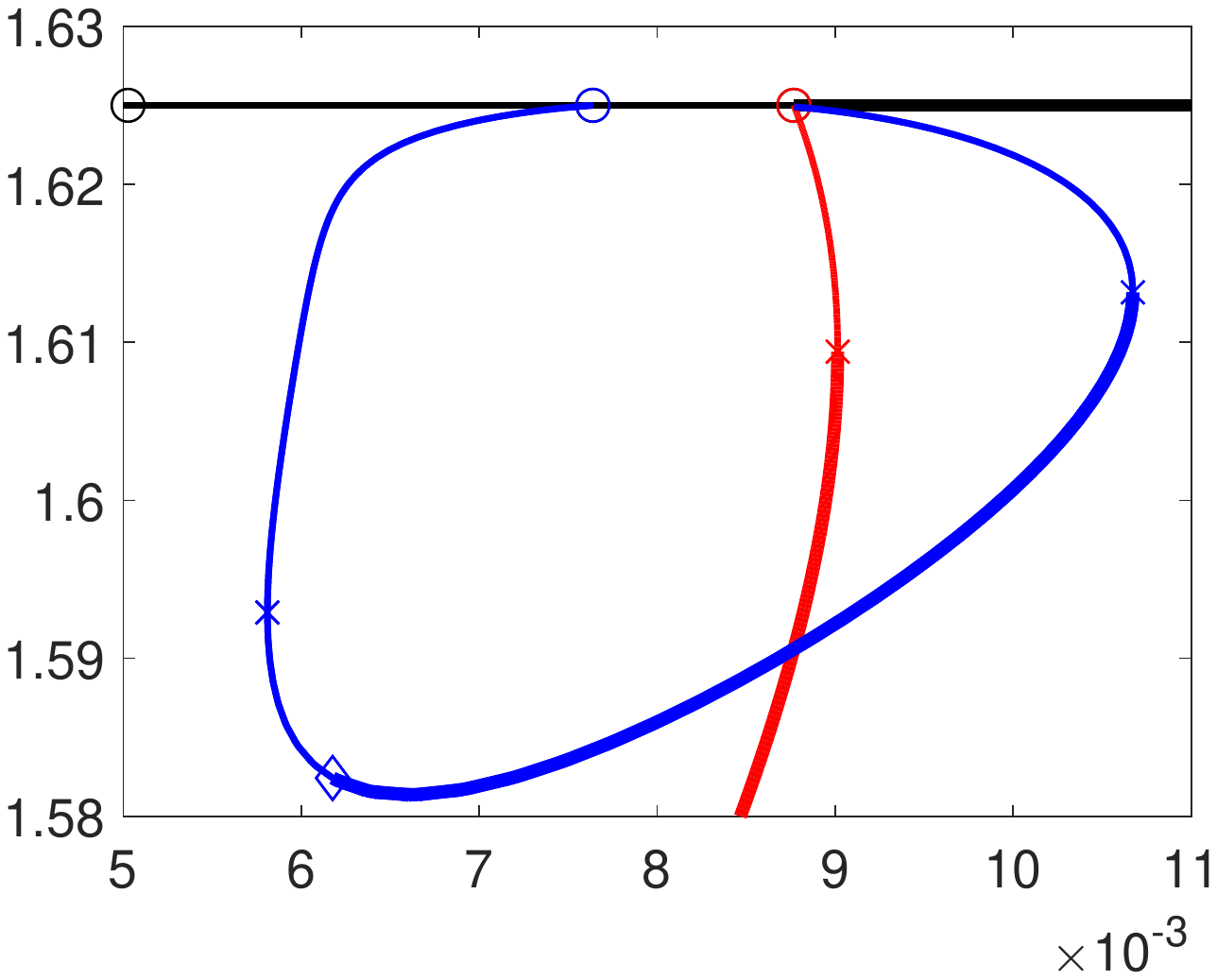}
\put(0,30){\rotatebox{90}{$||u||_{L_2}$}}
\put(90,10){$d$}
\end{overpic}
}\\
\subfloat[\label{wI_d21_seq_0p032}$d_{21}=0.032$]{
\begin{overpic}[width=0.5\textwidth,tics=10,trim={3cm 8.5cm 3cm 9cm},clip]{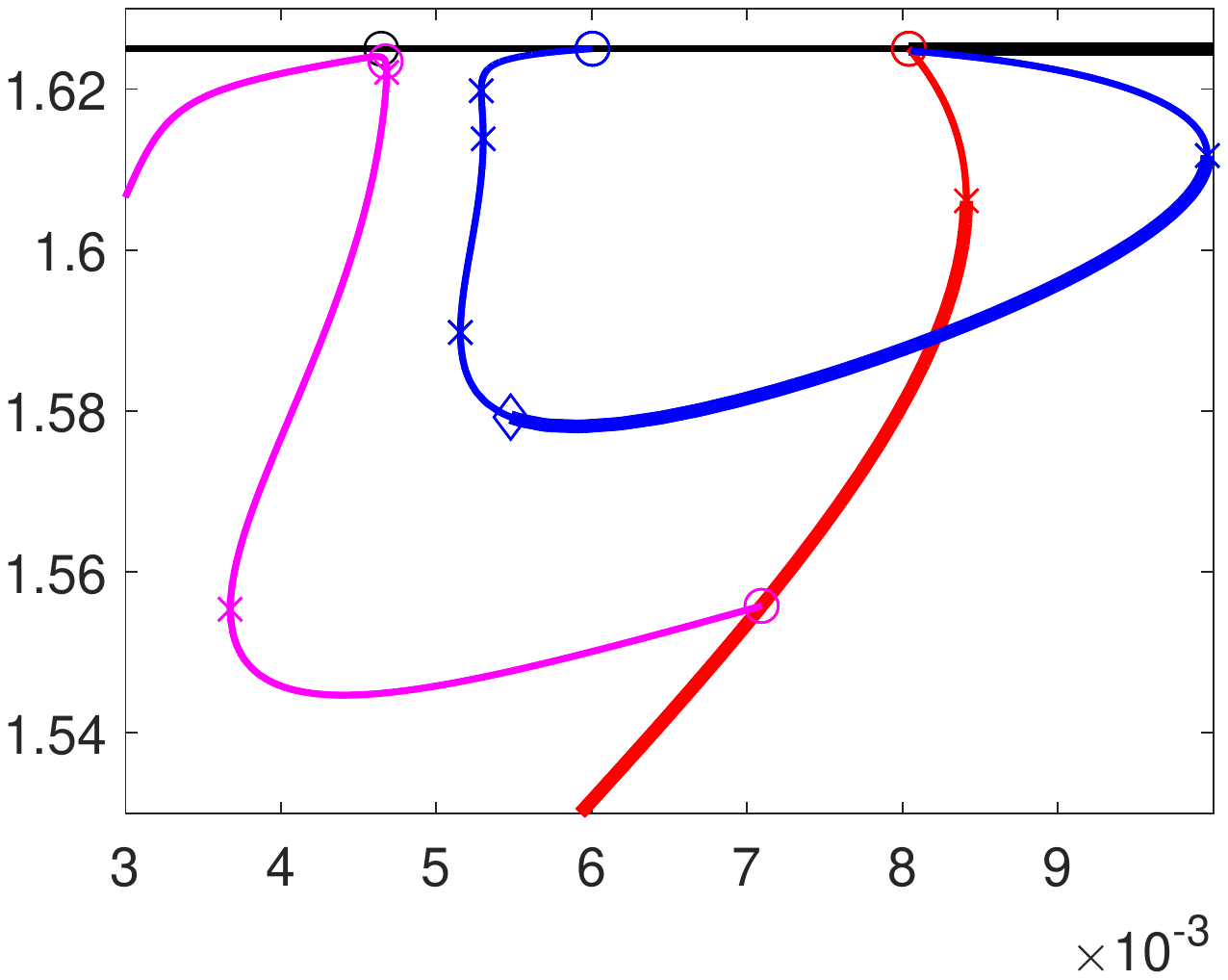}
\put(0,30){\rotatebox{90}{$||u||_{L_2}$}}
\put(90,10){$d$}
\end{overpic}
}
\subfloat[\label{wI_d21_seq_0p035}$d_{21}=0.035$]{
\begin{overpic}[width=0.5\textwidth,tics=10,trim={3cm 8.5cm 3cm 9cm},clip]{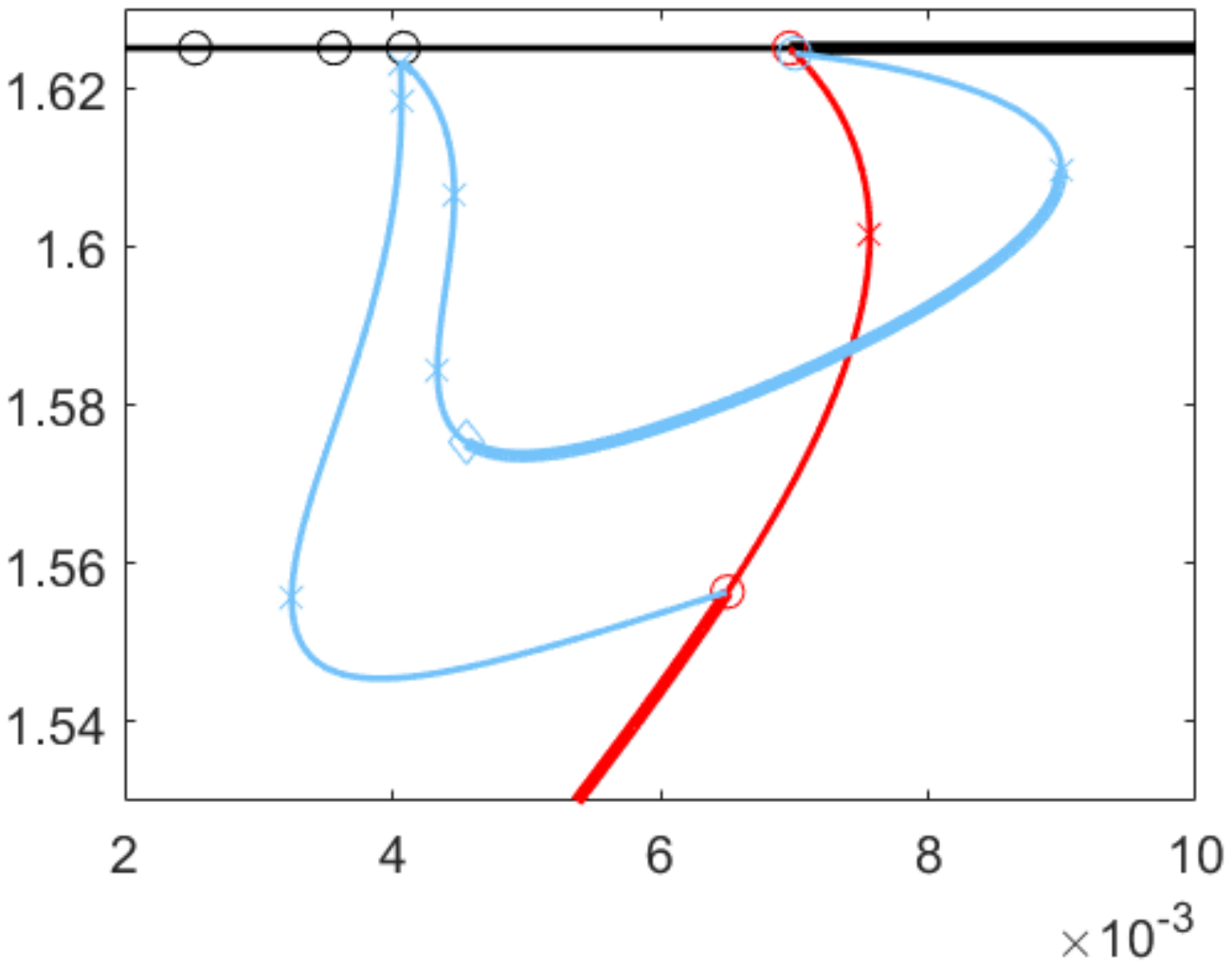}
\put(0,30){\rotatebox{90}{$||u||_{L_2}$}}
\put(90,10){$d$}
\end{overpic}
}
\caption{Bifurcation diagrams for different values of the cross-diffusion coefficient $d_{21}$ in the weak competition case (first parameter set in Table \ref{tab:param}, with $d_{12}=3$). We start in Figure~\ref{wI_d21_seq_0} with a zoom in of Figure~\ref{wI_L2_d21_0}, and then slowly increase $d_{21}$.}
\label{DB_weak_I_sequence}
\end{figure}

%%%%%%%%%%%%%%%%%%%%%%%%%%%%%%%%%%%%%%%%%%%%%%%%%%%%%%%%%%%
\subsubsection{Asymptotic behavior when $d_{12}$ goes to infinity}

We now want to verify the values of $d$ for which each bifurcation takes place, in the limit $d_{12}\to\infty$. To this end we still consider the first parameter set in Table \ref{tab:param}, fixing $d_{21}=0$ and increasing $d_{12}$. In this case we have that $\alpha>0,\;\beta<0$, therefore the limiting bifurcation values are given by
\begin{equation}
\label{eq:d_lim_d12}
d_{\bif,\infty}^k := \lim_{d_{12}\to +\infty} d_\bif^k = \dfrac{\alpha}{v_*\lambda_k},
\end{equation}
corresponding to \eqref{eq:d_lim}. In Table \ref{tab:conv_d12increasing} the convergence of the first three bifurcation points, numerically detected, to the predicted limited values is shown. In particular we observe that the values $d_\bif^k$ increase when $d_{12}$ increases, in contrast to the behaviour with respect to $d_{21}$. Looking at the solutions for $d_{12}=1000$ in Figure \ref{d12incr_sol_uv}, we observe that, when $d_{12}$ is large, the product of the two densities $uv$ (red dotted line) is close to constant, as predicted in~\cite{LouNi99}. 

\begin{table}[!h]
\centering
\begin{tabular}{cccc}
\midrule[2pt]
 $d_{12}$ & $d^1_{\bif}$& $d^2_{\bif}$ & $d^3_{\bif}$ \\
\midrule[2pt]
3 & 0.0328& 0.0205& 0.0113\\
\midrule
10 & 0.0762& 0.0273& 0.0131\\
\midrule
100  & 0.1190 & 0.0311 & 0.0139\\
\midrule
1000  & 0.1258 & 0.0315 & 0.0140\\
\midrule
\midrule
$d^k_{\bif,\infty}$  & 0.1267 &  0.0317 & 0.0141\\          
\bottomrule[2pt]
\end{tabular}
\caption{Convergence of the first three bifurcation points, numerically detected, to the predicted limited values  $\alpha/(v_*\lambda_k)$ when $d_{12}$ increases ($d_{21}=0$).}
\label{tab:conv_d12increasing}
\end{table}

\begin{figure}[!ht]
\centering
\subfloat[\label{d12incr_1000_1}]{
\begin{overpic}[width=0.4\textwidth,tics=10,trim={3cm 8.5cm 3cm 9cm},clip]{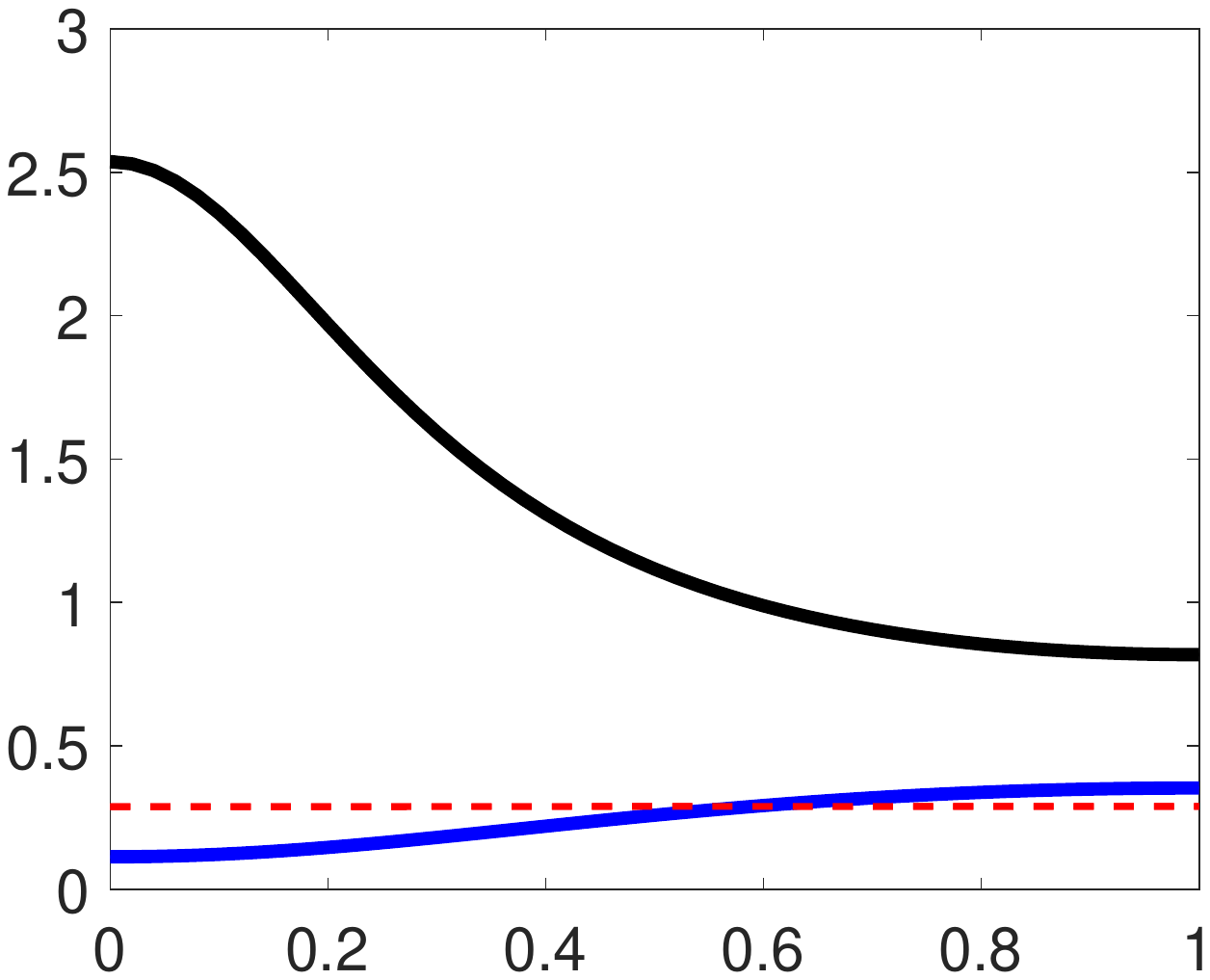}
\put(5,30){\rotatebox{90}{$u,\;v,\;uv$}}
\put(80,-2){$x$}
\end{overpic} 
}
\subfloat[\label{d12incr_1000_2}]{
\begin{overpic}[width=0.4\textwidth,tics=10,trim={3cm 8.5cm 3cm 9cm},clip]{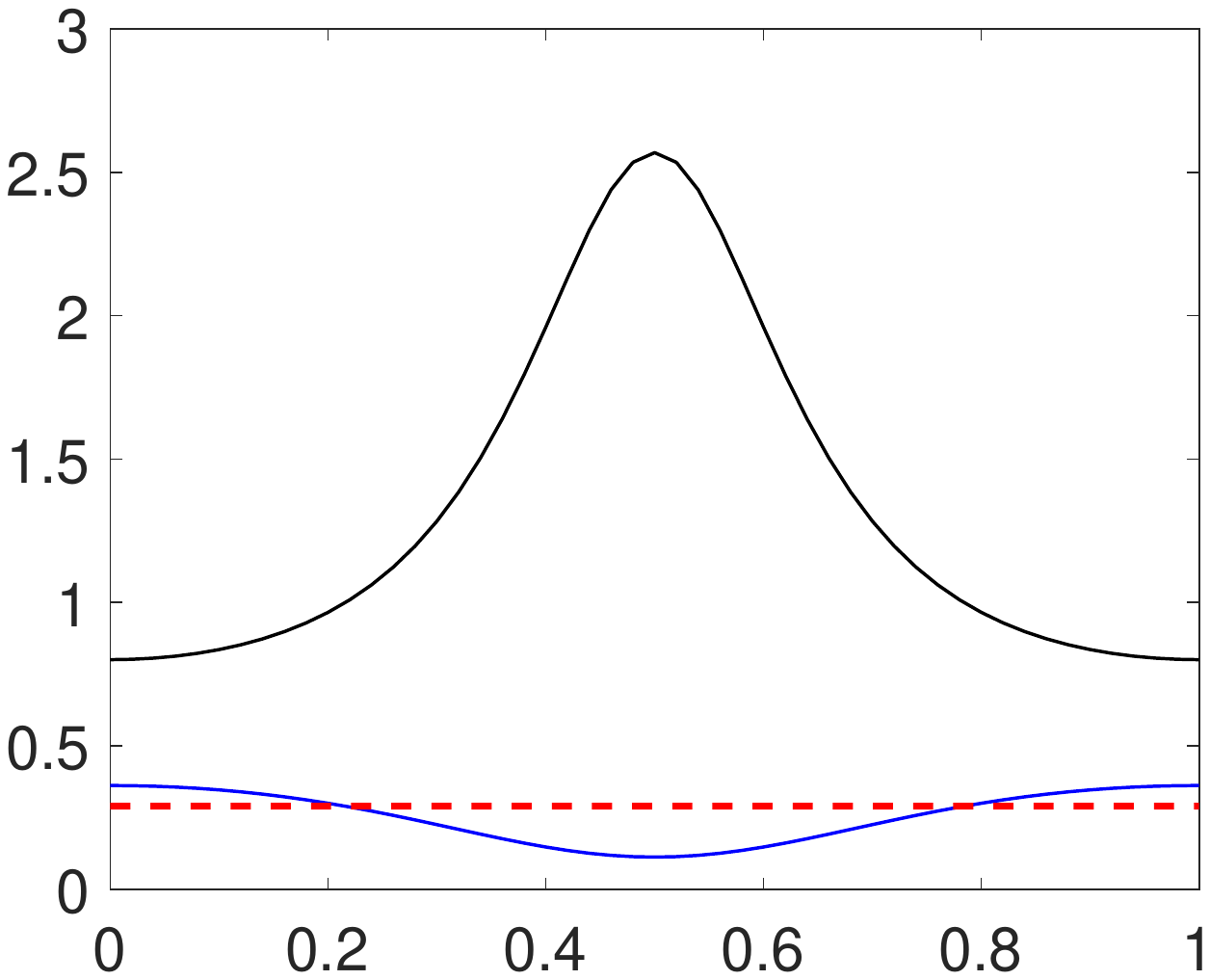}
\put(5,30){\rotatebox{90}{$u,\;v,\;uv$}}
\put(80,-2){$x$}
\end{overpic} 
}
\caption{Solutions belonging to different branches when $d_{12}=1000$. The species $u$ and $v$ are denoted with black and blue lines respectively, while the red line corresponds to $uv$.  As in the bifurcation diagrams, thin lines indicate unstable solutions, while thicker lines corresponds to stable ones. The red dotted line represents the product $uv$.}
\label{d12incr_sol_uv}
\end{figure}

%%%%%%%%%%%%%%%%%%%%%%%%%%%%%%%%%%%%%%%%%%%%%%%%%%%%%%%%%%%

\subsection{Second parameter set in Table~\ref{tab:param}: singular behavior due to having $\beta=0$}
\label{sec:beta0}

We consider now the second parameter set of Table \ref{tab:param}, already used in \cite{izuhara2008reaction} and corresponding to a very specific case in the strong competition regime, since $\beta=0$. According to \eqref{eq:d_lim}, we expect that the bifurcation points on the homogeneous branch collapse to zero as the cross-diffusion coefficient $d_{21}$ is increased, but we do not have information about the behaviour of the bifurcating branches. In Figure \ref{strong_BD_d21}, the bifurcation diagrams for increasing values of the cross-diffusion coefficient $d_{21}$ are reported, showing that, at first, the whole diagram seems to collapse as $d_{21}$ is increased. Incidentally, we can also see that at the same time the stable part of the first (blue) branch (arising from a Hopf bifurcation) shrinks, reducing the possibility of having stable inhomogeneous solutions. However, further increasing the value of $d_{21}$ we observe a different trend, also shown in Figure \ref{strong_mimura_d12increasing}. Here we plot on the same graph the first branch for increasing values of $d_{21}$: the lightest blue corresponds to $d_{21}=0$, the darkest to $d_{21}=1000$, while as usual the homogeneous solution is denoted in black. As predicted, the bifurcation points move to zero, but the branch folds (the corresponding value in the picture is $d_{21}=20$) and then expands again, leading to the appearance of stable inhomogeneous solutions. The same considerations apply to the other branches.

\begin{figure}[!ht]
\centering
\subfloat[\label{strong_d21_0}$d_{21}=0$]{
\begin{overpic}[width=0.5\textwidth,tics=10,trim={3cm 8.5cm 3cm 9cm},clip]{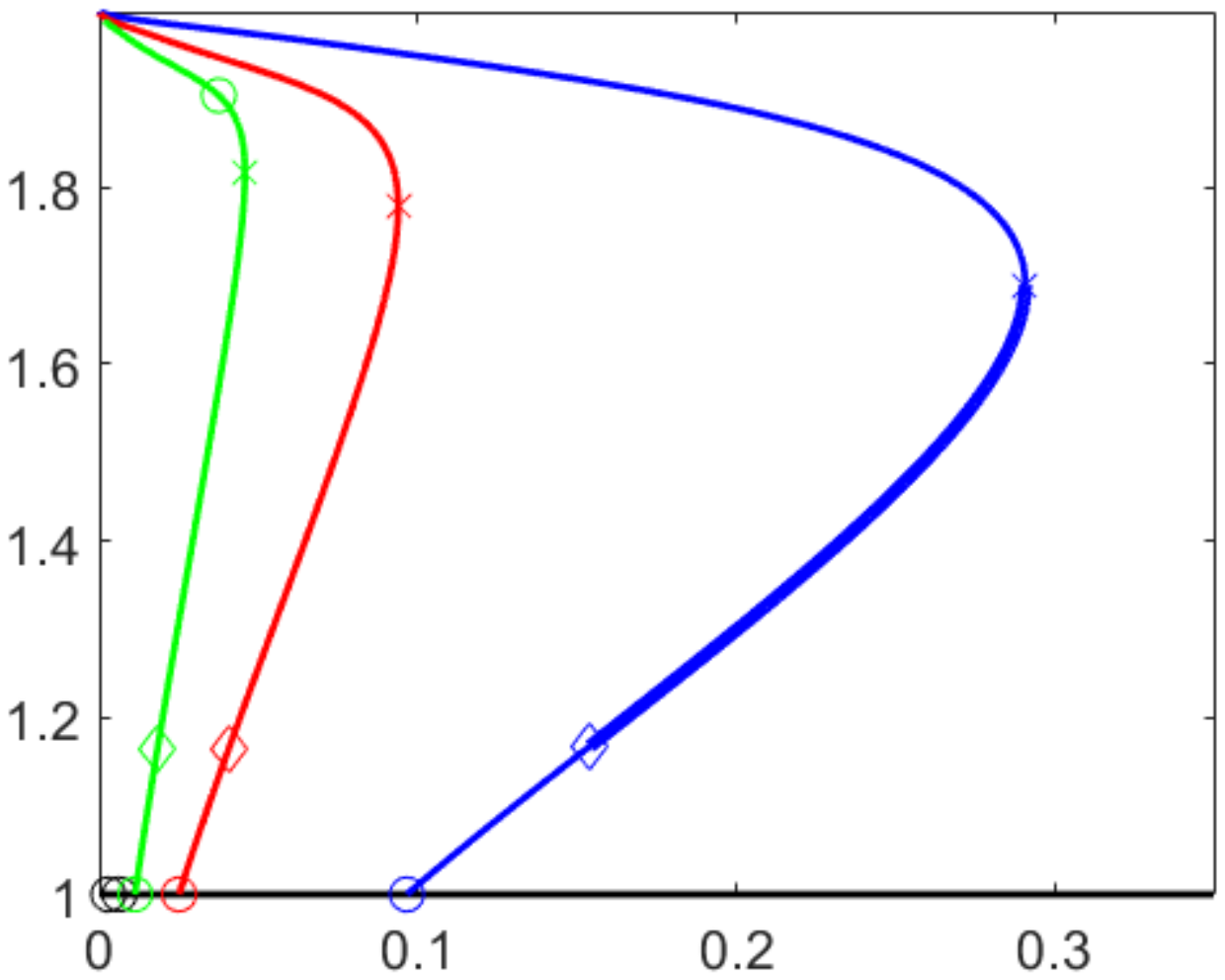}
\put(0,30){\rotatebox{90}{$||u||_{L_2}$}}
\put(85,-2){$d$}
\end{overpic} 
}
\subfloat[\label{strong_d21_3}$d_{21}=3$]{
\begin{overpic}[width=0.5\textwidth,tics=10,trim={3cm 8.5cm 3cm 9cm},clip]{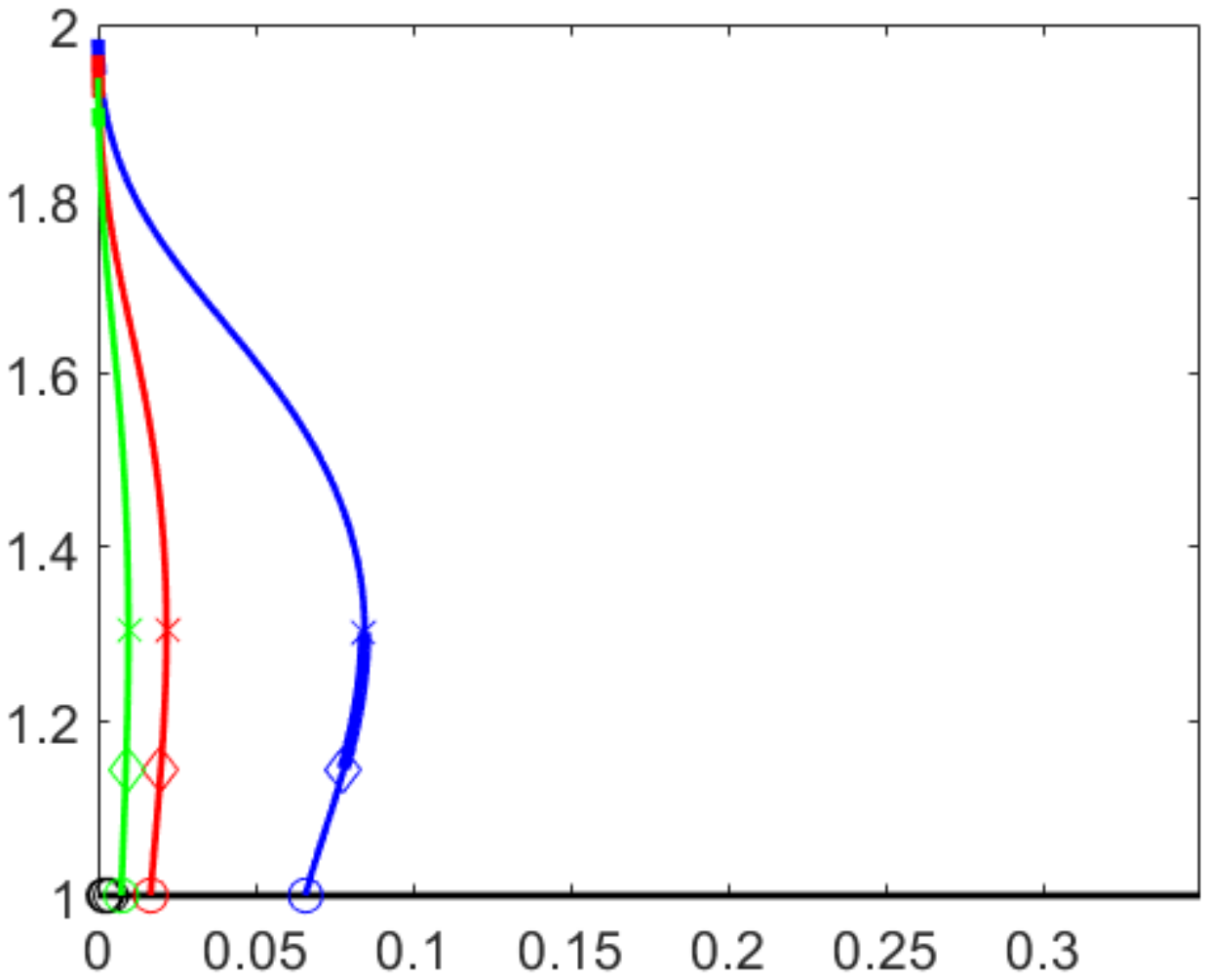}
\put(0,30){\rotatebox{90}{$||u||_{L_2}$}}
\put(85,-2){$d$}
\end{overpic} 
}\\
\subfloat[\label{strong_d21_6}$d_{21}=6$ ]{
\begin{overpic}[width=0.5\textwidth,tics=10,trim={3cm 8.5cm 3cm 9cm},clip]{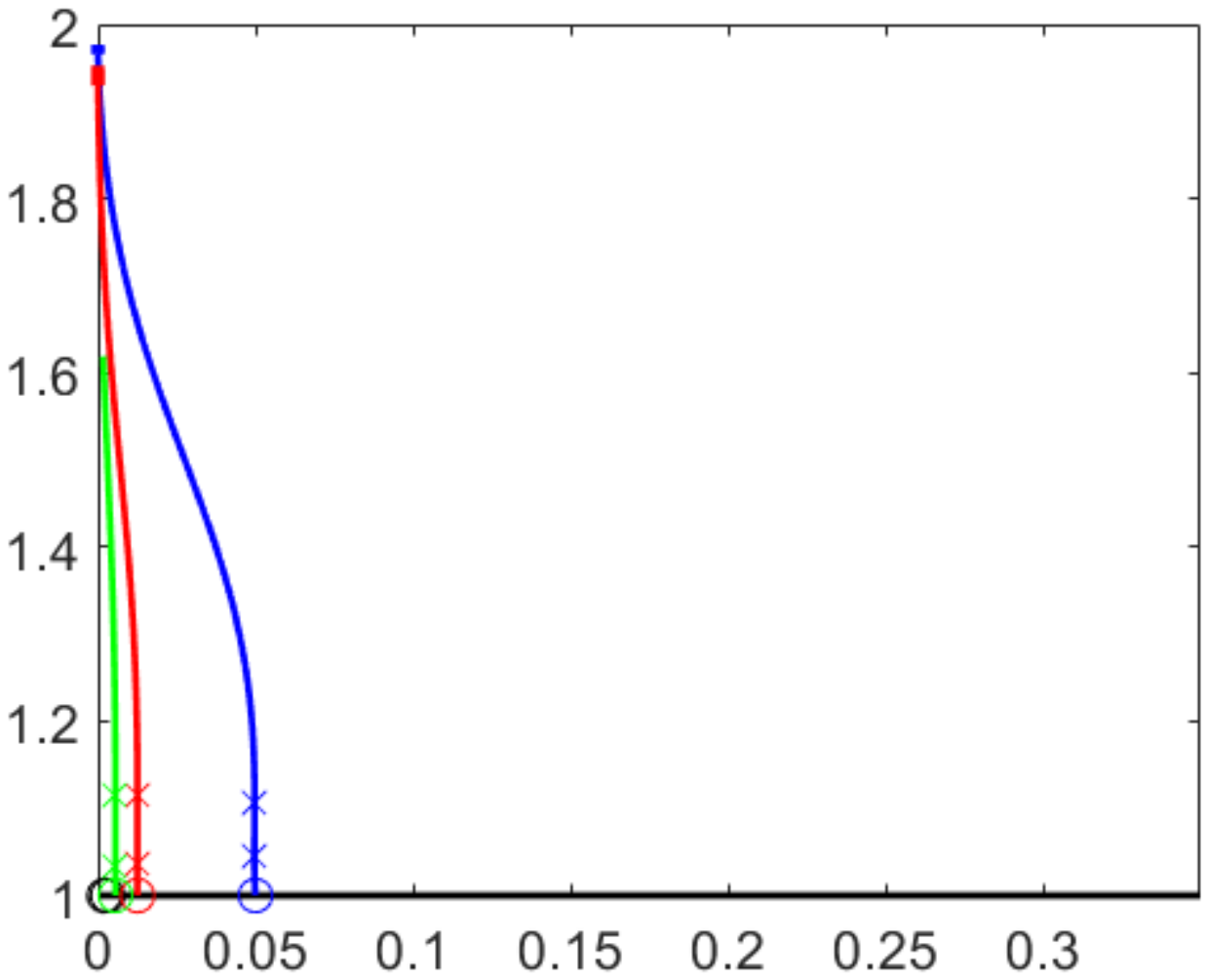}
\put(0,30){\rotatebox{90}{$||u||_{L_2}$}}
\put(85,-2){$d$}
\end{overpic} 
}
\subfloat[\label{strong_d21_20}$d_{21}=20$]{
\begin{overpic}[width=0.5\textwidth,tics=10,trim={3cm 8.5cm 3cm 9cm},clip]{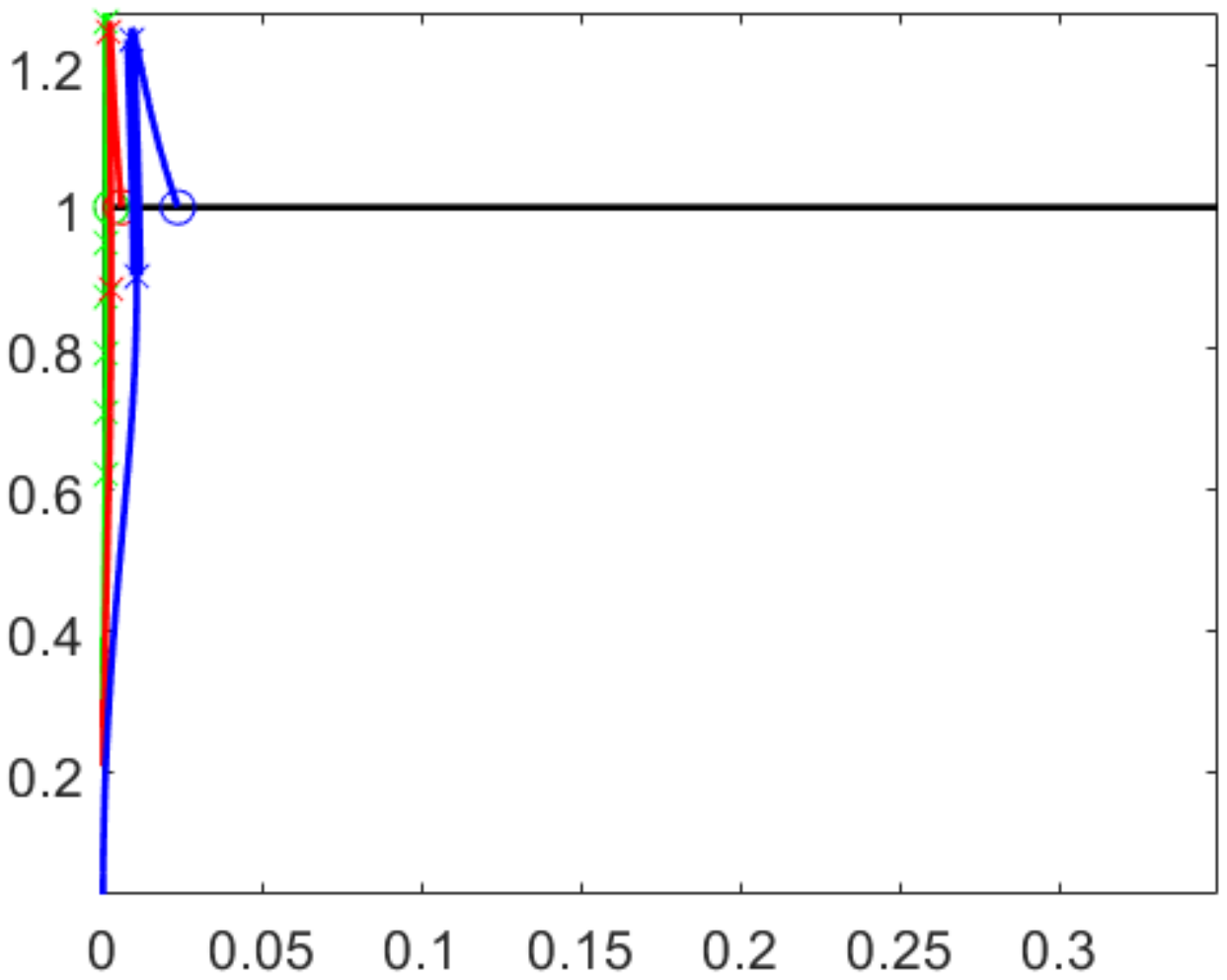}
\put(0,30){\rotatebox{90}{$||u||_{L_2}$}}
\put(85,-2){$d$}
\end{overpic} 
}\\
\subfloat[\label{strong_d21_100}$d_{21}=100$]{
\begin{overpic}[width=0.5\textwidth,tics=10,trim={3cm 8.5cm 3cm 9cm},clip]{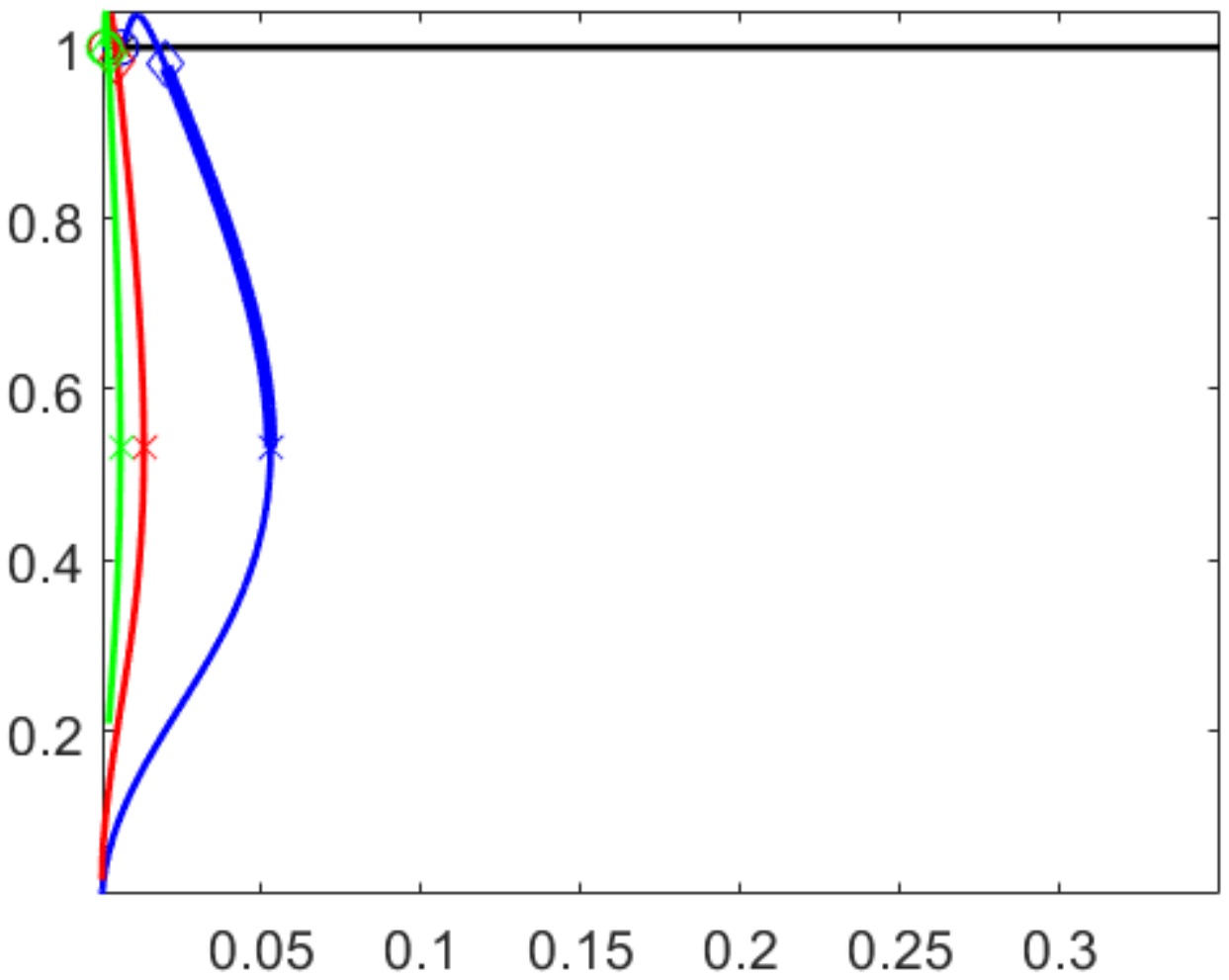}
\put(0,30){\rotatebox{90}{$||u||_{L_2}$}}
\put(85,-2){$d$}
\end{overpic} 
}
\subfloat[\label{strong_d21_1000}$d_{21}=1000$]{
\begin{overpic}[width=0.5\textwidth,tics=10,trim={3cm 8.5cm 3cm 9cm},clip]{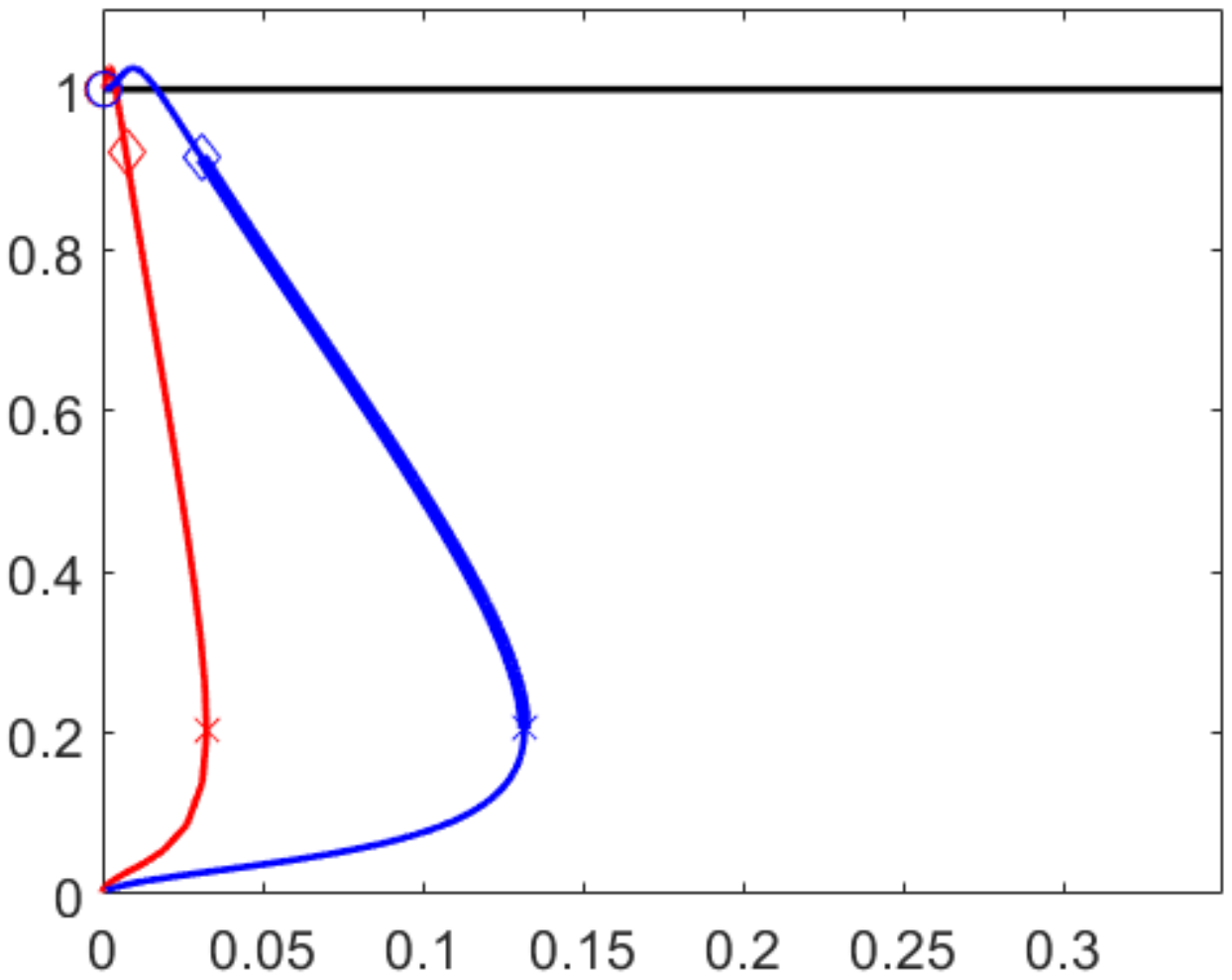}
\put(0,30){\rotatebox{90}{$||u||_{L_2}$}}
\put(85,-2){$d$}
\end{overpic} 
}
\caption{Bifurcation diagrams for different values of the cross-diffusion parameter $d_{21}$ in the strong-competition case (second parameter set in Table \ref{tab:param}, with $d_{12} =3$).
}
\label{strong_BD_d21}
\end{figure}

\begin{figure}[!ht]
\centering
\begin{overpic}[width=0.6\textwidth,tics=10,trim={3cm 8cm 3cm 8cm},clip]{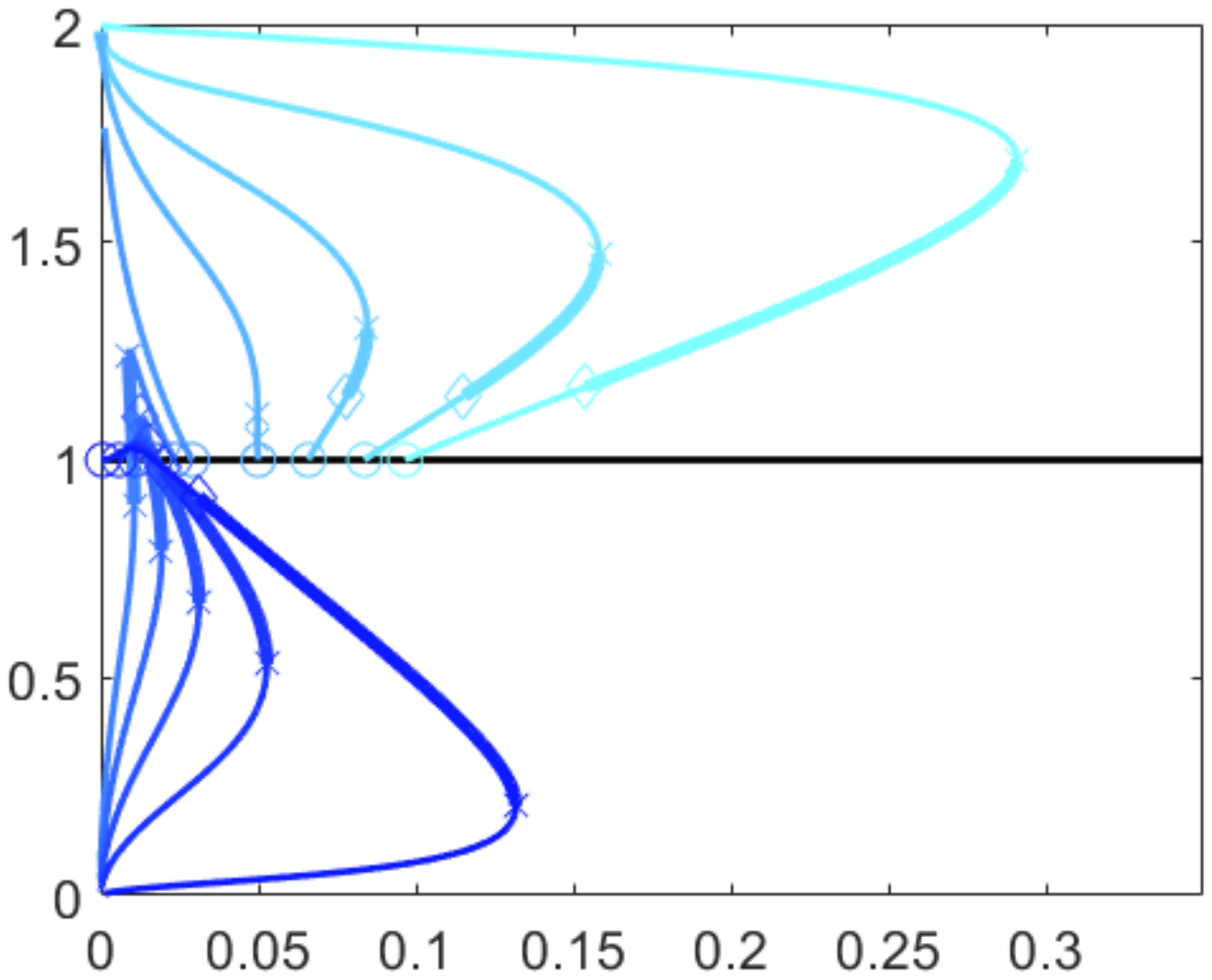}
\put(0,30){\rotatebox{90}{$||u||_{L_2}$}}
\put(85,4){$d$}
\end{overpic} 
\caption{Evolution of the first bifurcation branch for increasing values of $d_{21}\in [0,1000]$ (darker the blue, greater the value), corresponding to the second parameter set in Table \ref{tab:param}, with $d_{12}=3$.}
\label{strong_mimura_d12increasing}
\end{figure}

%%%%%%%%%%%%%%%%%%%%%%%%%%%%%%%%%%%%%%%%%%%%%%%%%%%%%%%%%%%

\subsection{Third parameter set in Table~\ref{tab:param}: existence of non homogeneous solutions in the weak competition case when both cross-diffusion parameters are large}
	
The third parameter set corresponds to $\alpha+\beta>0$ and it illustrates Theorem~\ref{th:large_cd}, showing that non-homogeneous steady states exist when both cross-diffusion coefficients are large and qualitatively similar. Figure \ref{AlphaPlusBeta} shows the bifurcation diagram with large cross-diffusion coefficients ($d_{\textnormal{cross}}=d_{12}=d_{21}=100$) and some solutions. In particular, stable non-homogeneous solutions arise on the first branch.

\begin{figure}[!ht]
\centering
\subfloat[\label{alphabeta_BD}]{
\begin{overpic}[width=0.6\textwidth,tics=10,trim={3cm 8cm 3cm 8cm},clip]{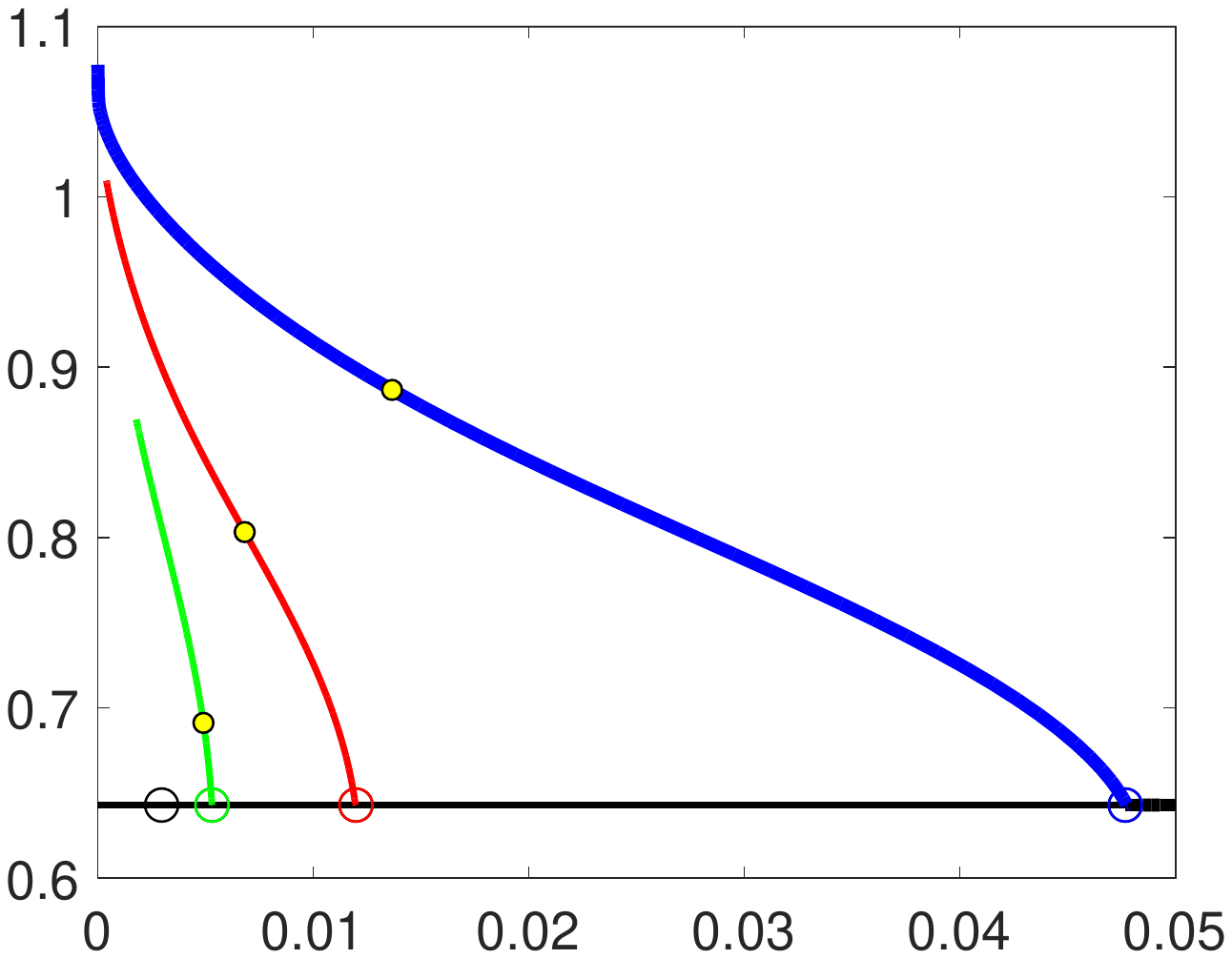}
\put(0,30){\rotatebox{90}{$||u||_{L_2}$}}
\put(80,2){$d$}
\put(34,46){$(b)$}
\put(26,36){$(c)$}
\put(23,22){$(d)$}
\end{overpic} 
}\\
\hspace{-0.7cm}
\subfloat[\label{alphabeta_sol1}]{
\begin{overpic}[width=0.35\textwidth,tics=10,trim={3cm 9cm 3cm 9cm},clip]{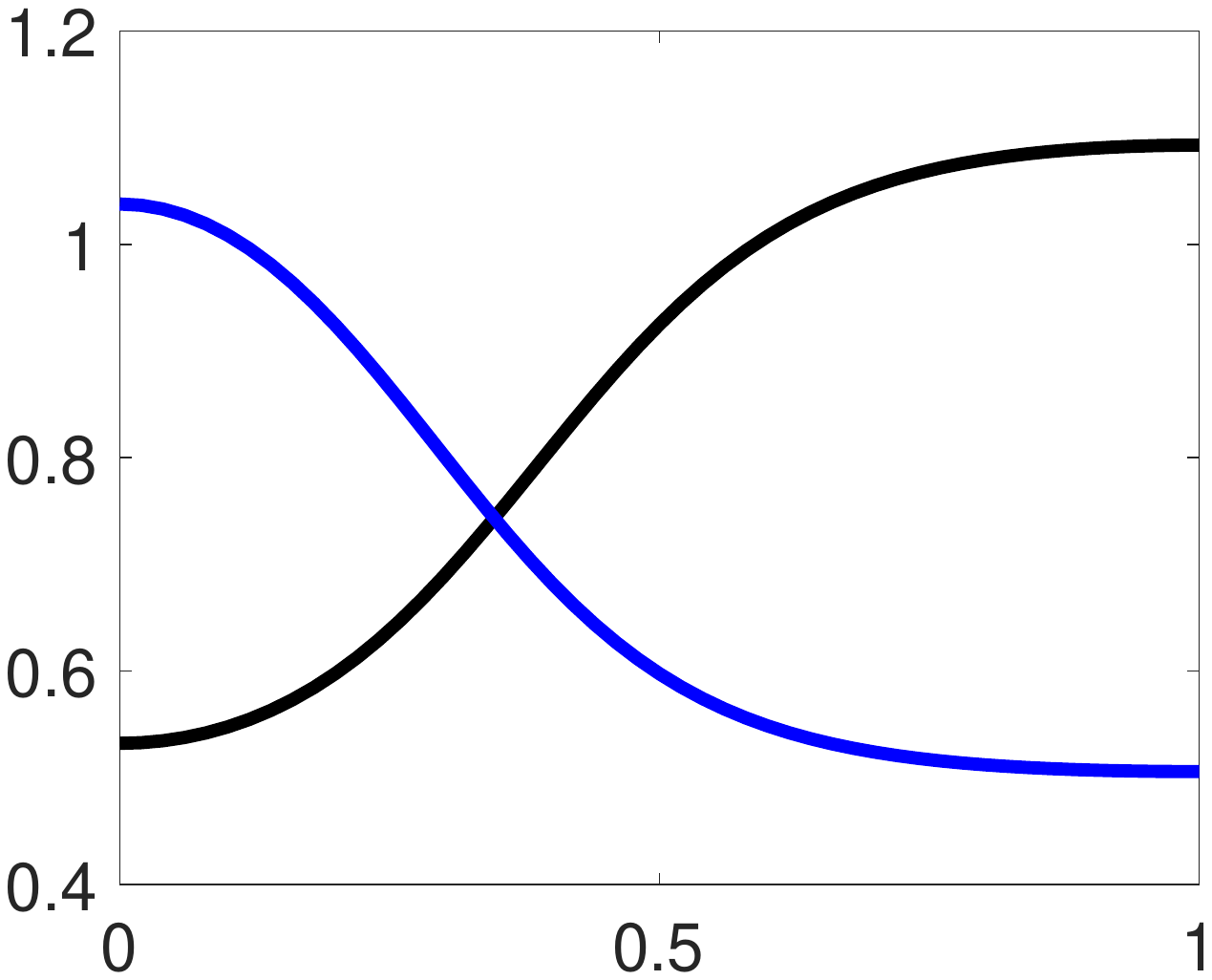}
\put(0,30){\rotatebox{90}{$u,v$}}
\put(75,0){$x$}
\end{overpic} 
}
\hspace{-0.7cm}
\subfloat[\label{alphabeta_sol2}]{
\begin{overpic}[width=0.35\textwidth,tics=10,trim={3cm 9cm 3cm 9cm},clip]{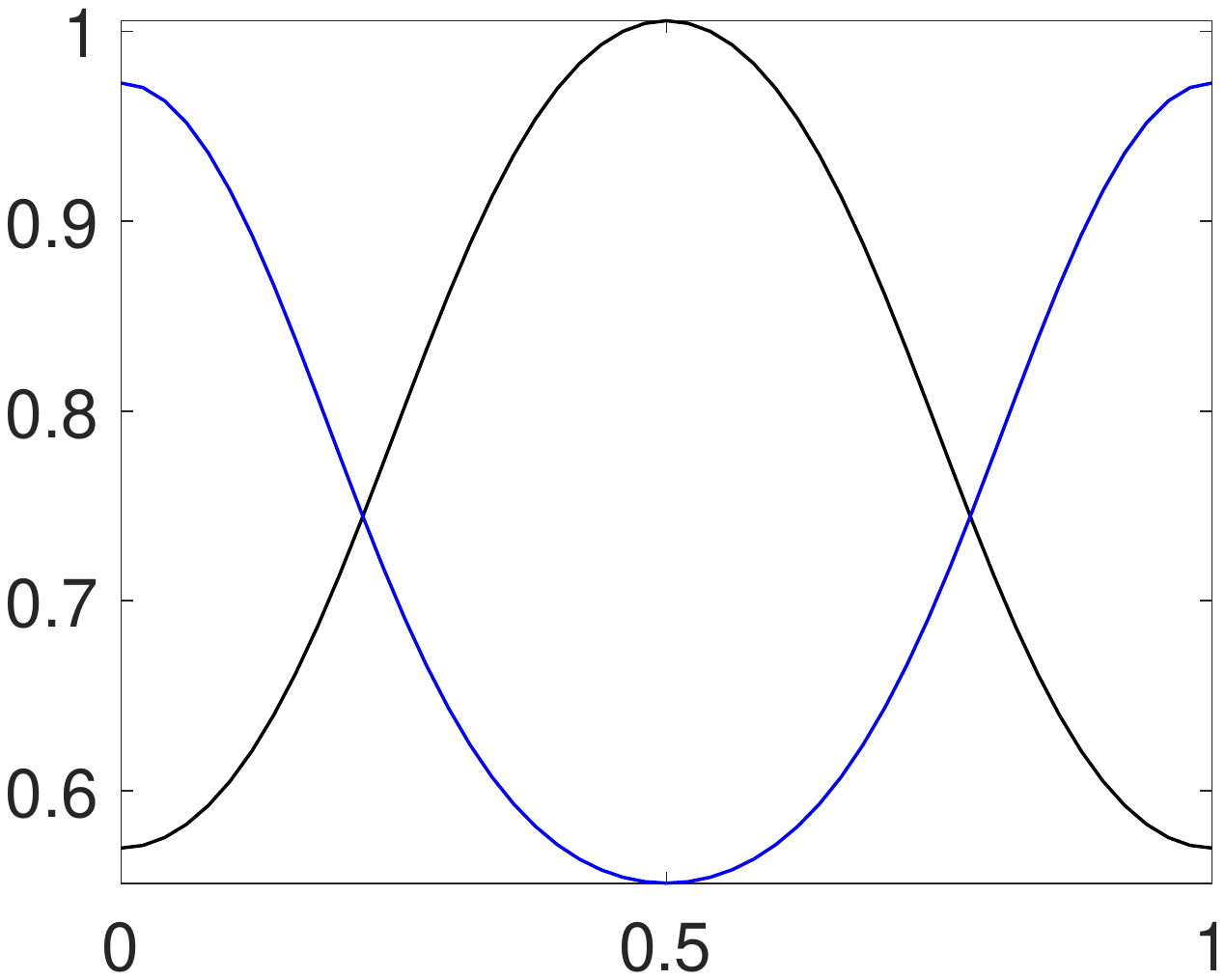}
\put(0,30){\rotatebox{90}{$u,v$}}
\put(75,0){$x$}
\end{overpic} 
}
\hspace{-0.7cm}
\subfloat[\label{alphabeta_sol3}]{
\centering
\begin{overpic}[width=0.35\textwidth,tics=10,trim={3cm 9cm 3cm 9cm},clip]{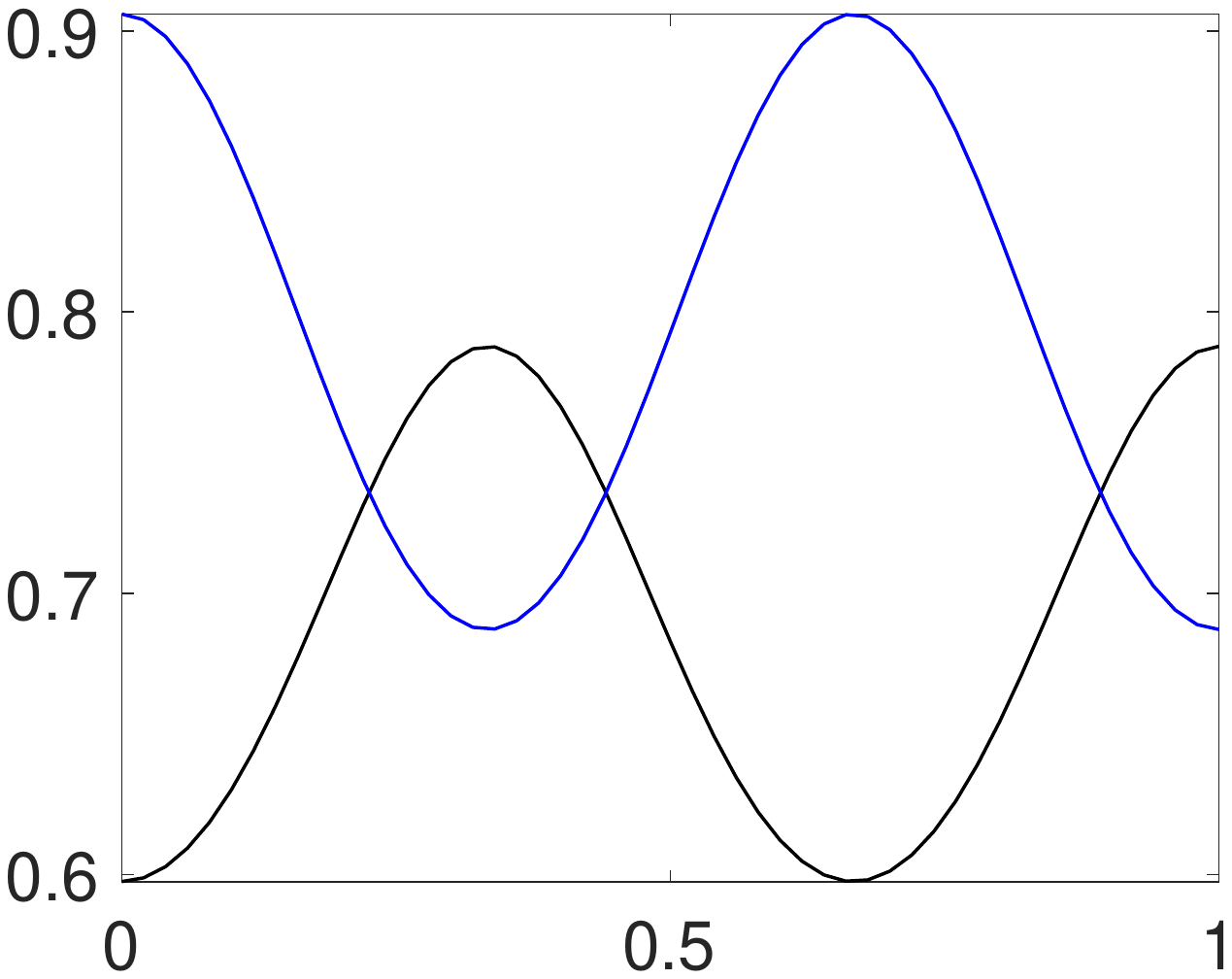}
\put(0,30){\rotatebox{90}{$u,v$}}
\put(75,0){$x$}
\end{overpic} 
}
\caption{Bifurcation diagram and some solutions (third parameter set in Table \ref{tab:param} and $d_{12}~=~d_{21}=100$). (a) Bifurcation diagram with respect to the parameter $d$. Yellow points indicate the positions on the branches at which solutions are shown. (b)--(d) Solutions $u(x),\;v(x)$ for different values of $d$ (black lines correspond to species $u$, while blue lines to species $v$). %The red and the green branches are not complete because the numerical continuation fails as $d$ becomes small.
}
\label{AlphaPlusBeta}
\end{figure}

%%%%%%%%%%%%%%%%%%%%%%%%%%%%%%%%%%%%%%%%%%%%%%%%%%%%%%%%%%%

\subsection{Fourth parameter set in Table~\ref{tab:param}: changes of stability in the strong competition regime}

The fourth parameter set in Table \ref{tab:param} corresponds to the strong competition case, with $\beta>0$ in order to select a more generic case than the one considered in~\cite{izuhara2008reaction}. Note that in the strong competition case, there is already a full bifurcation diagram without cross-diffusion, but it does not contain any stable solutions. Increasing the cross-diffusion coefficient $d_{12}$ (meaning that we put cross-diffusion in the system), we can observe when/how stable solutions appear.  In Figure \ref{BD_4_d12increasing} the evolution of the first branch  for increasing values of $d_{12}\in [0,3]$ and $d_{21}=0$ is shown: the lightest blue corresponds to $d_{12}=0$, the darkest to $d_{12}=3$, while as usual the homogeneous solution is denoted in black. We observe that the first bifurcation point on the homogeneous branch decreases as $d_{12}$ is increasing, while the branch tends to fold. In particular, looking at stable solutions, we can see that the system needs enough cross-diffusion to admit stable non-homogeneous solutions. In addition, these regions arises first between two successive fold points, and then between a Hopf and a fold. The same considerations apply to the other branches.

\begin{figure}[!ht]
\centering
\begin{overpic}[width=0.6\textwidth,tics=10,trim={3cm 8cm 3cm 8cm},clip]{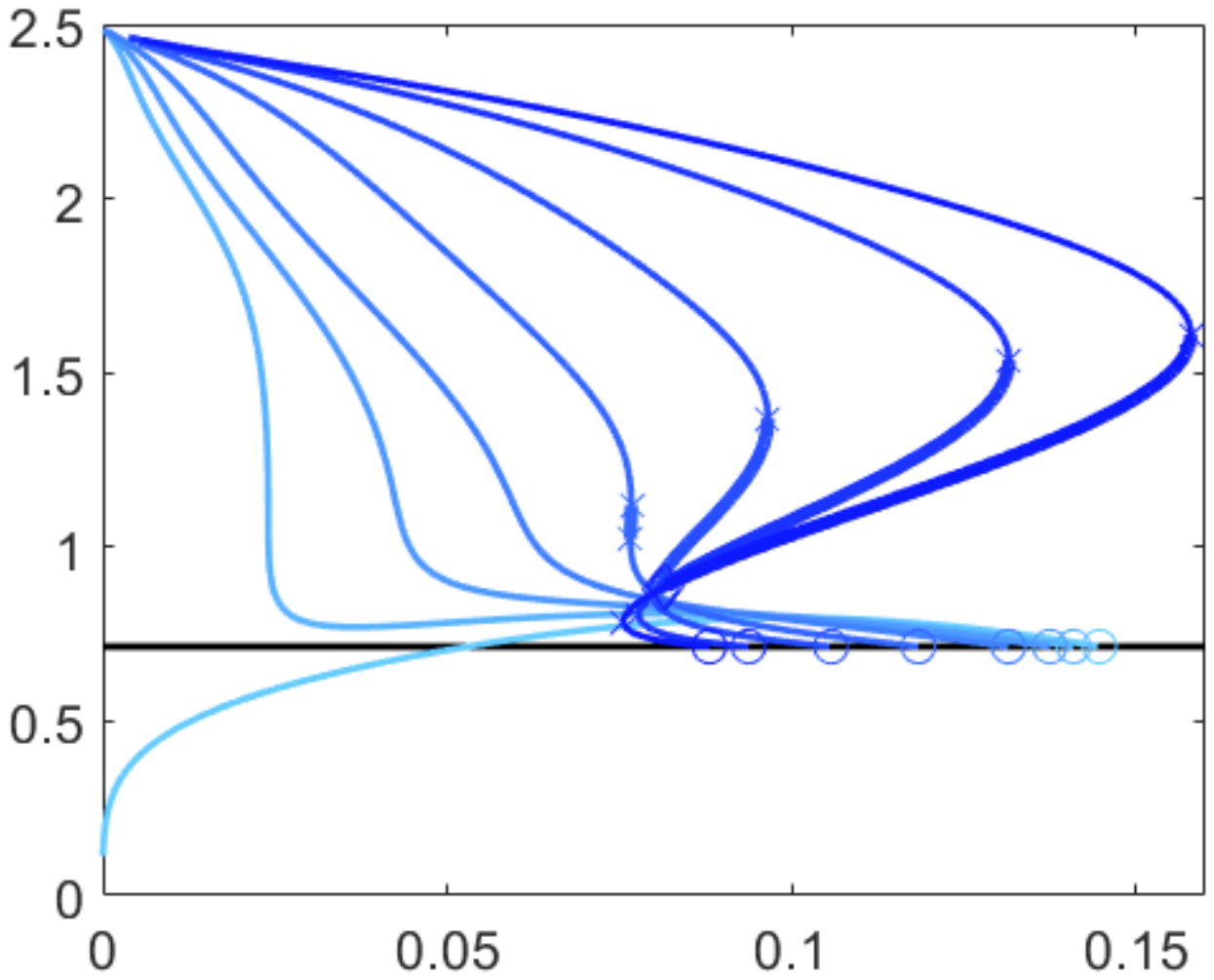}
\put(0,30){\rotatebox{90}{$||u||_{L_2}$}}
\put(90,10){$d$}
\end{overpic} 
\caption{Evolution of the first bifurcation branch for increasing values of $d_{12}\in [0,3]$ (darker the blue, greater the value), corresponding to the fourth parameter set in Table \ref{tab:param}, with $d_{21}=0$.}
\label{BD_4_d12increasing}
\end{figure}

\section{Beyond the usual weak and strong competition regime}\label{sec:r1}

The reason why~\eqref{weak_comp} and~\eqref{strong_comp} are sometimes referred to as the \textit{strong intra-specific competition} regime and the \textit{strong inter-specific competition} regime, is because they respectively imply $b_1b_2<a_1a_2$ and $a_1a_2<b_1b_2$ (we recall that the $a_i$ represent the intra-specific competition rates, whereas the $b_i$ represent the inter-specific competition rates). These conditions on the competition rates are usually supplemented with conditions on the growth rates $r_i$ (as those in~\eqref{weak_comp} and~\eqref{strong_comp}), ensuring that the homogeneous steady state $(u_*,v_*)$ is positive, and thus that it can be used as a starting point for the analysis. 

However, we believe that these conditions on the growth rates are not necessarily meaningful when spatial inhomogeneities are accounted for in the model. Indeed, we report here some numerical experiments suggesting that, when cross-diffusion is taken into account, there exist positive and stable non-homogeneous steady states of~\eqref{cross} outside of the range of parameters for which $(u_*,v_*)$ is positive.

We present in Figure \ref{BD_r1} bifurcation diagrams with respect to the parameter $r_1$ (the growth rate of species $u$), in the weak and strong competition cases, where positive non-homogeneous steady states appear outside of the range of parameters where $(u_*,v_*)$ is admissible. They are obtained in the triangular case ($d_{21}=0$), using the first and second parameter sets reported in Table \ref{tab:param} except for $r_1$ which is varied, and fixing the value of the diffusion coefficient ($d=0.005$ in the weak competition case and $d=0.05$ in the strong competition case).

In Figure \ref{BDr1_weak}, there are stable non-homogeneous solutions when $r_1>6$, that is outside of the range of parameters for which $(u_*,v_*)$ is positive; recall that we are in the usual weak competition regime~\eqref{weak_comp} if and only if $r_1$ belongs to $(2/3,6)$. Moreover those solutions, reported in Figures \ref{r1_weak_sol_red}, \ref{r1_weak_sol_green} for $r_1=7.5$ and in Figure \ref{r1_weak_sol_green_15} for $r_1=15$, are qualitative similar to the others already obtained in the previous sections. In Figure \ref{BDr1_strong},  we are in the usual strong competition regime~\eqref{weak_comp} if and only if $r_1$ belongs to $(2/3,5/2)$. We can see that also in this case the system admits stable non-homogeneous steady states (up to four) where the homogeneous steady state $(u_*,v_*)$ is no longer meaningful. Note also that the red branch, corresponding to the second (red) branch in Figure \ref{v0-L2}, can become stable. Stable solutions are shown in Figures \ref{r1_strong_b1sol}, \ref{r1_strong_b2sol} with $r_1=10$ and Figure \ref{r1_strong_b1sol_18} for $r_1=18$.

\begin{figure}[!ht]
\centering
\subfloat[\label{BDr1_weak}]{
\begin{overpic}[width=0.5\textwidth,tics=10,trim={3cm 8cm 3cm 8cm},clip]{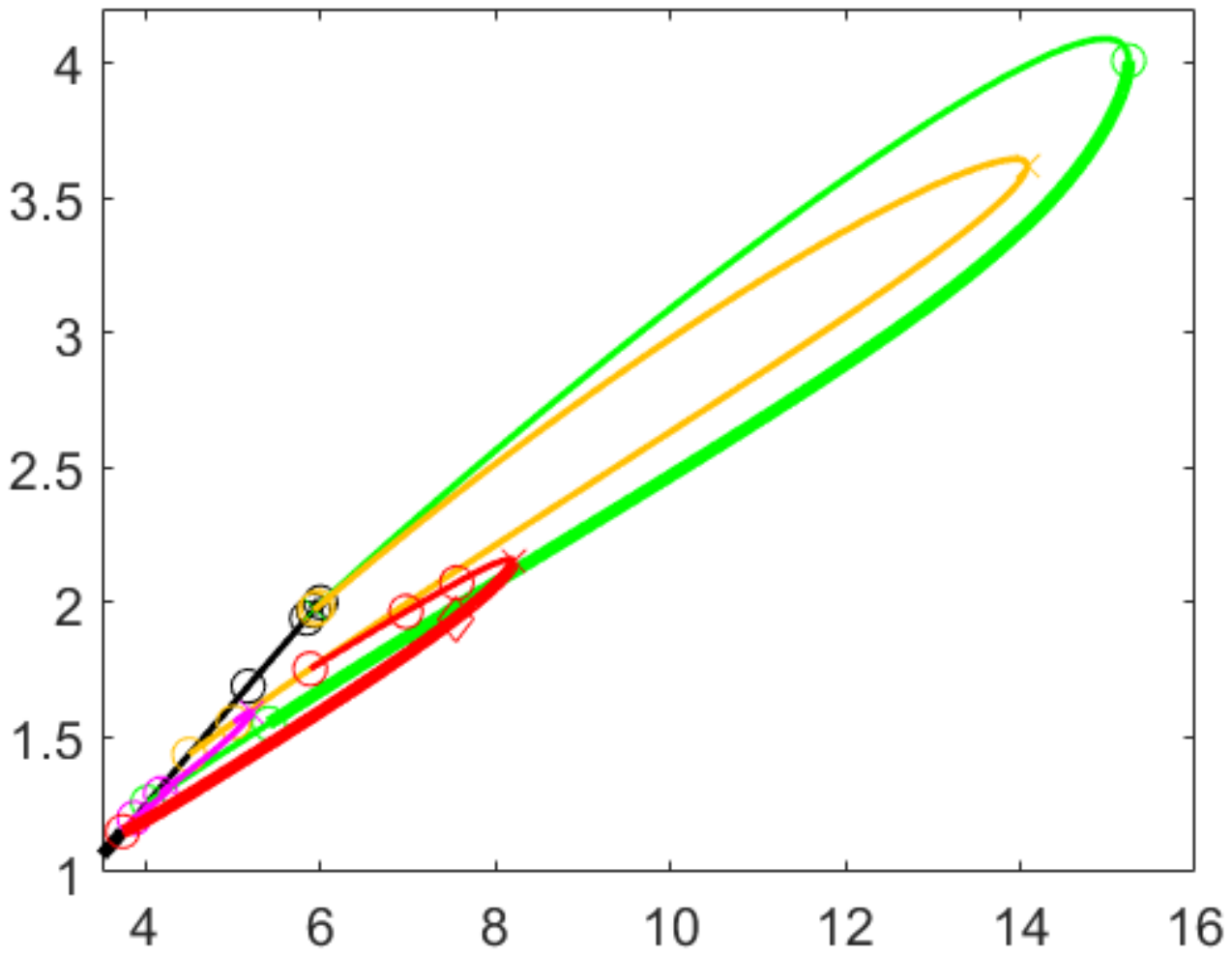}
\put(80,0){$r_1$}
\put(0,30){\rotatebox{90}{$||u||_{L_2}$}}
\end{overpic} 
}
\subfloat[\label{BDr1_strong}]{
\begin{overpic}[width=0.5\textwidth,tics=10,trim={3cm 8cm 3cm 8cm},clip]{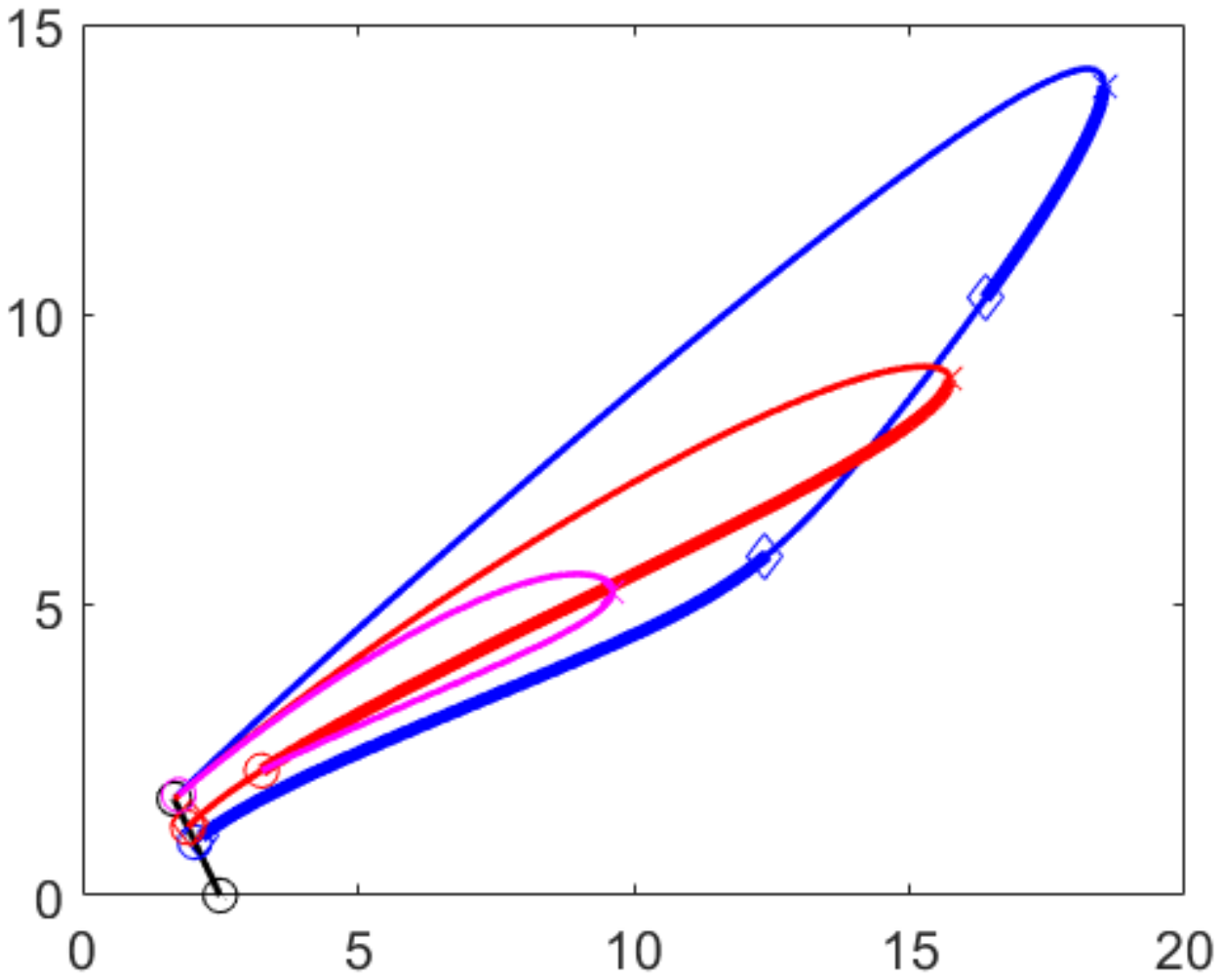}
\put(80,0){$r_1$}
\put(2,30){\rotatebox{90}{$||u||_{L_2}$}}
\end{overpic} 
}

\caption{Bifurcation diagrams with bifurcation parameter $r_1$. (a) Weak competition case with $d=0.005$ and the other parameter values as in the first parameter set in Table \ref{tab:param}, with $d_{12}=3$ and $d_{21}=0$. (b) Strong competition case with $d=0.05$ and the other parameter values as in the second parameter set in Table \ref{tab:param}, with $d_{12}=3$ and $d_{21}=0$.}
\label{BD_r1}
\end{figure}

\begin{figure}[!ht]
\centering
\subfloat[\label{r1_weak_sol_red} $r_1=7.5$, red branch]{
\begin{overpic}[width=0.4\textwidth,tics=10,trim={3cm 8cm 3cm 8cm},clip]{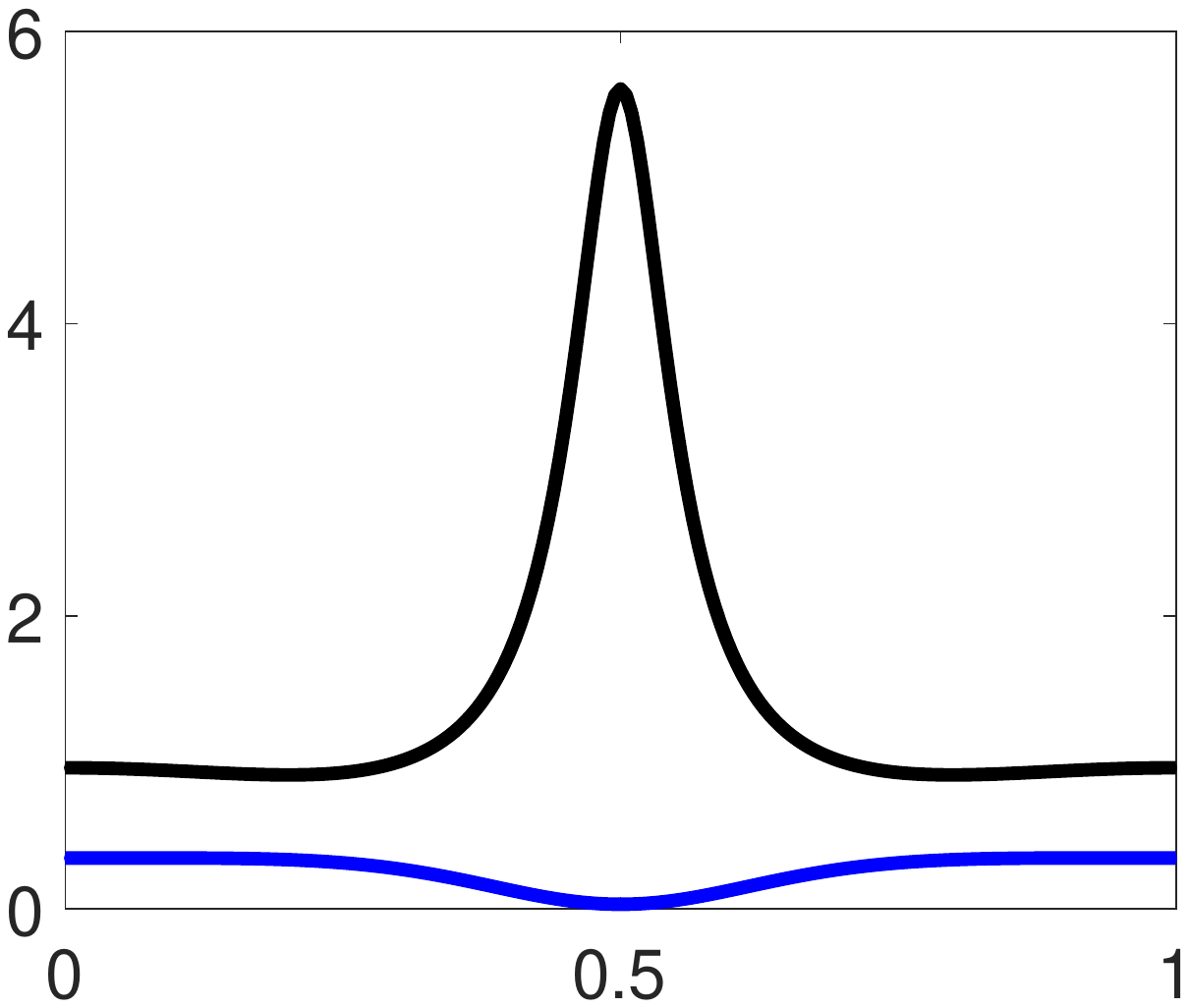}
\put(75,0){$x$}
\put(0,30){\rotatebox{90}{$u,\;v$}}
\end{overpic} 
}
\subfloat[\label{r1_strong_b1sol} $r_1=10$, blue branch]{
\begin{overpic}[width=0.4\textwidth,tics=10,trim={3cm 8cm 3cm 8cm},clip]{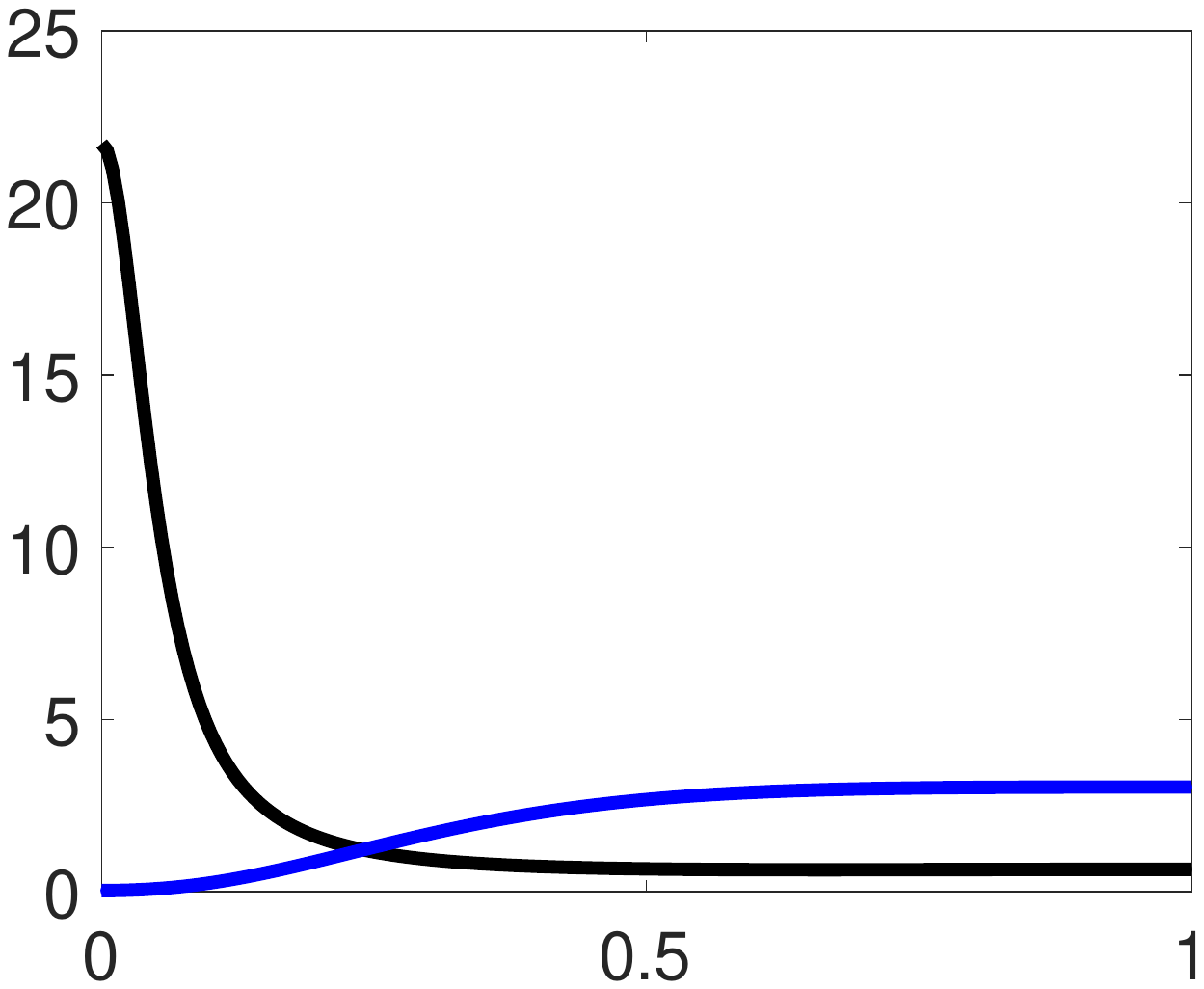}
\put(75,0){$x$}
\put(2,30){\rotatebox{90}{$u,\;v$}}
\end{overpic} 
}\\

\subfloat[\label{r1_weak_sol_green}$r_1=7.5$, green branch]{
\begin{overpic}[width=0.4\textwidth,tics=10,trim={3cm 8cm 3cm 8cm},clip]{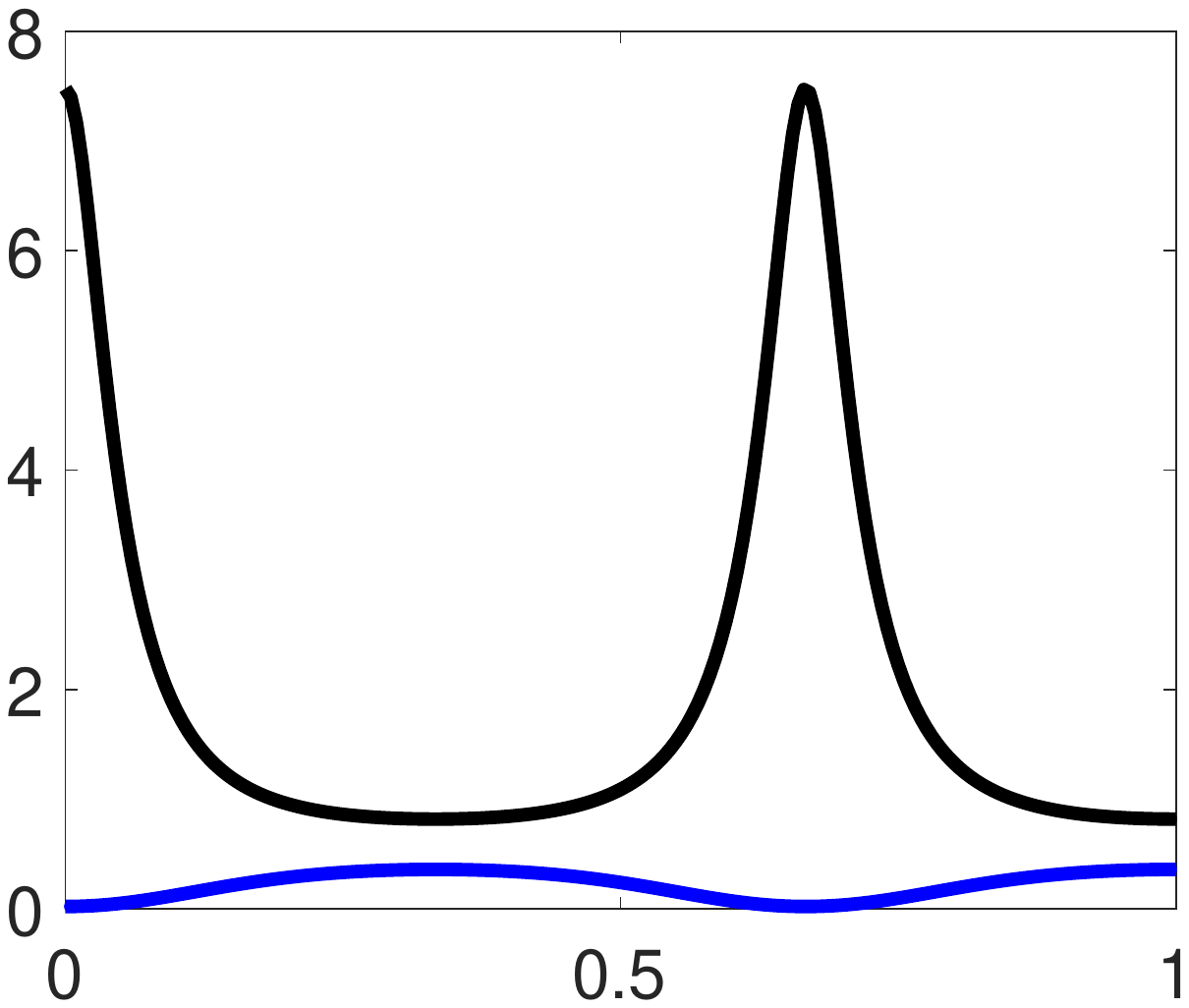}
\put(75,0){$x$}
\put(0,30){\rotatebox{90}{$u,\;v$}}
\end{overpic} 
}
\subfloat[\label{r1_strong_b2sol}$r_1=10$, red branch]{
\begin{overpic}[width=0.4\textwidth,tics=10,trim={3cm 8cm 3cm 8cm},clip]{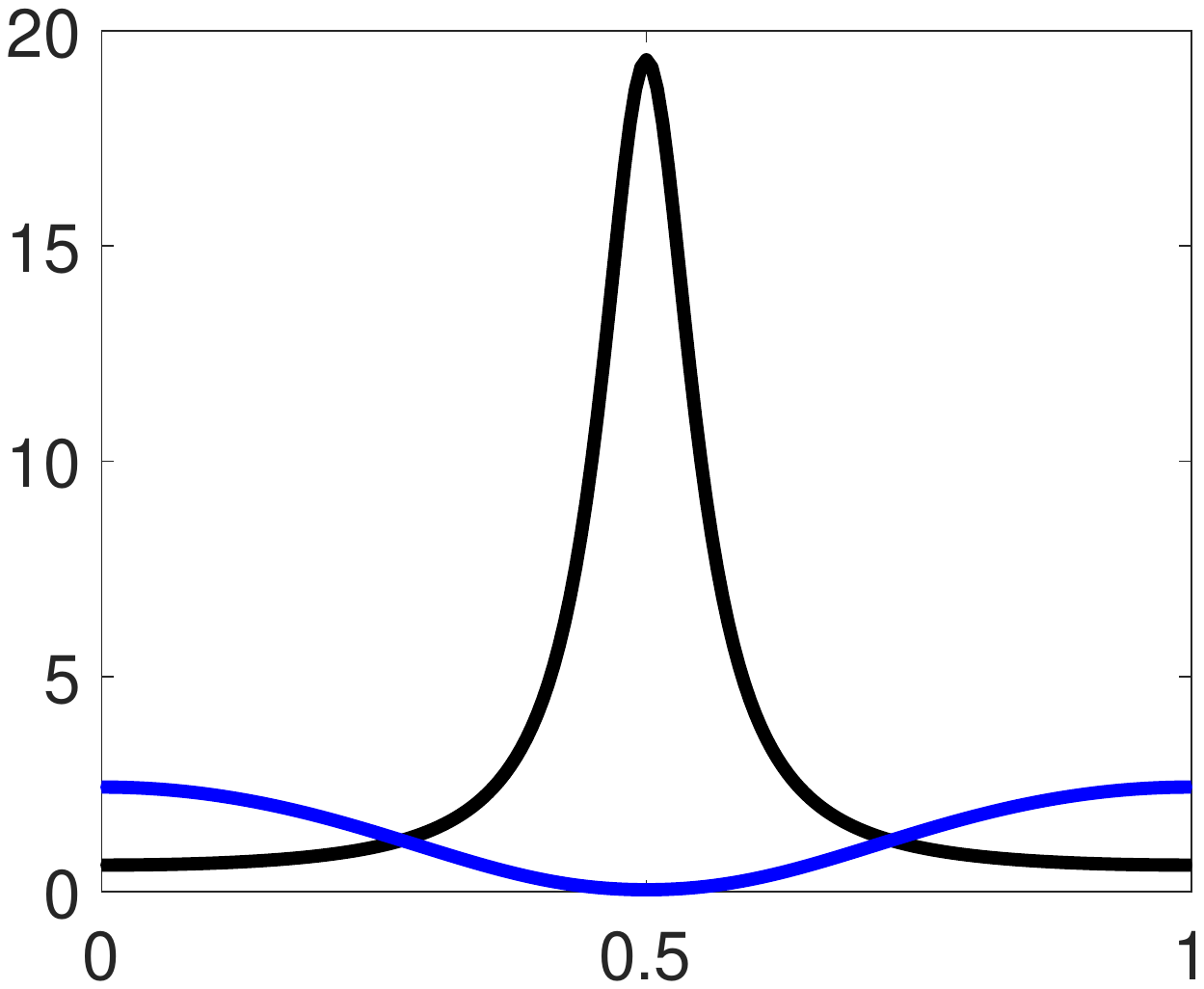}
\put(75,0){$x$}
\put(2,30){\rotatebox{90}{$u,\;v$}}
\end{overpic} 
}

\subfloat[\label{r1_weak_sol_green_15}$r_1=15$, green branch]{
\begin{overpic}[width=0.4\textwidth,tics=10,trim={3cm 8cm 3cm 8cm},clip]{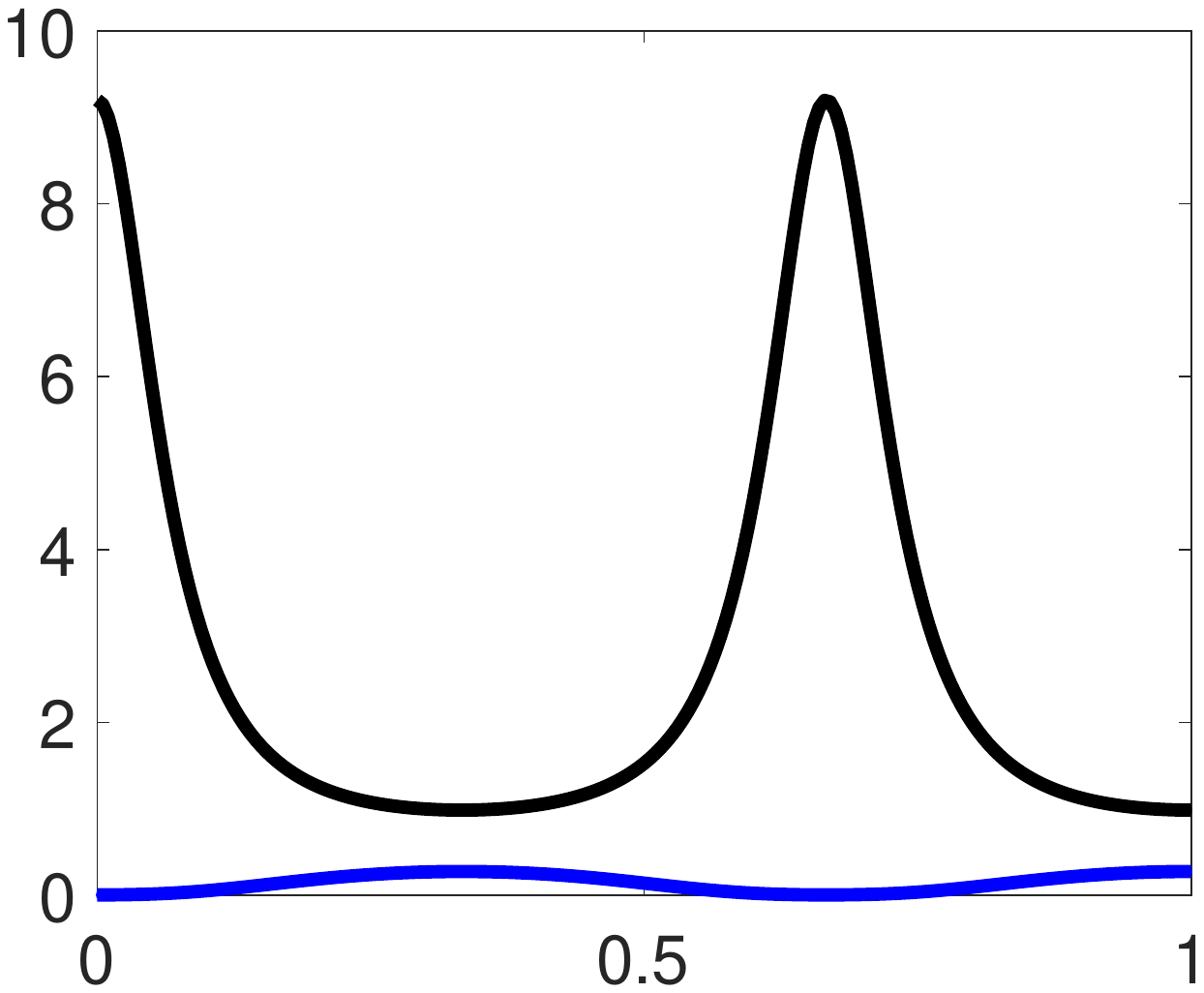}
\put(75,0){$x$}
\put(0,30){\rotatebox{90}{$u,\;v$}}
\end{overpic} 
}
\subfloat[\label{r1_strong_b1sol_18}$r_1=18$, blue branch]{
\begin{overpic}[width=0.4\textwidth,tics=10,trim={3cm 8cm 3cm 8cm},clip]{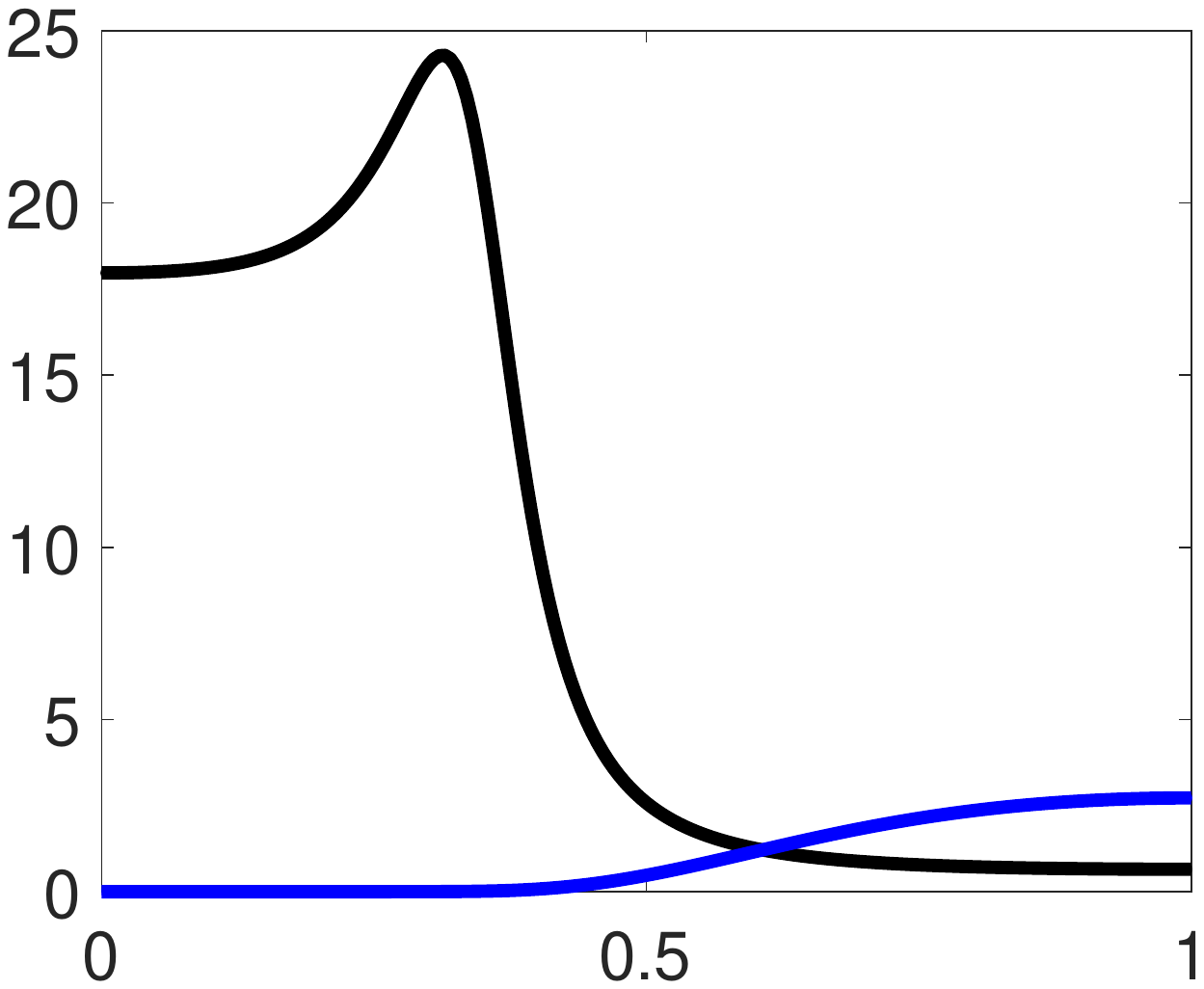}
\put(75,0){$x$}
\put(2,30){\rotatebox{90}{$u,\;v$}}
\end{overpic} 
}

\caption{Stable non-homogeneous solutions appearing beyond the usual weak or strong competitions regimes (species $u$ and $v$ are denoted with black and blue lines respectively). Weak competition case ($d=0.005$, the other parameter values as in the first parameter set in Table \ref{tab:param}, with $d_{12}=3$ and $d_{21}=0$): (a), (c) stable solutions with $r_1=7.5$, (e) stable solution with $r_1=15$. Strong competition case ($d=0.05$, the other parameter values as in the second parameter set in Table \ref{tab:param}, with $d_{12}=3$ and $d_{21}=0$): (b), (d) stable solutions with $r_1=10$, (e) stable solution with $r_1=18$. The colors of the branches refer to Figure~\ref{BD_r1}.}
\label{BD_r1_sol}
\end{figure}

\section{Changes induced by self-diffusion}
\label{sec:with_selfdiff}

In this section, we briefly discuss the changes that are induced by taking non-zero self-diffusion coefficients $d_{11}$ and $d_{22}$. Taking them into account in the computations made in Section~\ref{sec:without_selfdiff}, the characteristic matrix becomes
$$\tilde M_k^*=J_*-\tilde J_\Delta^*\lambda_k=
\begin{pmatrix}
-a_1u_*-(d+d_{12}v_*+2d_{11}u_*)\lambda_k   &  -b_1u_*-d_{12}u_*\lambda_k\\
-b_2v_* -d_{21}v_*\lambda_k         & -a_2v_*-(d+d_{21}u_*+2d_{22}v_*)\lambda_k
\end{pmatrix}.$$ 
It still has a negative trace, and its determinant is
\begin{equation}
\tilde P_k(d):= \det \tilde M_k = \tilde A_k d^2 + \tilde B_k d + \tilde C_k,
\label{eq:det_self}
\end{equation}
where
\begin{align*}
\tilde A_k = \lambda_k^2, \quad \tilde B_k = 2(d_{11}u_*+d_{22}v_*)\lambda_k^2+d_{12}v_*\lambda_k^2+d_{21}u_*\lambda_k^2 -\tr J_* \lambda_k,
\end{align*}
and
\begin{equation*}
\tilde C_k = -d_{12}(\alpha -2d_{22}v_*^2\lambda_k)\lambda_k- d_{21}(\beta -2d_{11}u_*^2\lambda_k)\lambda_k+2u_*v_*\lambda_k((d_{11}a_2+d_{22}a_1)+2d_{11}d_{22}\lambda_k)+\det J_*.
\end{equation*}
Again, we get a bifurcation for the $k$-th mode if and only if $\tilde C_k<0$. Therefore, adding self-diffusion can only hinder the appearance of non-homogeneous steady states, because $\tilde C_k$ increases with $d_{11}$ and $d_{22}$. Even for small (but non zero) values of the self-diffusion coefficients $d_{11}$ and $d_{22}$, we see a main difference with some of the situations described in Section~\ref{sec:without_selfdiff}, namely that we can only have a finite number of bifurcations from $(u_*,v_*)$. Indeed, as soon as 
\begin{equation}\label{eq:self_bound}
\lambda_k \geq \max\left(\frac{\beta}{2d_{11}u_*^2},\frac{\alpha}{2d_{22}v_*^2}\right),
\end{equation}
there is no bifurcation associated to the $k$-th mode. 

The \texttt{pde2path} setup requires slight and straightforward modifications (see \cite{CKCS} for the code setup) to implement self-diffusion terms. In Figure \ref{BD_self} we show how the bifurcation structure behaves when one self-diffusion coefficient increases. As in the previous sections, the first, second and third branches are denoted in blue, red and green, respectively. As in Figure \ref{v0-L2}, the magenta and the orange branches correspond to secondary bifurcations.
We consider the first parameter set of Table \ref{tab:param}, with $d_{12}=3$ and $d_{21}=0$, fix $d_{11}=0$ and vary $d_{22}$. The bifurcation diagram corresponding to $d_{22}=0$ is obviously the usual one, already shown in Figure \ref{v0-L2}, while with non-zero values the successive disappearance of primary bifurcations points is confirmed. Figure \ref{BD_self_1} is obtained with $d_{22}=0.03$: the bifurcation diagram seems shifted to the left and stretched, and only four primary bifurcation points are detected for positive values of $d$. When $d_{22}=0.05$ (Figure \ref{BD_self_1}) we observe only two primary bifurcation points, and on the first (blue) branch no unstable regions have been detected. Finally in Figure \ref{BD_self_2} the second bifurcation point is no longer present. 

Moreover, in this case the numerical computation of the branches does not present criticalities approaching the value $d=0$, but instead the bifurcation diagram exists for negative value of the standard diffusion coefficients and solutions are still meaningful. Looking at system \eqref{cross} with $d_{11}=d_{21}=0$ and $d_{12},\;d_{22}>0$, negative values of $d$ can still yield an elliptic operator. In general, considering expressions \eqref{eq:det_self} and \eqref{detM*d} of the determinants, we have that 
\begin{equation*}
\tilde{P}_k(d)=P_k(d-d_s)+p_s,
\end{equation*}
where $d_s$ and $p_s$ depend on the parameters and on $\lambda_k$. 

\begin{figure}[!ht]
\centering
\subfloat[\label{BD_self_1}$d_{11}=0,\; d_{22}=0.03$]{
\begin{overpic}[width=0.5\textwidth,tics=10,trim={3cm 8.5cm 3cm 9cm},clip]{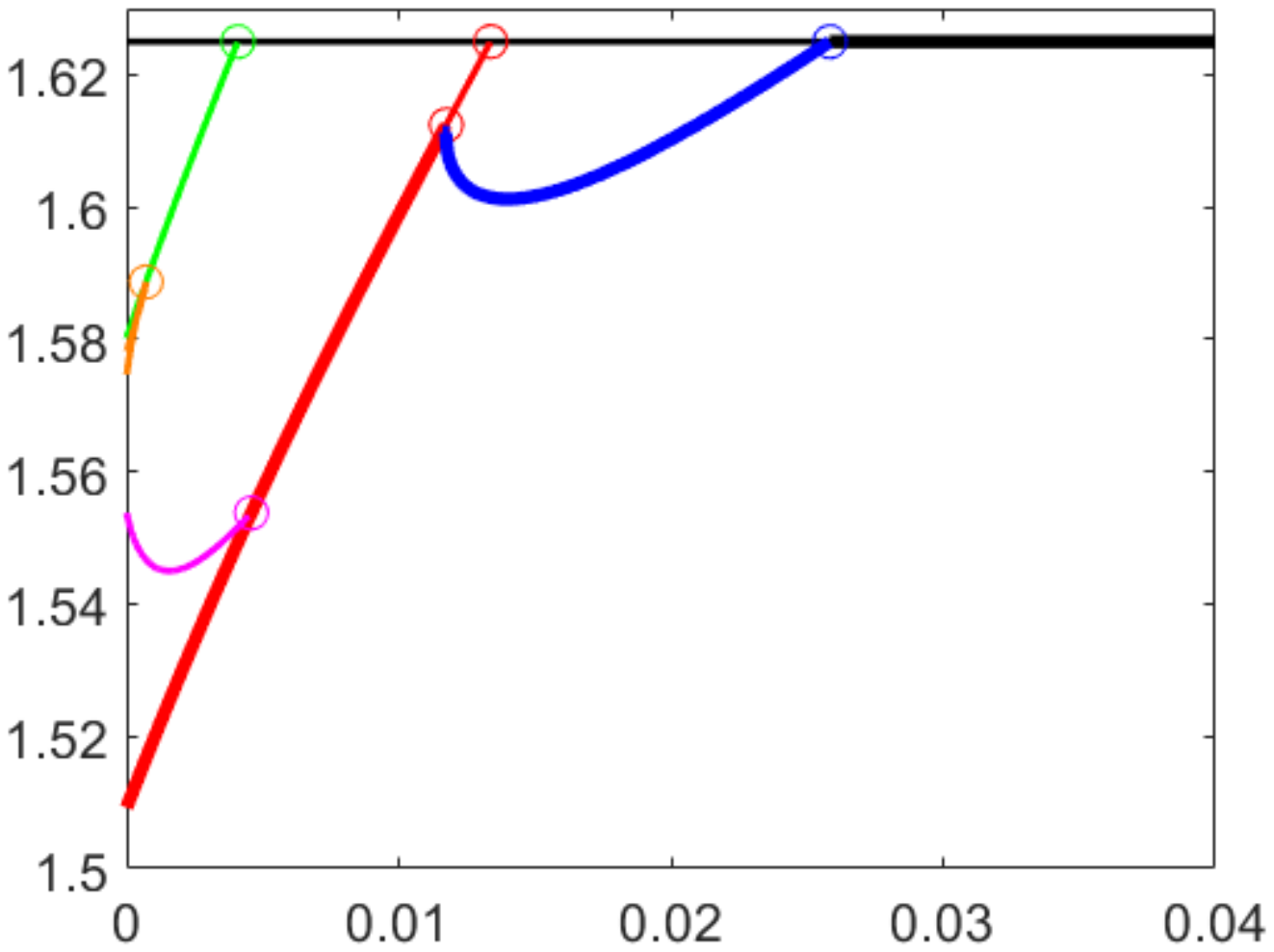}
\put(90,10){$d$}
\put(0,28){\rotatebox{90}{$||u||_{L_2}$}}
\end{overpic} 
}\\
\subfloat[\label{BD_self_2}$d_{11}=0,\; d_{22}=0.05$]{
\begin{overpic}[width=0.5\textwidth,tics=10,trim={3cm  8.5cm 3cm 9cm},clip]{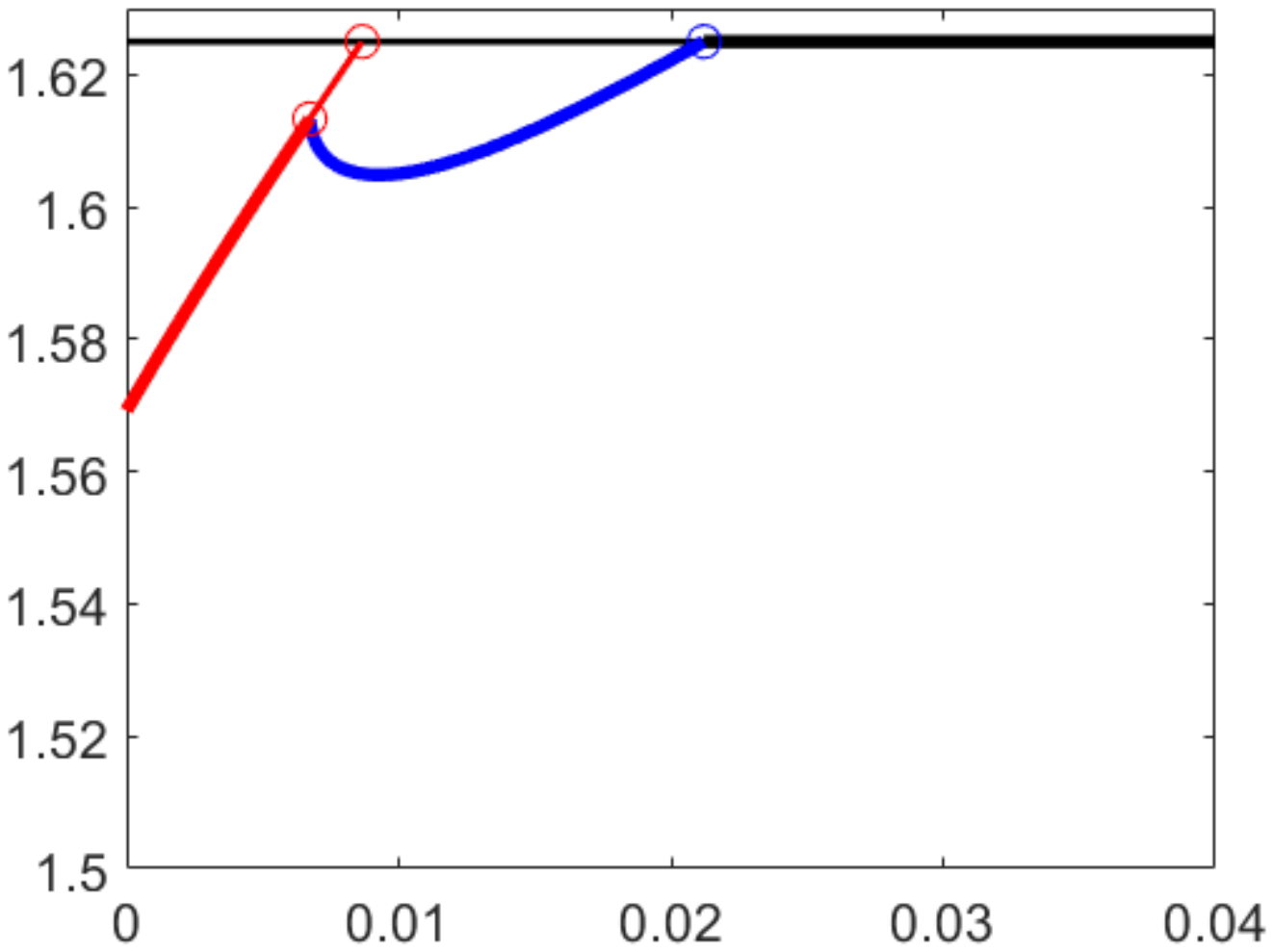}
\put(90,10){$d$}
\put(2,28){\rotatebox{90}{$||u||_{L_2}$}}
\end{overpic} 
}
\subfloat[\label{BD_self_3}$d_{11}=0,\; d_{22}=0.1$]{
\begin{overpic}[width=0.5\textwidth,tics=10,trim={3cm  8.5cm 3cm 9cm},clip]{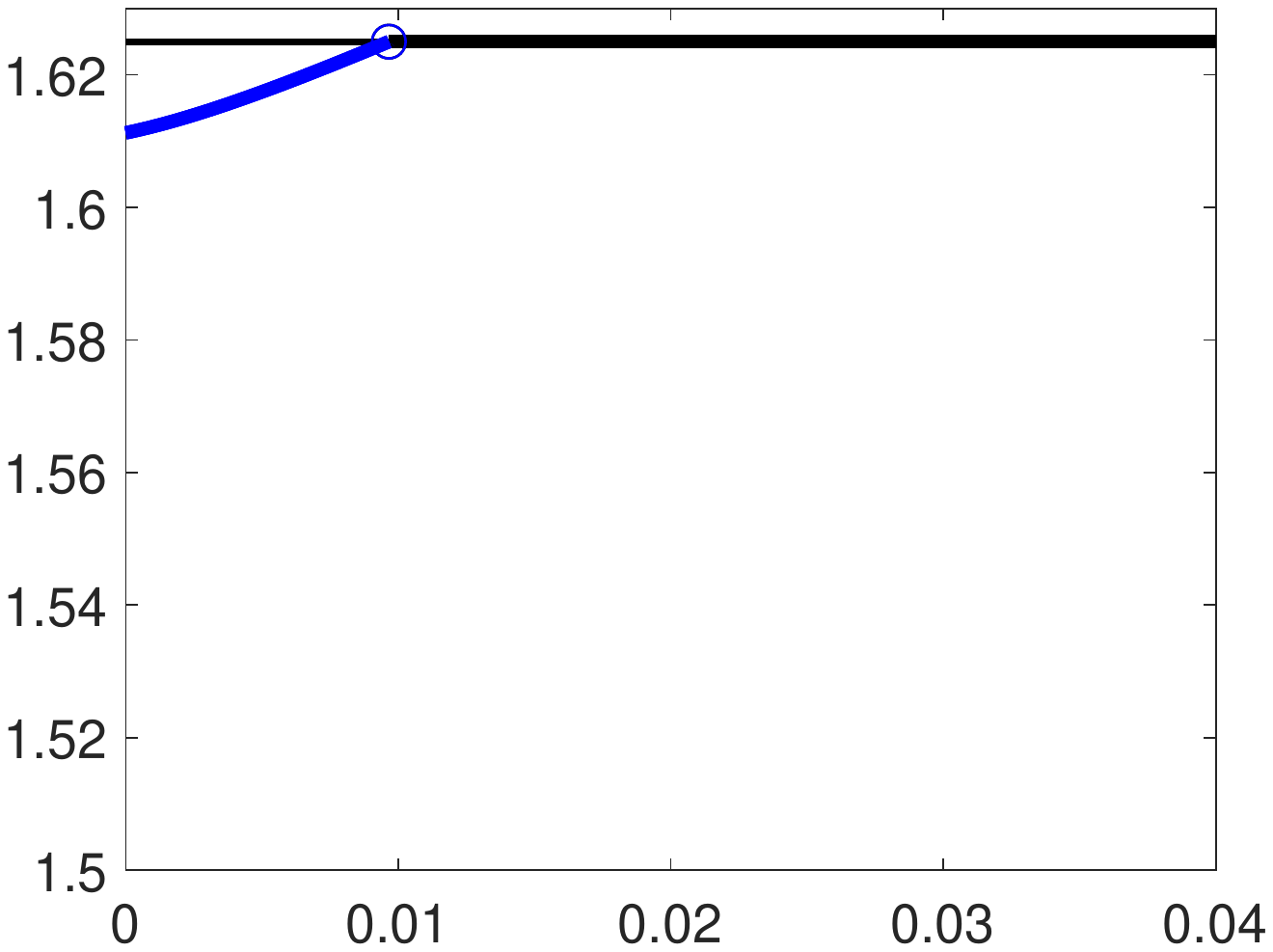}
\put(90,10){$d$}
\put(2,28){\rotatebox{90}{$||u||_{L_2}$}}
\end{overpic} 
}
\caption{Bifurcation diagrams with bifurcation parameter $d$ obtained with the first parameter set in Table \ref{tab:param} (weak competition case), with $d_{12}=3$, $d_{21}=0$, $d_{11}=0$ and different values of the self-diffusion coefficient $d_{22}$.}
\label{BD_self}
\end{figure}

%%%%%%%%%%%%%%%%%%%%%%%%%%%%%%%%%%%%%%%%%%%%%%%%%%%
\section{Concluding remarks}\label{sec:concl}

In this work we have extensively analyzed the cross-diffusion SKT model by detailed local linearized analysis and global numerical continuation. This approach has revealed some interesting effects of the cross-diffusion terms on the local and global steady states, and the results can also be interpreted in the original context of species competition. 

First of all, we proved that the intraspecific competition pressure (self-diffusion) contrasts the spatial segregation.
The effect of interspecific competition pressure (cross-diffusion) is more intricate instead. If the competition for the resources is weak, meaning that the two populations could coexist, the habitat segregation of too similar species is not feasible even in the presence of interspecific pressure. On the contrary, it is more likely that distinguishable populations exhibit habitat segregation. In this case the relative competition pressures play an important role on the outcomes. They non-simultaneously help to destabilize the homogeneous equilibrium. When the competition is strong and the exclusion principle is the predicted scenario, the competition pressure helps the species to coexist, segregated in different habitat regions through bifurcations of the homogeneous steady state.  
Moreover, there is a relation between the decreasing of the standard diffusion coefficient $d$ and the type of pattern: when the mobility of the individuals is reduced, so there is less mixing in the system, we observe the formation of multiple clusters.

Furthermore, the model often shows multistability of solutions, and in particular stable inhomogenous solution can coexist with the homogeneous one. The outcome of the system depends on the initial conditions and it is possible to pass from one to the other perturbing the system.

\medskip
Due to its innovative interplay between linearized analysis and numerical continuation, this work highlights new interesting aspects and opens several different questions that can be addressed in future works. From the theoretical point of view, the analytical characterization of the first bifurcation point (sub or supercritical) is an interesting open question. From the numerical study we found the appearance of Hopf bifurcation points in both weak and strong competition regimes, suggesting the formation of time-periodic spacial patterns. The effective presence of these time-varying patterns, their type and stability properties are biologically relevant, as well as the influence of the cross-diffusion terms on the Hopf point. The linearized analysis and the numerical continuation also reveal the possible presence of higher co-dimension points. Even though their detection with the continuation software \texttt{pde2path} is not possible yet, it would be interesting to better characterize their presence and their impact on the global bifurcation diagram and on the stability of steady states. In addition, we focused on bifurcation curves connected to the homogeneous branch, while it could also be interesting to study isolated bifurcation curves (``isolas''), which may be found using \texttt{pde2path} in combination with multiple solution methods~\cite{KuehnEllipticCont}.

Finally, the same study could be carried out for other quasilinear problems involving cross-diffusion terms. For instance, in the context of predator--prey systems, it is possible to derive by time-scale arguments another type of cross-diffusion terms \cite{conforto2018,desvillettes2018}. The linearized analysis suggests that they do not increase the parameter region in which patterns appear, but as in the present work, the global influence cannot be captured only by the linearized analysis. Taken together, these results will better clarify the role of cross-diffusion terms as the key ingredients in pattern formation.

%%%%%%%%%%%%%%%%%%%%%%%%%%%%%%%%%%%%%%%%%%%%%%%
%%%%%%%%%%%%%%%%%%%%%%%%%%%%%%%%%%%%%%%%%%%%%%%
\section*{Acknowledgements}
MB and CK have been supported by a Lichtenberg Professorship of the VolkswagenStiftung. CS has received funding from the European Union's Horizon 2020 research and innovation programme under the Marie Sk\l odowska--Curie grant agreement No. 754462. Support by INdAM-GNFM is gratefully acknowledged by CS. 
%%%%%%%%%%%%%%%%%%%%%%%%%%%%%%%%%%%%%%%%%%%%%%%
%%%%%%%%%%%%%%%%%%%%%%%%%%%%%%%%%%%%%%%%%%%%%%%

\bibliographystyle{plain}
\bibliography{bibliography}

\end{document}